\documentclass[11pt]{article}

\usepackage[numbers,sort&compress]{natbib}
 \usepackage[colorlinks=true, linkcolor=blue, citecolor=blue]{hyperref}

\usepackage[utf8]{inputenc} \usepackage[T1]{fontenc}    \usepackage{hyperref}       \usepackage{url}            \usepackage{booktabs}       \usepackage{amsfonts}       \usepackage{nicefrac}       \usepackage{microtype}      \usepackage{xcolor}

\usepackage{amsmath}
\usepackage{amsfonts}
\usepackage{amssymb, amsthm, epsfig}
\usepackage{mathrsfs}
\usepackage{euscript}

\usepackage{graphicx}
\usepackage{verbatim}
\usepackage{color}
\usepackage{epstopdf}

\usepackage{mathtools}
\usepackage{float}
\usepackage{algorithm}
\usepackage[noend]{algorithmic}

\usepackage{xspace}
\usepackage{color}

\usepackage{wrapfig}
\usepackage{bbm}

\usepackage{booktabs}
\usepackage{subcaption}

\usepackage{thmtools} 
\usepackage{thm-restate}

\newcommand{\eps}{\varepsilon}

\newcommand{\R}{\ensuremath{\mathbb{R}}}

\newcommand{\dir}{\ensuremath{\mathsf{d}}}

\newcommand{\norm}[1]{\lVert #1 \rVert}

\newcommand\numberthis{\addtocounter{equation}{1}\tag{\theequation}}
\newcommand{\abs}[1]{| #1 |}
\newcommand{\setdef}[2]{\left\{ #1 \mid #2 \right\}}
\newcommand{\inner}[2]{\langle #1,#2 \rangle}

\newcommand{\E}{\ensuremath{\textsf{E}}}

\newtheorem{theorem}{Theorem}
\newtheorem{lemma}[theorem]{Lemma}

\newtheorem{corollary}[theorem]{Corollary}

\newtheorem{remark}[theorem]{Remark}

\usepackage[inline]{enumitem}

\usepackage{bm}

\DeclareMathOperator{\CV}{CV}

\renewcommand{\ne}{\mathsf{ne}}
\newcommand{\wtheta}{\widehat{\theta}}

\newcommand{\bX}{\mathbf{X}}

\newcommand{\wne}{\widehat{\ne}}
\newcommand{\sign}{\mathsf{sign}}
\newcommand{\wGamma}{\widehat{\Gamma}}
\newcommand{\Dtheta}{\theta^{\Delta}}
\newcommand{\ww}{\widehat{w}}
\newcommand{\svmin}{\sigma_{\min}}
\newcommand{\svmax}{\sigma_{\max}}
\newcommand{\erf}{\mathsf{erf}}
\newcommand{\lambdao}{\lambda_*}
\newcommand{\lambdacv}{\lambda_{\CV}}
\newcommand{\lambdaio}[1]{\lambda_{#1,*}}
\newcommand{\wv}{\widehat{v}}
\newcommand{\wu}{\widehat{u}}
\newcommand{\bZ}{\mathbf{Z}}
\newcommand{\poly}{\mathsf{poly}}
\newcommand{\bzero}{0}

\usepackage[a4paper]{geometry}
\title{Inconsistency of cross-validation for structure learning in \\Gaussian graphical models}
\author{Zhao Lyu \and Wai Ming Tai \and Mladen Kolar \and Bryon Aragam}

\date{\emph{University of Chicago}}

\begin{document}

\maketitle

\begin{abstract}
Despite numerous years of research into the merits and trade-offs of various model selection criteria, obtaining robust results that elucidate the behavior of cross-validation remains a challenging endeavor. In this paper, we highlight the inherent limitations of cross-validation when employed to discern the structure of a Gaussian graphical model. We provide finite-sample bounds on the probability that the Lasso estimator for the neighborhood of a node within a Gaussian graphical model, optimized using a prediction oracle, misidentifies the neighborhood. Our results pertain to both undirected and directed acyclic graphs,  encompassing general, sparse covariance structures. To support our theoretical findings, we conduct an empirical investigation of this inconsistency by contrasting our outcomes with other commonly used information criteria through an extensive simulation study. Given that many algorithms designed to learn the structure of graphical models require hyperparameter selection, the precise calibration of this hyperparameter is paramount for accurately estimating the inherent structure. Consequently, our observations shed light on this widely recognized practical challenge.
\end{abstract}

\section{Introduction}
\label{sec:intro}

Parameter tuning, also known as hyperparameter selection or model selection, is an unavoidable aspect of modern machine learning. For predictive tasks such as classification with deep neural networks, cross-validation is the gold standard for evaluating the performance of a model and tuning hyperparameters, and comes with its own set of practical challenges \citep{lei2020cross,wilson2020approximate,bates2023cross}. However, other applications, such as structure learning in graphical models, are not purely predictive in nature, and alternative criteria such as the Bayesian information criterion \citep[BIC,][]{schwarz1978estimating} and the Akaike information criterion \citep[AIC,][]{akaike1974new} are often used instead. Structure learning goes one step beyond prediction.
Due to practical applications in causal inference, fairness, interpretability, and domain generalization, structure learning of undirected \citep{cai2011constrained,friedman2008sparse,meinshausen2006,yuan2007model} and directed graphs \citep{schmidt2007,xiang2013,fu2013,aragam2015ccdr} has received renewed attention. Existing work suggests that predictive criteria, such as cross-validation, are not suitable for learning structure (learning the edge structure or the pattern of non-zero elements of the precision matrix) of a graphical model \citep{meinshausen2006,friedman1996}.
In particular, \citet{meinshausen2006} (Proposition~1) proved that tuning the Lasso hyperparameter via a prediction oracle \emph{provably} returns the wrong structure in the infinite-sample limit. This seminal result for Lasso-based graphical model estimators provides formal justification to the well-known folklore that choosing a model using predictive criteria may lead to undesirable overfitting.
In practice, this leads to the question of which criterion to use for structure learning. A complete picture has yet to emerge despite decades of research studying the tradeoffs of these criteria for different tasks.

In this paper, we study the properties of cross-validation (CV) for structure learning of a graphical model using the Lasso to estimate the neighborhood of each node, extending the results of \citet{meinshausen2006} to more general settings. We show that in a precise sense cross-validation is inherently fallible and provide finite-sample bounds on the probability of structure inconsistency. As a motivating example, we will consider the special case of undirected Gaussian graphical models; however, our main result applies to neighborhood selection in general linear Gaussian models, and hence can be applied to directed acyclic graphs as well (Section~\ref{sec:gm}). Thus, the main message of this paper can be summarized as follows:
\begin{quote}
    \emph{For many structure learning applications, cross-validation is provably inconsistent and alternative criteria should be used instead.}
\end{quote}
While our results are specific to the Lasso, we note that this phenomenon does not appear to be specific to the Lasso: See Section~\ref{sec:related} for a discussion of similar results for subset selection and bridge estimators.

Which alternative criteria should be used? This is an intriguing question, with numerous consistency results in the literature \citep{haughton1988,foygel2010extended,chen2008extended,kim2012consistent}. To probe this question empirically, we conduct an extensive empirical study to compare the performance of different criteria. 
Our results indicate that extended BIC \citep{foygel2010extended} performs well, especially in high-dimensions (see Section~\ref{sec: Experiment}).

While certain information criteria (IC) may be computationally challenging to evaluate---often necessitating the optimization of a nonconvex likelihood---CV is frequently proposed as a default substitute. However, our findings caution against such an approach. Despite the potential complexities in computing IC, resorting to CV as an alternative  will lead to incorrect results, at least as far as structure learning is concerned.

\paragraph{Contributions}
We make the following contributions:
\begin{enumerate}
    \item We provide finite-sample bounds on the probability that a Lasso estimate tuned with a prediction oracle will provably recover the \emph{wrong} neighborhood in a linear Gaussian model (Theorem~\ref{thm:main}, Section~\ref{sec:main});
\item We prove that CV indeed approximates this prediction oracle, which implies inconsistency of CV (Theorem~\ref{thm:cv}, Section~\ref{sec:main});
    \item We apply these results to demonstrate inconsistency in structure learning for both undirected and directed acyclic graphical models (Corollaries~\ref{cor:ug} and ~\ref{cor:dag}, Section~\ref{sec:gm});
    \item We provide an extensive simulation study comparing the virtues and tradeoffs of different parameter tuning strategies and algorithms (Section~\ref{sec: Experiment}). These experiments confirm that the issues with CV are not restricted to the particular setting of our theoretical results and extend to other algorithms and non-Gaussian data. 
\end{enumerate}

\section{Gaussian graphical models and structure learning}
\label{sec:prelim}

We will use the classical undirected Gaussian graphical model as a motivating example, but note that our results apply more generally (see Section~\ref{sec:gm}). This preliminary section is intended to provide background and context for the structure learning problem; formal setup and details of our particular theoretical result can be found in Section~\ref{sec:main}.

Gaussian graphical models (GGMs) are widely used to represent and model statistical relations between variables. GGMs encode conditional independences with undirected graphs, which can also be read from the zero pattern in precision matrices~\citep{lauritzen1996graphical}. GGMs have a wide range of applications in natural language processing~\citep{manning1999foundations}, computer vision~\citep{cross1983markov}, and computational biology~\citep{menendez2010gene, varoquaux2010brain}.
We begin by recalling the definition of a GGM and then discuss the various learning tasks that one might consider in this model.

Let $X = \begin{bmatrix}
X_1,X_2,\dots,X_{p}
\end{bmatrix}^\top\sim\mathcal{N}(\bzero,\Sigma)$ be $p$-dimensional joint Gaussian random vector with $\Sigma\succ0$. The conditional independence relationships between each random variable can be represented by a graph $G=(V,E)$ on $p$ nodes $V=X$ with edge set $E$. In particular, $X_i$ and $X_j$ are conditionally independent given all remaining variables $X_{\backslash\{i,j\}}$ if and only if $\Sigma^{-1}_{ij} = 0 $, which corresponds to a missing edge between $i$ and $j$ \citep[see, e.g.,][for details]{lauritzen1996graphical}. Thus, estimating the zero pattern of $\Sigma^{-1}$ is equivalent to recovering $E$, a problem known as \emph{structure learning}. Through its connection with neighborhood selection (Section~\ref{sec:main}), structure learning generalizes the variable selection and support recovery problems in classical regression models. To avoid complicating the presentation, we will not distinguish between these problems in the discussion.

It is worth comparing the different learning tasks in a GGM. \emph{Prediction} refers to predicting the value of a particular node $X_{i}$ given the values of the remaining nodes and is equivalent to linear regression. \emph{Parameter estimation} refers to learning the precision matrix $\Sigma^{-1}$ in some norm such as $\ell_2$ or Frobenius. Both of these tasks are quite different from structure learning: One can predict $X_{i}$ and/or estimate $\Sigma^{-1}$ while at the same time getting the zero pattern of $\Sigma^{-1}$ completely wrong in any finite sample, and in general this is what happens in practice.

The distinction between prediction/estimation and structure learning is crucial when tuning hyperparameters, since the optimal choice of hyperparameter \emph{depends on} what the learning goal is. This is well-known in the literature: For predictive tasks, CV/AIC are efficient (i.e. for prediction), but for structure learning, BIC is consistent \citep[see, e.g.,][for detailed review of such results]{arlot2010survey}.\footnote{More generally, structure learning is known as \emph{model identification} or \emph{model selection consistency}, i.e., selecting the correct model as opposed to an approximately correct one that obtains fast (e.g., minimax) rates of convergence.} Notably, these results do not imply that CV is in fact \emph{inconsistent} for structure learning, which is a stronger negative result for CV.

\section{Related work}
\label{sec:related}
\vspace{-0.7em}

The related problems of hyperparameter tuning and model selection have been extensively studied in the machine learning and statistics literature, and we invite the reader to consult one of the many monographs on the subject for a detailed overview \citep[][]{grunwald2007minimum,claeskens2008model,arlot2010survey}. 

The study of model selection procedures such as CV, BIC, AIC, etc. dates back several decades ~\citep{ff265a3d-1ec8-3034-9346-ab129dbd9576,akaike1974new,stone1974cross,geisser1975predictive,wahba1975completely,stone1977asymptotic,schwarz1978estimating,efron1983estimating,picard1984cross,herzberg1986note}.
It is well-known that BIC is consistent for structure learning in finite-dimensional models \citep{schwarz1978estimating,haughton1988} as well as high-dimensional models \citep{chen2008extended,foygel2010extended,kim2012consistent}. 
These results are central in the classical theory of structure learning for DAGs \citep{meek1997thesis,chickering2002,chickering2003}. We also mention the recent proposal to tune parameters in DAG models by \cite{biza2020tuning}.
At the same time, BIC can be inconsistent under misspecification \cite{grunwald2006bayesian,grunwald2007minimum,grunwald2017inconsistency}.
On the other hand, AIC is known to select (possibly misspecified) models that are minimax optimal in estimation and/or prediction in a sense that can be made precise \citep[see, e.g.,][and the references therein]{barron1999,massart2007concentration}. 

More relevant to the present work, \citet{li1987asymptotic} and \citet{shao1993linear} studied the properties of CV and generalized CV (GCV). \citet{li1987asymptotic} proved the \emph{loss consistency} of CV, which is not the same as structure (or model selection) consistency, and more closely related to the minimax optimality results for AIC. To the best of our knowledge, the first proof of \emph{structure (in)consistency} of CV appeared in \citet{shao1993linear} for a fixed design regression model using subset selection. Specifically, he showed that while leave-one-out CV (LOO) is inconsistent in selecting the true model,
leave-$k$-out CV is consistent as long as $k/n\to 1$.
Further work along these lines includes \citet{zhang1993model,shao1997asymptotic,yang2007consistency}.
\citet{meinshausen2006}, Proposition~1, established the inconsistency of a prediction oracle for GGMs with a single edge; see also \citet{meinshausen2008}.
More recently, \cite{chetverikov2021cross} recently showed that CV-tuned Lasso is minimax optimal for prediction and estimation, while for structure learning, \cite{su2017false} showed that false discoveries are asymptotically unavoidable. Later, \citet{wang2020which} went beyond the Lasso and studied two-stage bridge estimators, and showed that the Lasso can be improved by two-stage approaches.  
Our results build upon these works in the Lasso setting and develops finite-sample bounds for arbitrary linear Gaussian models, with a focus on implications for graphical model structure learning.

Finally, while our paper is focused on provably \emph{negative} consequences of CV, there is a long line of work proposing alternative tuning parameter selection methods. 
For example, 
tuning-parameter free approaches to structure learning abound \citep{wang2020learning,lederer2015trex,yu2019estimating,belloni2011square,sun2012scaled,chichignoud2016practical,liu2017tiger}.
Where structure learning is not the goal, the virtue of CV for predictive tasks is still a subject of intense study~\citep[see][and the references therein]{wilson2020approximate,lei2020cross,bates2023cross}.

\section{Main results}
\label{sec:main}
\vspace{-0.8em}

In this section, we present our setup and main theoretical result on the inconsistency of CV-tuned Lasso.
\vspace{-1em}
\subsection{Neighborhood Selection}
\vspace{-0.2em}
We first describe the formal setup for our main result. The neighborhood selection problem can be defined for a general linear model and does not require specific reference to a graphical model. In order to apply our main result to different types of graphs, we state our main result first in general, and then in Section~\ref{sec:gm} apply the general result to specific graphical models.

Let $p$ be a positive integer and $[p]:=\{1,2,\dots,p\}$.
Let $\Sigma$ be a $p$-by-$p$ positive definite matrix
and define $\Gamma$ to be the $(p-1)$-by-$(p-1)$ submatrix such that $\Gamma_{i,j} = \Sigma_{i,j}$ for $i,j\in [p-1]$, $v$ to be the $(p-1)$-dimensional vector such that $v_{i} = \Sigma_{i,p}$ for $i\in [p-1]$ and $a=\Sigma_{p,p}$, i.e.
\begin{align*}
    \Sigma
    =
    \begin{bmatrix}
    \Gamma & v \\
    v^\top & a
    \end{bmatrix}.\numberthis \label{eq:sigma_def}
\end{align*}
Let $X = \begin{bmatrix}
X_1 , X_2 , \dots , X_{p}
\end{bmatrix}^\top\sim\mathcal{N}(\bzero,\Sigma)$.
The neighborhood selection problem seeks to learn the dependence of a target node in $X$ on the rest of the observed variables. 
Without loss of generality, let the target node be $X_p$.
Define 
\vspace{-0.5em}
\begin{align}
\label{eq:nbhd:param}
    \begin{aligned}
        \theta^*
    &:=
    \arg\min_{\theta\in \mathbb{R}^{p-1}} \E_{X\sim\mathcal{N}(\bzero,\Sigma)}(X_p - \sum_{j=1}^{p-1}\theta_j X_j)^2, 
    \\
    \ne^* 
    &:= 
    \setdef{i\in[p-1]}{\theta^*_i\neq 0}. 
    \end{aligned}
\end{align}
It is easy to show that $\theta^* = \Gamma^{-1}v$.
We assume that $\theta^*$ 
has $s$ non-zero entries or equivalently $\abs{\ne^*}=s$. The neighborhood selection problem attempts to recover $\ne^*$ from $n$ i.i.d.~observations of $X$, which we assemble into an $n\times p$ data matrix $\bX$.
For any matrix $\bX\in\mathbb{R}^{n \times p}$ and $\lambda>0$, the Lasso estimate with penalty parameter $\lambda$ can be written as 
\begin{align}
\label{eq:nhbdlasso}
    \wtheta^{\lambda,\bX}
    :=
    \arg\min_{\theta\in\mathbb{R}^{p-1}}\frac{1}{2n}\norm{\bX_p - \sum_{j=1}^{p-1}\theta_j \bX_j}_2^2 + \lambda\norm{\theta}_1,
\end{align}
where $\bX_i$ is the $i$th column of $\bX$.
The neighborhood estimated by the Lasso is defined by the non-zero entries of $\wtheta^{\lambda,\bX}$, i.e.,
\begin{align*}
    \wne^{\lambda,\bX}
    :=
    \setdef{i\in [p-1]}{\wtheta^{\lambda,\bX}_i \neq 0}.
\end{align*}
See \cite{meinshausen2006} for a more detailed review of the neighborhood selection problem.

Standard methods for selecting the penalty parameter $\lambda$ include CV, BIC, and AIC, as discussed above.
Here, we consider an oracle choice of penalty parameter, which reflects the limiting value of CV as $n\to\infty$ (see Theorem~\ref{thm:cv}).
For any matrix $\bX \in \mathbb{R}^{n\times p}$, the \emph{oracle penalty} $\lambdao^\bX$ is defined as
\begin{align}
\label{def:oracle}
    \lambdao^\bX
    :=
    \arg\min_{\lambda>0} \E_{Y\sim \mathcal{N}(\bzero,\Sigma)}(Y_p - \sum_{j=1}^{p-1}\wtheta^{\lambda,\bX}_jY_j)^2.
\end{align}
Here, $Y$ is a \emph{new} sample from $\mathcal{N}(\bzero,\Sigma)$, independent of the training data $\bX$.
This is known as the ``oracle'' penalty, as it involves the unknown data distribution. In practice, this value is approximated by CV (See Theorem~\ref{thm:cv} for a formal statement).
For a fixed set of $n$ samples $\bX$, we shorten our notation to $\wtheta^{\lambda,\bX}, \wne^{\lambda,\bX}, \lambdao^\bX$ to $\wtheta^\lambda, \wne^{\lambda},\lambdao$ respectively if there is no ambiguity.

\subsection{Inconsistency of cross-validation}
We would like to ask: When $p$ and $n$ are large, can we recover $\ne^*$ via the Lasso with the oracle penalty?
Our first main result provides a finite-sample bound on the probability of exact recovery:

\begin{theorem}
\label{thm:main}
Let $\Sigma$ be a $p$-by-$p$ positive definite matrix such that $\abs{\ne^*}=s$ for some positive integer $s$.
Given a sample matrix $\bX\in\mathbb{R}^{n \times p}$ where each row is an i.i.d.~sample drawn from $\mathcal{N}(\bzero,\Sigma)$, we have
\begin{multline*}
    \mathsf{Pr}_{\bX\sim \mathcal{N}(\bzero,\Sigma)^{n}}(\wne^{\lambdao} = \ne^*) \\
    <
    O\bigg(2^s\cdot \bigg( p^{-\Omega(\frac{1}{\kappa(\Sigma)})} + spe^{-\Omega(\frac{n}{s^2\kappa(\Sigma)^6})}\bigg) \bigg),
\end{multline*}
where $\kappa(\Sigma)$ is the condition number of $\Sigma$.
\end{theorem}

When the probability in Theorem~\ref{thm:main} is strictly less than one, the prediction-oracle estimate is inconsistent.
In particular, we have the following corollary, which answers the question in the negative in the sublinear sparsity regime: 
\begin{corollary}\label{col:asym_thm}
Assume the setting in Theorem \ref{thm:main}. Let $\delta\in[0,1)$ and $p>1$. 
There exists a universal constant $C>0$ such that if $s\leq C(\frac{1}{\kappa(\Sigma)}\log p-\log\frac{1}{\delta})$, then
\begin{align*}
    \mathsf{Pr}_{\bX\sim \mathcal{N}(\bzero,\Sigma)^{n}}(\wne^{\lambdao} = \ne^*)
    & < \delta
    \quad\text{as $n \to \infty$.}
\end{align*}
\end{corollary}
Corollary~\ref{col:asym_thm} states that for sufficiently sparse graphs with $s=O(\log p)$, choosing $\lambda$ via the prediction oracle is \emph{provably inconsistent for structure learning}.
With probability $1-\delta$, neighborhood selection will not recover the correct neighborhood in any non-trivial dimension $p$, even when $p$ is fixed as $n\to\infty$.
Corollary~\ref{col:asym_thm} is an immediate consequence of Theorem~\ref{thm:main}. Furthermore, this is just one possible example of inconsistency for certain choices of $(n,p,s)$; clearly other configurations may lead to inconsistency as well.

In practice, the oracle penalty $\lambdao$ is unknown and is usually estimated by CV.
For any positive integer $K$ that divides $n$ (for simplicity), let $I_1,I_2,\dots,I_K$ be a partition of $[n]$ such that the size of each $I_k$ is $n/K$ for $k\in[K]$.
For any matrix $\bX\in\mathbb{R}^{n \times p}$, the CV penalty $\lambdacv^{\bX}$ is defined as
\begin{align*}
    \lambdacv^{\bX}
    & :=
    \arg\min_{\lambda>0} \frac{1}{K}\sum_{k=1}^K \frac{1}{n/K}\norm{\bX^{I_k}_p - \sum_{j=1}^{p-1}\wtheta^{\lambda,\bX^{-I_k}}_j\bX^{I_k}_j}_2^2
\end{align*}
where $\bX^{I_k}$ (resp. $\bX^{-I_k}$) is the $(n/K)$-by-$p$ submatrix of $\bX$ whose row indices are in $I_k$ (resp. not in $I_k$) for $k\in [K]$.
Although it is common to use $\lambdacv^{\bX}$ to estimate $\lambdao^{\bX}$ in practice, we could not find a formal proof of this approximation in the literature.
Therefore, we also prove the following theorem for completeness.
\begin{theorem}
\label{thm:cv}
Let $\Sigma$ be a $p$-by-$p$ positive definite matrix.
Suppose we are given a sample matrix $\bX\in\mathbb{R}^{n \times p}$ where each row is an i.i.d sample drawn from $\mathcal{N}(\bzero,\Sigma)$.
Then, for every $\delta>0$,
\begin{align*}
    \mathsf{Pr}_{\bX\sim \mathcal{N}(\bzero,\Sigma)^{n}}(|\lambdacv^{\bX} 
    -
    \lambdao^{\bX}|<\delta)\to 1 \quad \text{as $n\to \infty$.}
\end{align*}
\end{theorem}

By combining Theorems~\ref{thm:main} and~\ref{thm:cv} with the known properties of the solution path for the Lasso \citep{efron2004least}, 
it is not hard to show that the CV-tuned neighborhoods are inconsistent, i.e., 
\begin{align*}
    \mathsf{Pr}_{\bX\sim \mathcal{N}(\bzero,\Sigma)^{n}}(\wne^{\lambdacv^\bX,\bX} = \ne^*)<\delta\quad\text{as $n \to \infty$.}
\end{align*}
The number of folds $K$ affects $\lambdacv^{\bX}$ and further affects $\wne^{\lambdacv^\bX,\bX}$. Briefly, if we only consider the dependence on $n$ and $K$, the probability is roughly $1-O(Ke^{-n^{< 1}/K})$. Therefore, if $K=o(n)$ we have the probability $\to 1$ as $n\to \infty$. Details are postponed to proofs of Theorem \ref{thm:cv}.

\section{Application to graphical models}
\label{sec:gm}

As stated above, our main result applies to neighborhood selection in a general linear Gaussian model with $X\sim\mathcal{N}(\bzero,\Sigma)$. In this section, we apply this result to two important special cases: Undirected Gaussian graphical models and Gaussian DAG models. 

\subsection{Undirected graphs}
\label{sec:ug}

A popular approach to learning Gaussian graphical models is to directly apply neighborhood selection node-by-node, and use the neighborhood of each node to define a $p\times p$ graph \citep{meinshausen2006}. 
Let $\omega_{j}\in\R^{p}$ be the coefficient vector for the $j$th nodewise neighborhood regression problem, where $\omega_{j}=\begin{bmatrix}\omega_{1j},\ldots,\omega_{pj}\end{bmatrix}^\top$ for each $j$.
Formally, $\omega_{j}$ solves \eqref{eq:nbhd:param} with $p$ replaced by $j$ (i.e. the target node is $j$), and we add a zero in the $j$th position.
This defines a matrix $\Omega=[\omega_{1}\,|\,\cdots\,|\,\omega_{p}]=\begin{bmatrix}\omega_{ij}\end{bmatrix}\in\mathbb{R}^{p\times p}$. The zero pattern of this matrix defines an undirected graph $G=(V,E)$, and is the same as the zero pattern of $\Sigma^{-1}$.
We estimate $\Omega$ by $\widehat{\Omega}(\lambda)=[\widehat{\omega}_{1}(\lambda)\,|\,\cdots\,|\,\widehat{\omega}_{p}(\lambda)]$, where $\widehat{\Omega}(\lambda)$ is the solution to the following optimization problem:
\begin{align}
\label{eq:pseudo}    \min_{\substack{\omega_{1},\ldots,\omega_{p}\\\omega_{j}\in\mathbb{R}^{p-1}}}\frac{1}{2n}\sum_{j=1}^{p}\Big\{\norm{\bX_j - \sum_{i\neq  j}\omega_{ij} \bX_i}_2^2 + \lambda\norm{\omega_{j}}_1\Big\}.
\end{align}
To estimate the structure $G$, we let $\widehat{G}(\lambda)$ be the undirected graph whose edges correspond to the nonzero entries in the solution $\widehat{\Omega}(\lambda)$ (see also Remark~\ref{rem:symmetric}).
It is easy to see that \eqref{eq:pseudo} is equivalent to solving $p$ nodewise regression problems \eqref{eq:nhbdlasso}. 
This is also known as the \emph{pseudo-likelihood} approach, since the objective is not a true (joint) likelihood. Nonetheless, it is well-known to provide a consistent estimate of the structure of $G$ for certain choices of $\lambda$. 
Finally, let $\widehat{G}_{\CV}=\widehat{G}(\lambda_{\CV})$ be the estimate when CV is used to tune $\lambda$.

The following corollary is immediate from Theorems~\ref{thm:main} and~\ref{thm:cv}:
\begin{corollary}
\label{cor:ug}
    Suppose $\bX\in\mathbb{R}^{n \times p}$ is a sample matrix where each row is an i.i.d.~sample drawn from $\mathcal{N}(\bzero,\Sigma)$ and let $G$ be the undirected Gaussian graphical model associated with $\Sigma^{-1}$. Then, for any $\delta>0$ satisfying the conditions in Corollary~\ref{col:asym_thm},
    \begin{multline*}
    \mathsf{Pr}_{\bX\sim \mathcal{N}(\bzero,\Sigma)^{n}}(\widehat{G}_{\CV}\neq G) <\delta
    \quad\text{as $n\to\infty$.}
    \end{multline*}
\end{corollary}
Thus, CV is inconsistent for learning the structured of an undirected Gaussian graphical model.

\begin{remark}
    \label{rem:symmetric}
    Since $\Omega$ and $\widehat{\Omega}$ essentially capture partial regression coefficients, these matrices are not symmetric in general. Nonetheless, the support of $\Omega$ is always symmetric (see e.g.~Sec~5.1.3 in \citealp{lauritzen1996graphical}), but $\widehat{\Omega}$ may not have a symmetric support on finite samples. Asymptotically, this does change anything, but on finite-samples we need to use either the AND or the OR rule to symmetrize $\widehat{\Omega}$, as discussed in \citet{meinshausen2006}.
\end{remark}

\subsection{Directed acyclic graphs}
\label{sec:dag}

Our results also apply to DAG models, which are popular for modeling causal relationships in ML.
First, recall the general linear structural equation model (SEM):
\begin{align}
\label{eq:sem}
\begin{aligned}
        X_j
    &= \sum_{i=1}^{p}\beta_{ij}X_{i} + \eps_{j},
    \quad
    \eps_{j}\sim\mathcal{N}(0,\sigma_{j}^{2}),
    \\
    E &= \{(i,j) : \beta_{ij}\neq 0\}.
\end{aligned}
\end{align}
We collect the SEM coefficients $\beta_{ij}$ into a $p\times p$ matrix $B=[\beta_{1}\,|\,\cdots\,|\,\beta_{p}]=(\beta_{ij})\in\mathbb{R}^{p\times p}$, with the same indexing conventions as $\Omega$ in Section~\ref{sec:ug}. 
This defines a graph $G = (V,E)$ that be read off from the nonzero entries in $B$.
When $G$ is a DAG, \eqref{eq:sem} defines a Gaussian DAG model. We assume throughout that $G$ is acyclic.
 
A common procedure to learn a DAG is to first learn a topological ordering of $G$, and then regress each node onto its predecessors in this ordering \citep[e.g.][]{shojaie2010,ghoshal2017ident,ghoshal2017sem,chen2018causal,park2020identifiability}. 
More precisely, given an ordering $\prec$ on the variables $X_{i}$, we define an SEM \eqref{eq:sem} by regressing each $X_{j}$ onto the set $A_{j}=\{X_{i} : X_{i}\prec X_{j}\}$. The set of nonzero coefficients $\{i : \beta_{ij}\neq0\}$ defines the parents of $X_{j}$ in the ordering $\prec$.

Following this literature, let $\widehat{G}(\prec,\lambda)$ denote the estimate of $G$ that results from using the order $\prec$ and $\ell_{1}$-regularized least squares with $\lambda>0$ to estimate each parent set from the candidate set $A_{j}$:
\begin{align}
\min_{\substack{\beta_{1},\ldots,\beta_{p}\\\beta_{j}\in\mathbb{R}^{|\!A_{j}\!|}}}\frac{1}{2n}\sum_{j=1}^{p}\Big\{\norm{\bX_j - \sum_{i\in A_{j}}\beta_{ij} \bX_i}_2^2 + \lambda\norm{\beta_{j}}_1\Big\}.
\end{align}
For each $j$, we are solving a neighborhood regression problem similar to \eqref{eq:nhbdlasso}, except instead of regressing the $j$th node onto every other variable, we restrict attention to the candidate variables $A_{j}$ induced by the ordering $\prec$.
Finally, let $\widehat{G}_{\CV}(\prec)=\widehat{G}(\prec,\lambda_{\CV})$ be the resulting graph when CV is used to tune $\lambda$.

The following corollary is also immediate from Theorems~\ref{thm:main} and~\ref{thm:cv}:
\begin{corollary}
\label{cor:dag}
    Suppose we are given $n$ i.i.d. samples from the model \eqref{eq:sem} with DAG $G$, and suppose further that we know the true ordering $\prec$ of $G$. Then, for any $\delta>0$ satisfying the conditions in Corollary~\ref{col:asym_thm},
    \begin{align*}
    \mathsf{Pr}_{\bX\sim \mathcal{N}(\bzero,\Sigma)^{n}}(\widehat{G}_{\CV}(\prec)\neq G) <\delta
    \quad\text{as $n\to\infty$.}
    \end{align*}
\end{corollary}

Thus, even if we know the true ordering, CV will return the wrong DAG. If we do not know the true ordering, this result says that $\widehat{G}_{\CV}(\prec)$ is an inconsistent estimate of the minimal I-map corresponding to $\prec$ (see \citealp{lauritzen1996graphical} for definitions). Of course, assuming everything else is equal, structure learning with unknown ordering is at least as difficult as with a known ordering.

\section{Proof overview}
\label{sec:proof}

In this section, we outline the main idea of the proof of Theorem~\ref{thm:main}. Detailed proofs of both Theorem~\ref{thm:main} and~\ref{thm:cv} are deferred to the supplementary materials. 

We start with some observations. Without loss of generality, we assume that $\ne^* =[s]$.
For any matrix $\bX  \in \mathbb{R}^{n\times p}$, $i\in[p-1]$, and $\theta\in\mathbb{R}^{p-1}$, define 
\begin{align}
    G_{\bX,i}(\theta)
    :=\frac{1}{n}\inner{\bX_p - \sum_{j=1}^{p-1}\theta_j \bX_j}{\bX_i}.
\end{align}
We can write the classical KKT conditions for the Lasso as follows:
\begin{align}
\label{eq:kkt:main}
    \begin{cases}
        G_{\bX,i}(\theta)=\sign(\theta_i)\lambda, & \text{for $\theta_i\neq 0$} \\
        \abs{G_{\bX,i}(\theta)}\leq \lambda, & \text{for $\theta_i=0$}.
    \end{cases}
\end{align}
Then $\theta$ satisfies \eqref{eq:kkt:main} if and only if $\theta=\wtheta^{\lambda}$ is a Lasso solution for $\lambda$.
Moreover, if we define an ellipsoid $\mathcal{E}$ by 
\begin{align*}
    \mathcal{E}
    & :=
    \{(\theta - \theta^*)^\top\Gamma(\theta - \theta^*) \leq (\wtheta^{\lambdao} - \theta^*)^\top\Gamma(\wtheta^{\lambdao} - \theta^*)\},
\end{align*}
we can show (Appendix~\ref{sec:o_ellipsoid}) that any point $\theta\in\mathcal{E}$
cannot be a Lasso solution for any penalty $\lambda>0$.
It is easy to see that the Lasso solution for the oracle penalty $\lambdao$, $\wtheta^{\lambdao}$, lies on the boundary of this ellipsoid.

To prove Theorem \ref{thm:main}, we want to argue that the event $\wne^{\lambdao} = \ne^*=[s]$ is unlikely.
Let 
\begin{align}
\wne^G
=\setdef{i\in[p-1]}{\abs{G_{\bX,i}(\wtheta^{\lambdao})} = \lambdao}.    
\end{align}
We have $\wne^{\lambdao} \subseteq \wne^G$ by the KKT conditions \eqref{eq:kkt:main}.
We further argue in Appendix~\ref{sec:wne_size} that $\wne^G$ has at most one extra element almost surely,
and without loss of generality, we may assume that $\wne^G=[s]$ or $\wne^G=[s+1]$.
We will consider these two cases separately.
Namely, we will bound the following probability:
\begin{align*}
    \MoveEqLeft \mathsf{Pr}_{\bX\sim \mathcal{N}(\bzero,\Sigma)^{n}}(\wne^{\lambdao} = [s])\\
    & =
    \mathsf{Pr}_{\bX\sim \mathcal{N}(\bzero,\Sigma)^{n}}(\wne^{\lambdao} = \wne^G = [s]) \\
    & \qquad + 
    \mathsf{Pr}_{\bX\sim \mathcal{N}(\bzero,\Sigma)^{n}}(\wne^{\lambdao} = [s]\wedge \wne^G = [s+1])
\end{align*}

\paragraph{Case I: $\wne^G=[s]$.}

The first step is to show that there exists a line passing through $\wtheta^{\lambdao}$ such that any point in the intersection of a small neighborhood of $\wtheta^{\lambdao}$ and this line is also a Lasso solution for some penalty $\lambda$. See Figure~\ref{fig:line} for an illustration of the following argument.

We will do this by defining a line $L$ and show that it satisfies the KKT conditions \eqref{eq:kkt:main}.
Consider the following system of equations:
\begin{align*}
\begin{cases}
    G_{\bX,i}(\theta) = \sign(\wtheta^{\lambdao}_i)\lambda &\text{for $i\in[s]$} \\
    \theta_i = 0 &\text{for $i\notin[s]$}
\end{cases} \numberthis\label{eq:line:system}
\end{align*}
Geometrically, we can view these $p-1$ equations, which are linear in $\theta$ and $\lambda$, as hyperplanes in the $\theta$-$\lambda$ space which is a $p$-dimensional space.
It turns out that the intersection of these $p-1$ hyperplanes forms a line almost surely, which is the desired line $L$.
Since $\wtheta^{\lambdao}$ clearly is a solution of \eqref{eq:line:system} by the definition of $\wtheta^{\lambdao}$, we can write $L$ as 
\begin{align}
\label{eq:def:line}
    L
    & :=
    \setdef{\wtheta^{\lambdao} + \delta \theta'}{ \delta \in\mathbb{R}}
\end{align}
for some $\theta'\in\R^{p}$. 
We will define $\theta'$ formally in \eqref{eq:theta_prime} in the Supplementary Material.

Consider any point $\widetilde{\theta} = \wtheta^{\lambdao} + \delta \theta' \in L$ for sufficiently small $\abs{\delta}$.
We will check $\widetilde{\theta}$ satisfies the KKT conditions \eqref{eq:kkt:main}.
If $\abs{\delta}$ is sufficiently small, $\widetilde{\theta}_i\neq 0$ for $i\in[s]$ since $\wtheta^{\lambdao}_i\neq 0$ for $i\in[s]$.
By the construction of the system \eqref{eq:line:system}, it ensures that all $G_{\bX,i}(\widetilde{\theta})$ remain equal in magnitude for $i\in[s]$ and the signs are consistent, i.e. $\sign(\widetilde{\theta}_i) = \sign(G_{\bX,i}(\widetilde{\theta}))$ for $i\in [s]$.
It means that $\widetilde{\theta}$ satisfies the first condition in \eqref{eq:kkt:main}.
On the other hand, we set $\theta'_i=0$ for $i\notin [s]$ to ensure $\widetilde{\theta}_i\neq 0$ for $i\notin [s]$.
Recall the definition of $\wne^G$, we have $\abs{G_{\bX,i}(\wtheta^{\lambdao})}$ \emph{strictly} less than $\lambdao$ for $i\notin[s]$.
If $\abs{\delta}$ is sufficiently small, it ensures that $\abs{G_{\bX,j}(\widetilde{\theta})}\leq\abs{G_{\bX,i}(\widetilde{\theta})}$ for $i\in[s]$ and $j\notin [s]$.
It means that $\widetilde{\theta}$ satisfies the second condition in \eqref{eq:kkt:main}.
Hence, for a sufficiently small $\abs{\delta}$, $\widetilde{\theta}$ is a Lasso solution.

Now, if the line $L$ is not a tangent line of $\mathcal{E}$ at $\wtheta^{\lambdao}$, then some point in $L$ must be inside the ellipsoid $\mathcal{E}$ .
This contradicts the observation that no Lasso solution can be inside $\mathcal{E}$.
Therefore, $L$ must be a tangent line.
From here, we can explicitly bound the probability of $L$ being a tangent line and hence also the probability $\mathsf{Pr}_{\bX\sim \mathcal{N}(\bzero,\Sigma)^{n}}(\wne^{\lambdao}=\wne^G=[s])$.

\begin{figure}[t]
\centering
  \includegraphics[width=0.5\textwidth]{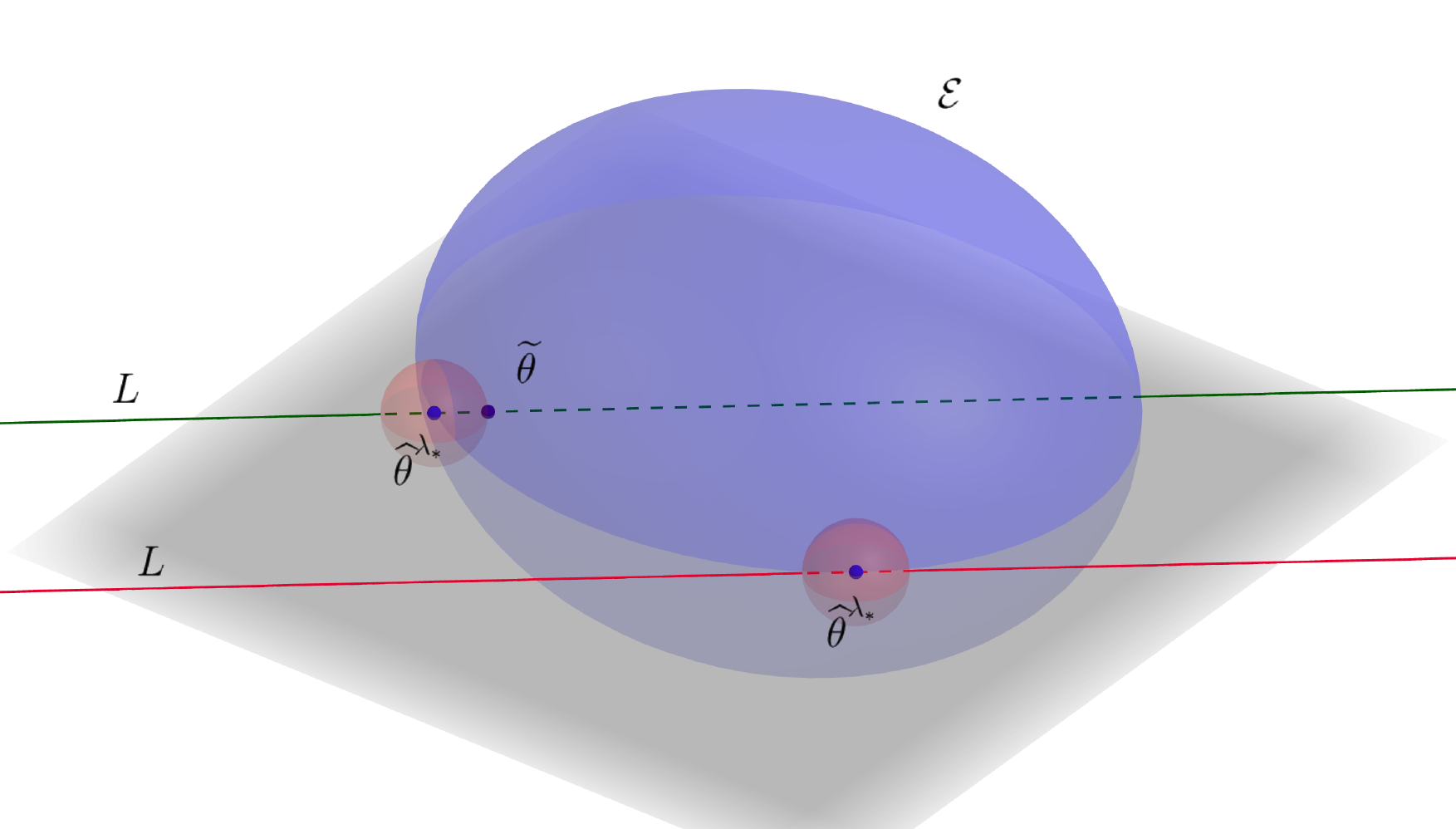}
\caption{
  Illustration of the ellipsoid $\mathcal{E}$ and the line $L$ when $p-1=3$ and $\wne^{\lambdao} = \wne^G=\{1,2\}$.  
  (green) There exists another Lasso solution $\widetilde{\theta}$ inside $\mathcal{E}$ when it is not a tangent line. 
  (red) No Lasso solution can be found inside $\mathcal{E}$ when it is a tangent line.
  }
  \label{fig:line}
\end{figure}

\paragraph{Case II: $\wne^{G}=[s+1]$.}
\begin{figure}[t]
\centering
  \includegraphics[width=0.5\textwidth]{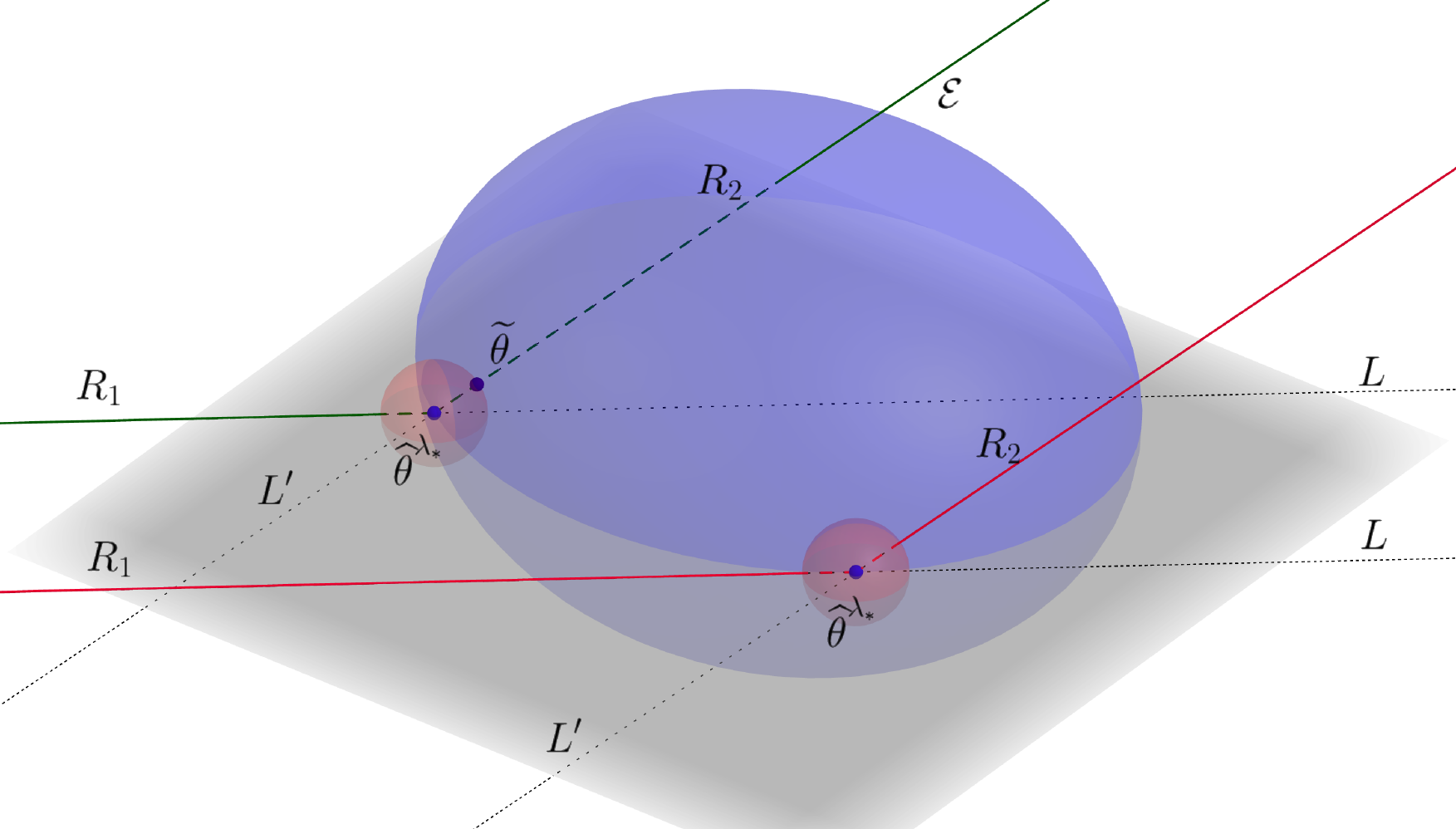}
\caption{
  Illustration of the ellipsoid $\mathcal{E}$ and the rays of intersection when $p-1=3$, $\wne^{\lambdao}=\{1,2\}$ and  $\wne^G=\{1,2,3\}$. 
  (green) There exists another Lasso solution $\widetilde{\theta}$ inside $\mathcal{E}$ when one of the rays shoots into $\mathcal{E}$. 
  (red) No Lasso solution can be found inside $\mathcal{E}$ when both rays shoot out of $\mathcal{E}$.
  }
  \label{fig:rays}
\end{figure}
In this case, instead of a line, there are two rays shooting from $\wtheta^{\lambdao}$ such that any point in the intersection of a small neighborhood of $\wtheta^{\lambdao}$ and these two rays is also a Lasso solution for some penalty $\lambda$. See Figure~\ref{fig:rays}.

In this case, we have $\abs{G_{\bX,s+1}(\wtheta^{\lambdao})} = \lambdao$.
If we follow the same argument as in the previous case, we can no longer establish $\abs{G_{\bX,s+1}(\widetilde{\theta})}\leq\abs{G_{\bX,i}(\widetilde{\theta})}$ for $i \in [s]$ \emph{no matter how small $\abs{\delta}$ is}.
Namely, the second condition in \eqref{eq:kkt:main} does not hold.
Fortunately, it turns out to only cause problems for one of $\delta\geq 0$ or $\delta\leq 0$.
We can indeed view the line $L$ in \eqref{eq:def:line} as two rays shooting from $\wtheta^{\lambdao}$ in the opposite directions which correspond to the cases of $\delta\geq 0$ or $\delta\leq 0$.
That means the argument in the case of $\wne^G=[s]$ still holds for one of these two rays which we denote by $R_1$.
Hence, one of the desired two rays is $R_1$.

We now only have one side of $L$ which is $R_1$.
The argument of constructing a Lasso solution inside $\mathcal{E}$ when $L$ is not a tangent line does not hold because $R_1$ probably does not intersect $\mathcal{E}$ (except $\wtheta^{\lambdao}$) even when $R_1$ is not a tangent ray.
It is intuitive that there is another ray $R_2$ instead of the other side of $L$ that any point in the intersection of a small neighborhood of $\wtheta^{\lambdao}$ and $R_2$ is a Lasso solution.

To define another ray $R_2$, consider the following system of equations:
\begin{align*}
\begin{cases}
    G_{\bX,i}(\theta) = \sign(G_{\bX,i}(\wtheta^{\lambdao}))\lambda, &\text{for $i\in[s+1]$} \\
    \theta_i = 0, &\text{for $i\notin[s+1]$}.
\end{cases} \numberthis\label{eq:ray:system}
\end{align*}
By a similar argument as in the previous case, the intersection of these hyperplanes forms a line passing through $\wtheta^{\lambdao}$ almost surely which we can write as 
\begin{align}
\label{eq:def:line_2}
    L'
    & :=
    \setdef{\wtheta^{\lambdao} + \delta \theta''}{ \delta \in\mathbb{R}}
\end{align}
for some $\theta''\in\R^{p}$. 
We will define $\theta''$ formally in \eqref{eq:theta_dprime} in the Supplementary Material.

Consider any point $\widetilde{\theta} =\wtheta^{\lambdao} + \delta \theta''\in L'$ for a sufficiently small $\abs{\delta}$.
For $i\neq s+1$, $\widetilde{\theta}_i$  satisfies the first or second condition of the KKT conditions \eqref{eq:kkt:main} accordingly by a similar analysis as in the previous case.
For $i=s+1$, $\widetilde{\theta}_{s+1} = \wtheta^{\lambdao}_{s+1} + \delta\theta''_{s+1} = \delta \theta''_{s+1}$ is no longer $0$ and we need to check if it satisfies the first condition in \eqref{eq:kkt:main}.
By the construction of the system \eqref{eq:ray:system}, it ensures that all $G_{\bX,i}(\widetilde{\theta})$ remain equal in magnitude for $i\in [s+1]$.
We also need to check the sign consistency, i.e. $\sign(\widetilde{\theta}_{s+1}) = \sign(G_{\bX,s+1}(\widetilde{\theta}))$.
We again view $L'$ in \eqref{eq:def:line_2} as two rays shooting from $\wtheta^{\lambdao}$ in the opposite directions.
It turns out that only one of them ensures this sign consistency and we choose this ray as the desired second ray $R_2$.

Now, 
if one of the rays $R_1,R_2$ shoots into (i.e. intersects) $\mathcal{E}$, this would contradict the observation that no Lasso solution can be inside $\mathcal{E}$.
Therefore, both rays $R_1,R_2$ shoot out of $\mathcal{E}$.
As before, we can bound the probability that both $R_1,R_2$ shoot out of $\mathcal{E}$ and hence the probability $\mathsf{Pr}_{\bX\sim \mathcal{N}(\bzero,\Sigma)^{n}}(\wne^{\lambdao}=[s]\wedge \wne^{G}=[s+1])$.

Combining these two cases, we obtain an explicit bound on the  probability $\mathsf{Pr}_{\bX\sim \mathcal{N}(\bzero,\Sigma)^{n}}(\wne^{\lambdao}=[s])$.

\begin{figure*}[t]
  \centering
  \includegraphics[width=0.7\textwidth]{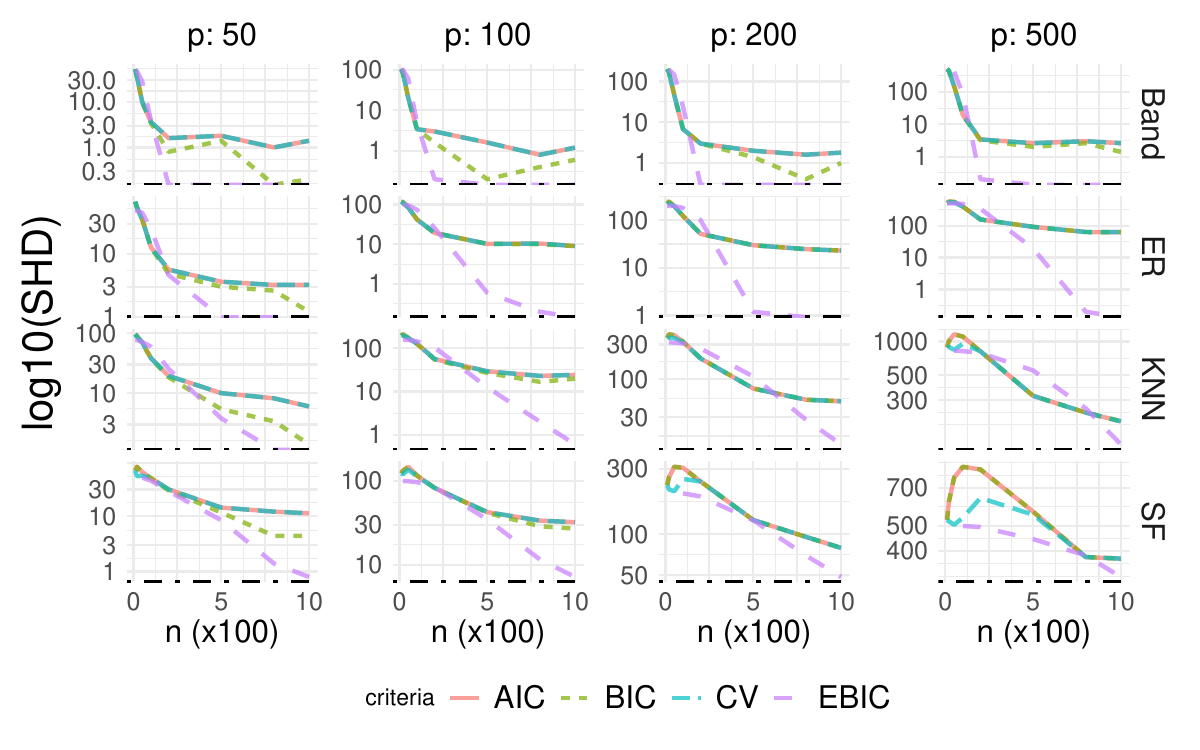}
  \caption{Log-SHD vs. $n$ (in hundreds) and $p$ on different graph types using NS to compare criteria. The black dot-dash line represents zero SHD, i.e. perfect neighborhood selection.}
  \label{fig: fig1 NS SHD}
\end{figure*}

\section{Experiments}\label{sec: Experiment}
In this section, we demonstrate through simulations the failure of CV for structure learning, verifying our main theoretical results. 
We include here only a snapshot of our results to convey the main point; complete details and results from our exhaustive experiments can be found in Appendix~\ref{Appendix Experiement}, including additional experiments on non-Gaussian data.
\begin{figure*}[!h]
  \centering
  \includegraphics[width=0.7\textwidth]{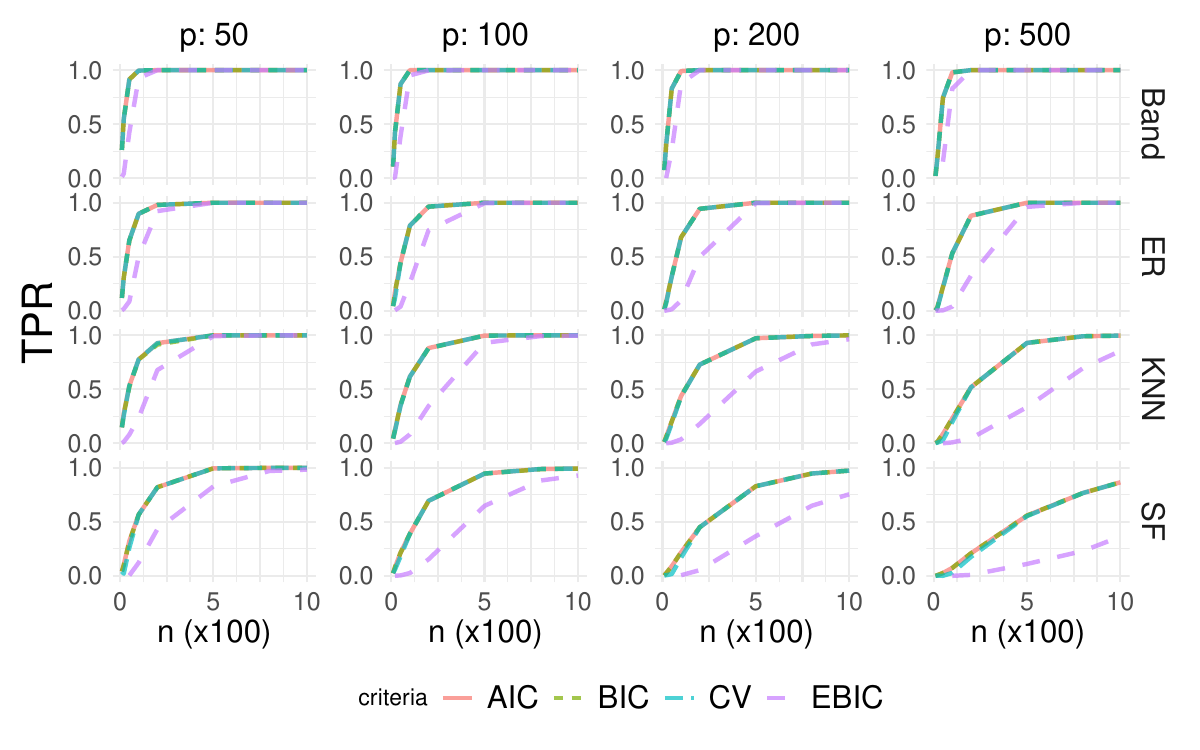}
  \caption{TPR vs. sample size $n$ (in hundreds) and $p$ on different graph types with NS to compare criteria.}
  \label{fig: fig2.1 NS TPR and FDR}
\end{figure*}

\begin{figure*}[!h]
  \centering
  \includegraphics[width=0.7\textwidth]{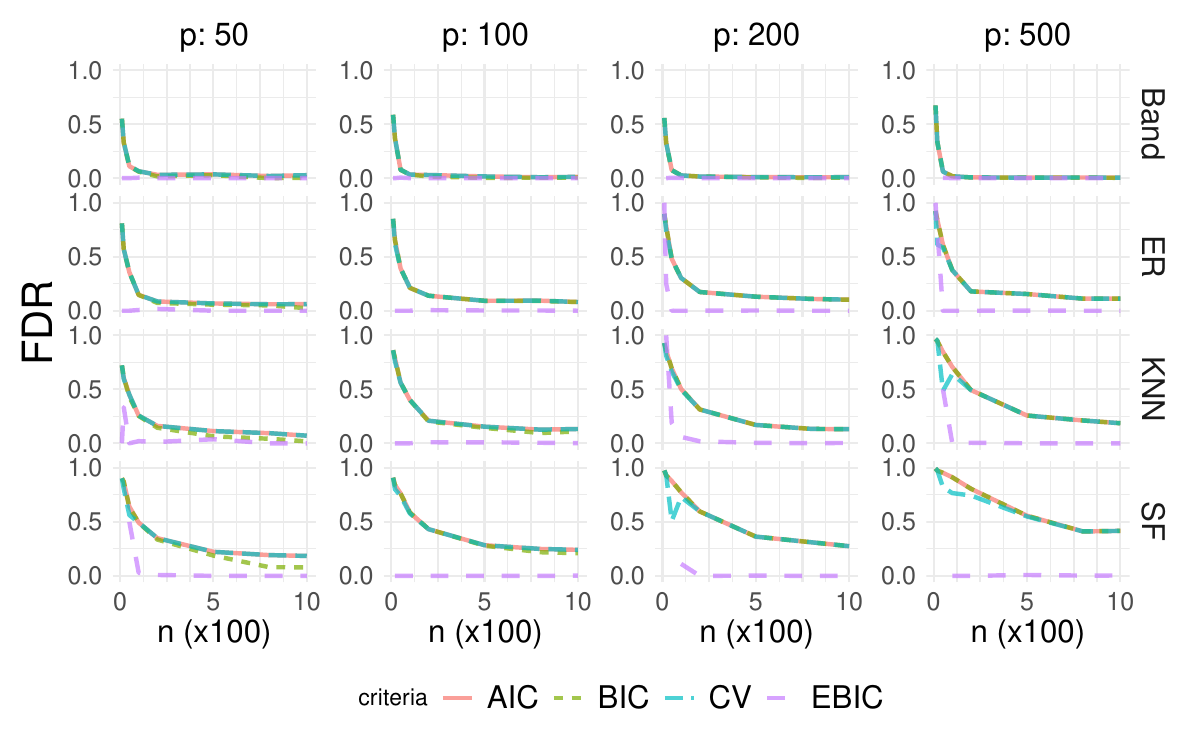}
  \caption{FDR vs. sample size $n$ (in hundreds) and $p$ on different graph types with NS to compare criteria.}
  \label{fig: fig2.2 NS TPR and FDR}
\end{figure*}

We use the  CV-selected $\lambdacv$ to approximate $\lambdao$ and compare its neighborhood estimate with those selected by several commonly used criteria: Akaike information criterion (AIC)~\cite{akaike1974new}, Bayesian information criterion (BIC)~\cite{schwarz1978estimating}, and Extended Bayesian information criteria (EBIC)~\cite{foygel2010extended}. The performance is evaluated via 
\begin{enumerate}
    \item The average structural hamming distance (SHD) to measure the number of incorrectly identified neighbors;
    \item The average true positive rate (TPR) and false discovery rate (FDR).
\end{enumerate} 
We let the number of observations $n$ and the number of observed variables $p$ to vary independently so one can be greatly larger than another. To validate comparisons between criteria, we simulate four different graphs: the Band graph, Scale-Free (SF), Erdös-Rényi (ER) and K-Nearest Neighbor (KNN) graphs. Besides the neighborhood Selection (NS), we also use three popular algorithms for Lasso estimators: Graphical lasso (Glasso)~\cite{friedman2008sparse}, constrained $\ell_1$-minimization for inverse matrix estimation (CLIME)~\cite{cai2011constrained} and Tuning-Insensitive Graph Estimation and Regression (TIGER)~\cite{liu2017tiger}. They are implemented in the \texttt{glasso} and \texttt{flare} packages for R.

\paragraph{Results}

We focus on NS here so as to corroborate our main theoretical results. Exceeding wall time limit (three hours) or undefined FDR value for all zero estimates are marked as missing points for plotting. Performance results with Glasso, Clime and Tiger are postponed to the supplementary materials, with similar conclusions.

In Figure~\ref{fig: fig1 NS SHD}, as expected, we see the number of incorrectly identified neighbors decreasing with increasing sample size $n$ given a fixed $p$. As $p$ increases, the task becomes increasingly harder, which is reflected by greater average SHD. The average SHD of CV always decreases with increasing $n$ but never reaches zero, regardless of how large $n$ gets. 
This pattern 
persists for Glasso, Clime and Tiger as well (see the supplementary materials). This confirms Theorem~\ref{thm:main}.
Besides CV, AIC performs poorly as well and never reaches zero average SHD.
On the other hand, BIC achieves smaller average SHD, but its performance in high-dimensions is unsatisfactory, which is consistent with known results for BIC \citep[e.g.][]{mestres2018selection}. The only candidate with constant decreasing trend with increasing $n$ for all $p$ is EBIC. Specifically, EBIC is always the first one to get closest to or reach zero, regardless of $p$.

The correctness of EBIC is more obvious when comparing averaged FDR in Figure~\ref{fig: fig2.1 NS TPR and FDR} and~\ref{fig: fig2.2 NS TPR and FDR}. Unsurprisingly, we see TPR gradually reaches 100\%, depending on the graph type. Specifically, CV and AIC are the first to reach 100\% TPR, while EBIC falls behind. 
However, this is not whole story:
The FDR for CV is problematic, and fails to get near 0\% on average.
(For the Band graph, we provide a more detailed numerical FDR summary via NS tuned by CV in the supplementary materials.) 
Moreover,
CV fails in average FDR in all other algorithms (see the supplementary materials) as well.

\section{Conclusion}

Cross-validation is the parameter selection criterion of choice in most ML applications, however, its suitability for structure learning problems is not well understood outside of empirical observations.
To address this gap, we proved that for a general family of Gaussian graphical models, including DAG models, CV is provably inconsistent for learning the structure of a graph. This shows that using CV as a naive alternative to difficult-to-implement selection criteria is ill-advised. It would be of interest to extend our proofs to non-Gaussian models, where our experiments indeed suggest CV is still inconsistent. On the positive side, our experiments indicate that EBIC is robust across a wide range of settings.

\bibliography{ref}
\bibliographystyle{abbrvnat}

\appendix
\section{Proof of Theorem \ref{thm:main}}

In this section, we will prove Theorem \ref{thm:main}. Detailed proof of various supporting lemmas can be found in Appendix~\ref{sec:omitted_proof}.
Before we go into the detail, we present some general observations.

We first state a useful lemma from \cite{meinshausen2006} which we will use often: These are the well-known KKT conditions for the Lasso solution.
We first define the following notation.
For any matrix $\bX$ and $i\in[p-1]$, we define
\begin{align*}
    G_{\bX,i}(\theta)
    :=
    \frac{1}{n}\inner{\bX_p - \sum_{j=1}^{p-1}\theta_j\bX_j}{\bX_i} \numberthis\label{eq:g_def}.
\end{align*}

\begin{lemma}[KKT conditions, \cite{meinshausen2006}]\label{lem:subgradient}
For any matrix $\bX  \in \mathbb{R}^{n\times p}$ and $\theta\in\mathbb{R}^{p-1}$, we have the following KKT conditions.
For any $\lambda>0$,
\begin{align*}
    \theta=\wtheta^{\lambda}
\end{align*}
if and only if
\begin{align*}
    \begin{cases}
        \text{if $\theta_i\neq 0$, then $G_{\bX,i}(\theta)=\sign(\theta_i)\lambda$} \\
        \text{if $\theta_i=0$, then $\abs{G_{\bX,i}(\theta)}\leq \lambda$}
    \end{cases}\numberthis\label{eq:kkt}
\end{align*}
\end{lemma}

Recall the definition of $\Gamma$ in \eqref{eq:nhbdlasso}, and define an ellipsoid by
\begin{align*}
\mathcal{E}
& :=
    \{\theta\in\R^{p-1} : (\theta - \theta^*)^\top\Gamma(\theta - \theta^*) \leq (\wtheta^{\lambdao} - \theta^*)^\top\Gamma(\wtheta^{\lambdao} - \theta^*)\}.
\end{align*}
It is clear that the Lasso solution for the oracle penalty $\lambdao$, $\wtheta^{\lambdao}$, lies on the boundary of this ellipsoid.
Our next observation is that for any matrix $\bX$, any point $\theta\in\mathbb{R}^{p-1}$ inside the ellipsoid $\mathcal{E}$ cannot be a Lasso solution for any penalty $\lambda>0$:
\begin{restatable}{lemma}{oellipsoid}\label{lem:o_ellipsoid}
For any matrix $\bX \in \mathbb{R}^{n\times p}$ and $\lambda>0$, if $\theta\in \mathbb{R}^{p-1}$ satisfies 
\begin{align*}
    (\theta - \theta^*)^\top\Gamma(\theta - \theta^*)<(\wtheta^{\lambdao} - \theta^*)^\top\Gamma(\wtheta^{\lambdao} - \theta^*)
\end{align*}
then $\theta\neq \wtheta^{\lambda}$.
\end{restatable}
The proof of this lemma can be found in Section \ref{sec:o_ellipsoid}.

By Lemma \ref{lem:subgradient}, for $i\in \wne^{\lambdao}$, we have $\abs{G_{\bX,i}(\wtheta^{\lambdao})} = \lambdao$.
Let $\wne^{G}$ be the set 
\begin{align*}
    \wne^G
    & =
    \setdef{i\in[p-1]}{\abs{G_{\bX,i}(\wtheta^{\lambdao})} = \lambdao}.\numberthis\label{eq:wneg_def}
\end{align*}
Note that $\wne^{\lambdao} \subseteq \wne^G$ by Lemma \ref{lem:subgradient}.
The following lemma shows that $\wne^G$ has at most one more extra element almost surely.

\begin{restatable}{lemma}{wnesize}\label{lem:wne_size}

Let $\bX\in\mathbb{R}^{n \times p}$ be a sample matrix where each row is an i.i.d.~sample drawn from $\mathcal{N}(\bzero,\Sigma)$.
Suppose the number of samples $n > \abs{\wne^{\lambdao}} + 2$.
Then, we have $\abs{\wne^{G}}\leq \abs{\wne^{\lambdao}} + 1$ almost surely.

\end{restatable}
The proof of this lemma can be found in Section \ref{sec:wne_size}.

To prove Theorem \ref{thm:main}, we want to argue that the event $\wne^{\lambdao} = \ne^*$ is unlikely to happen.
Without loss of generality, we assume that $\ne^* =[s]$ for some integer $s$.
If we assume $\ne^* = \wne^{\lambdao}$, then we have $\wne^{\lambdao} = [s]$.
Also, since $\abs{\wne^{G}}$ is either $\abs{\wne^{\lambdao}}$ or $\abs{\wne^{\lambdao}}+1$ and $\wne^{\lambdao}\subseteq\wne^G$, we assume $\wne^{G} = [s+1]$ if $\wne^{G}\neq \wne^{\lambdao}$.
We will consider these two cases separately.
Namely, we need to bound the following probability
\begin{align*}
    &
    \mathsf{Pr}_{\bX\sim \mathcal{N}(\bzero,\Sigma)^{n}}(\wne^{\lambdao} = \ne^*) \\
    & =
    \mathsf{Pr}_{\bX\sim \mathcal{N}(\bzero,\Sigma)^{n}}(\wne^{\lambdao} = [s]) \\
    & =
    \mathsf{Pr}_{\bX\sim \mathcal{N}(\bzero,\Sigma)^{n}}(\wne^{\lambdao}=\wne^{G} = [s]) + 
    \mathsf{Pr}_{\bX\sim \mathcal{N}(\bzero,\Sigma)^{n}}(\wne^{\lambdao} =[s]\wedge \wne^{G}=[s+1]).\numberthis\label{eq:main_prob}
\end{align*}
We preview that the first term can be bounded by $O\big(2^s\cdot \big( p^{-\Omega(\frac{1}{\kappa(\Gamma)})} + spe^{-\Omega(\frac{n}{s^2\kappa(\Gamma)^3})}\big) \big)$ (cf. Section \ref{sec:line_case}) and the second term can be bounded by $O\big(2^s\cdot \big( p^{-\Omega(\frac{1}{\kappa(\Gamma)})} + spe^{-\Omega(\frac{n}{s^2\kappa(\Gamma)^6})}\big) \big)$ (cf. Section \ref{sec:rays_case}).
By plugging these into \eqref{eq:main_prob},
\begin{align*}
    \mathsf{Pr}_{\bX\sim \mathcal{N}(\bzero,\Sigma)^{n}}(\wne^{\lambdao} = \ne^*)
    & \leq 
    O\bigg(2^s\cdot \bigg( p^{-\Omega(\frac{1}{\kappa(\Gamma)})} + spe^{-\Omega(\frac{n}{s^2\kappa(\Gamma)^3})}\bigg) \bigg) + O\bigg(2^s\cdot \bigg( p^{-\Omega(\frac{1}{\kappa(\Gamma)})} + spe^{-\Omega(\frac{n}{s^2\kappa(\Gamma)^6})}\bigg) \bigg) \\ 
    & \leq
    O\bigg(2^s\cdot \bigg( p^{-\Omega(\frac{1}{\kappa(\Gamma)})} + spe^{-\Omega(\frac{n}{s^2\kappa(\Gamma)^6})}\bigg) \bigg)
\end{align*}
we finish the proof of Theorem \ref{thm:main}.

\subsection{Case of $\wne^{\lambdao}=\wne^{G} = [s]$} \label{sec:line_case}

In this case, the main idea is to argue that there exists a line passing through $\wtheta^{\lambdao}$ such that any point in the intersection of a small neighborhood of $\wtheta^{\lambdao}$ and this line is also a Lasso solution for some penalty $\lambda$.

To help the construction of this line, we first define the following notations.
For any set of $n$ samples $\bX$, let $\wGamma$ be the $(p-1)$-by-$(p-1)$ matrix that $(r,c)$-entry is $\frac{1}{n}\inner{\bX_r}{\bX_c}$ for any $r,c\in[p-1]$ and $\wGamma_{[s]}$ be its $s$-by-$s$ submatrix where the indices are in $[s]$.
Also, let $q_{[s]}$ be the $s$-dimensional vector that the $i$-th entry is $\sign(\wtheta_i^{\lambdao})$ for $i\in[s]$.
We claim that the following line $L$ satisfies the above condition.
\begin{align*}
    L
    & :=
    \setdef{\wtheta^{\lambdao} + \delta \theta'}{ \delta \in\mathbb{R}}
\end{align*}
where $\theta'$ is the vector such that the $i$-the entry of $\theta'$ is 
\begin{align*}
    \theta'_i 
    & = 
    \begin{cases}
    -(\wGamma_{[s]}^{-1} q_{[s]})_i & \text{for $i\in [s]$} \\
    0 & \text{for $i\notin [s]$.}
    \end{cases} \numberthis \label{eq:theta_prime}
\end{align*}

The following lemma suggests that any $\theta$ in the intersection of a small neighborhood of $\wtheta^{\lambdao}$  and the line $L$ is also a Lasso solution for some penalty $\lambda>0$.
\begin{restatable}{lemma}{linelasso}\label{lem:line_lasso}
    For a sufficiently small $\abs{\delta}$, suppose $\widetilde{\theta}=\wtheta^{\lambdao}+\delta \theta'\in L$ where $\theta'$ is defined in \eqref{eq:theta_prime}.
    Then, $\widetilde{\theta}$ is a Lasso solution for some penalty $\lambda>0$.
\end{restatable}
The proof of this lemma can be found in Section \ref{sec:line_lasso}.

The main idea is to use Lemma \ref{lem:subgradient} and check the KKT conditions.
Since $\wtheta^{\lambdao}$ by definition is a Lasso solution for $\lambdao$, $\wtheta^{\lambdao}$ satisfies \eqref{eq:kkt} or we have
\begin{align*}
    \frac{1}{n}\inner{\bX_p - \sum_{j=1}^{p-1}\wtheta_j^{\lambdao}\bX_j}{\bX_i} & = \sign(\wtheta^{\lambdao}_i)\lambdao & \text{for $i\in[s]$.}
\end{align*}
Obviously, by the assumption of $\wne^{\lambdao}=[s]$, it also satisfies
\begin{align*}
    \wtheta^{\lambdao}_i &= 0  & \text{for $i\notin[s]$.}
\end{align*}
If we consider the following system of linear equations with $\theta\in\mathbb{R}^{p-1}$ and $\lambda\in\mathbb{R}$ as variables
\begin{align*}
    \frac{1}{n}\inner{\bX_p - \sum_{j=1}^{p-1}\theta_j\bX_j}{\bX_i} & = \sign(\wtheta^{\lambdao}_i)\lambda & \text{for $i\in[s]$} \\
    \theta_i & = 0 & \text{for $i\notin[s]$,}
\end{align*}
we can check that $(\wtheta^{\lambdao}+\delta \theta', \lambdao+\delta)$ satisfies this system of linear equations for all $\delta\in\mathbb{R}$.
Furthermore, for a sufficiently small $\abs{\delta}$, it satisfies \eqref{eq:kkt} in Lemma \ref{lem:subgradient}.
Hence, Lemma \ref{lem:line_lasso} follows.

Recall that, by Lemma \ref{lem:o_ellipsoid}, all points inside the ellipsoid $\mathcal{E}$ cannot be a Lasso solution for any penalty.
If the line $L$ is not a tangent of $\mathcal{E}$ at $\wtheta^{\lambdao}$, there exists a point inside $\mathcal{E}$ such that it is in the intersection of a neighborhood of $\wtheta^{\lambdao}$ and the line $L$.
It contradicts Lemma \ref{lem:o_ellipsoid} and hence $L$ has to be the tangent of $\mathcal{E}$ at $\wtheta^{\lambdao}$.
Note that the normal vector of the tangent space of $\mathcal{E}$ at $\wtheta^{\lambdao}$ is $\Gamma(\wtheta^{\lambdao} - \theta^*)$.
Then, $L$ is the tangent line of $\mathcal{E}$ at $\wtheta^{\lambdao}$ if and only if ${\theta'}^\top\Gamma(\wtheta^{\lambdao} - \theta^*) = 0$.
In other words, we have 
\begin{align*}
    \wne^{\lambdao} = \wne^G = [s]
    & \implies
    \theta'\Gamma(\wtheta^{\lambdao} - \theta^*) = 0.
\end{align*}
Let $\Dtheta$ be $\wtheta^{\lambdao} - \theta^*$.
Recall that if $i\notin[s]$ then $\Dtheta_i=0$ from the event $\wne^{\lambdao} = [s]$ and $\theta'_i= (\wGamma^{-1}q)_i=0$ from the construction in \eqref{eq:theta_prime}.
Namely, we can rewrite the expression $\theta'\Gamma \Dtheta$ as 
\begin{align*}
    \theta'\Gamma \Dtheta
    & =
    q_{[s]}^\top \wGamma_{[s]}^{-1} \Gamma_{[s]} \Dtheta_{[s]}.
\end{align*}
By Lemma \ref{lem:almost_tangent}, we have 
\begin{align*}
    \mathsf{Pr}_{\bX\sim \mathcal{N}(\bzero,\Sigma)^{n}}(\wne^{\lambdao} = [s] \wedge q_{[s]}^\top \wGamma_{[s]}^{-1} \Gamma_{[s]} \Dtheta_{[s]} = 0)
    & \leq
    O\bigg(2^s\cdot \bigg( p^{-\Omega(\frac{1}{\kappa(\Gamma)})} + spe^{-\Omega(\frac{n}{s^2\kappa(\Gamma)^3})}\bigg) \bigg)
\end{align*}
Note that the event $q_{[s]}^\top \wGamma_{[s]}^{-1} \Gamma_{[s]} \Dtheta_{[s]} = 0$ is more restrictive and indeed implies $\abs{q_{[s]}\wGamma_{[s]}^{-1} \Gamma_{[s]} \Dtheta_{[s]}}\leq \frac{1}{100\sqrt{s\sigma_{\max}(\Gamma)}}\sqrt{{\Dtheta_{[s]}}^\top\Gamma_{[s]}\Dtheta_{[s]}}$.
In other words, we have
\begin{align*}
    \mathsf{Pr}_{\bX\sim \mathcal{N}(\bzero,\Sigma)^{n}}(\wne^{\lambdao} = \wne^G = [s])
    & <
    O\bigg(2^s\cdot \bigg( p^{-\Omega(\frac{1}{\kappa(\Gamma)})} + spe^{-\Omega(\frac{n}{s^2\kappa(\Gamma)^3})}\bigg) \bigg).
\end{align*}

\subsection{Case of $\wne^{\lambdao} =[s]\wedge \wne^{G}=[s+1]$} \label{sec:rays_case}

In this case, the main idea is to argue that there exist two rays shooting from $\wtheta^{\lambdao}$ such that any point in the intersection of a small neighborhood of $\wtheta^{\lambdao}$ and these two rays is also a Lasso solution for some penalty $\lambda$.

To help the construction of this line, we define the following notations similar to the notations in last subsection.
For any set of $n$ samples $\bX$, let $\wGamma$ be the $(p-1)$-by-$(p-1)$ matrix that $(r,c)$-entry is $\frac{1}{n}\inner{\bX_r}{\bX_c}$ for any $r,c\in[p-1]$ and $\wGamma_{[s+1]}$ be its $(s+1)$-by-$(s+1)$ submatrix where the indices are in $[s+1]$.
Also, let $q_{[s+1]}$ be the $s$-dimensional vector that the $i$-th entry is $\sign(G_{\bX,i}(\wtheta^{\lambdao}))$ for $i\in[s+1]$ where $G_{\bX,i}$ is defined in \eqref{eq:g_def}.
Note that $q_i = \sign(\wtheta^{\lambdao}_i)$ when $i\in[s]$.
We claim that the following rays $R_1,R_2$ satisfy the above condition.
\begin{align*}
    R_1 
    :=
    \setdef{\wtheta^{\lambdao} + \delta\cdot\sign(Q) \theta'}{ \delta \geq 0}
\end{align*}
and
\begin{align*}
    R_2
    :=
    \setdef{\wtheta^{\lambdao} + \delta\cdot\sign(Q) \theta''}{\delta <0}
\end{align*}
where 
\begin{align*}
    \theta''_i
    & =
    \begin{cases}
        -(\wGamma_{[s+1]}^{-1}q_{[s+1]})_i & \text{for $i\in[s+1]$} \\
        0 & \text{for $i\notin[s+1]$}
    \end{cases} \numberthis\label{eq:theta_dprime}
\end{align*}
and
\begin{align*}
    Q
    =
    1-q_{s+1}\sum_{j=1}^{s}\theta'_j\cdot \frac{1}{n}\inner{\bX_j}{\bX_{s+1}}. \numberthis\label{eq:q_def}
\end{align*}
Recall that $\theta'$ is defined in \eqref{eq:theta_prime}.
Here we abuse the notation that $\sign(Q)=+1$ even when $Q=0$.

The following lemma shows that any $\theta$ in the intersection of a small neighborhood of $\wtheta^{\lambdao}$ and the union of these two rays $R_1\cup R_2$ is also a Lasso solution for some penalty $\lambda>0$.

\begin{restatable}{lemma}{rayslasso}\label{lem:rays_lasso}
For a sufficiently small $\abs{\delta}$, suppose  $\widetilde{\theta}=\wtheta^{\lambdao} + \delta\cdot \sign(Q) \theta' \in R_1$ for $\delta\geq 0$ and $\widetilde{\theta}=\wtheta^{\lambdao} + \delta \cdot \sign(Q)\theta'' \in R_2$ for $\delta<0$ where $\theta'$ and $\theta''$ are defined in \eqref{eq:theta_prime} and \eqref{eq:theta_dprime} respectively.
Then, $\widetilde{\theta}$ is a Lasso solution for some penalty $\lambda>0$.
\end{restatable}
The proof of this lemma can be found in Section \ref{sec:rays_lasso}.

The main idea is to use Lemma \ref{lem:subgradient} and check the KKT conditions like Lemma \ref{lem:line_lasso}.
When we consider $\widetilde{\theta}\in R_1$, the analysis is the same as in Lemma \ref{lem:line_lasso} except that we need to check $R_1$ shoots in the correct direction.
When we consider $\widetilde{\theta}\in R_2$, the analysis is similar.
The key difference is that we need to check the sign of $\widetilde{\theta}_{s+1}$ matches the sign of $G_{\bX,s+1}(\widetilde{\theta})$ since $\wtheta^{\lambdao}_{s+1}=0$ which may cause a sign mismatch in the perturbation $\widetilde{\theta}$.

Recall that, by Lemma \ref{lem:o_ellipsoid}, all points inside the ellipsoid $\mathcal{E}$ cannot be a Lasso solution for any penalty.
If the rays $R_1,R_2$ are shooting inside $\mathcal{E}$, there exists a point inside $\mathcal{E}$ such that it is in the intersection of a neighborhood of $\wtheta^{\lambdao}$ and the union of two rays $R_1\cup R_2$.
It contradicts Lemma \ref{lem:o_ellipsoid} and hence $R_1,R_2$ have to not shoot into $\mathcal{E}$.
Note that the normal vector of the tangent space of $\mathcal{E}$ at $\wtheta^{\lambdao}$ is $\Gamma(\wtheta^{\lambdao} - \theta^*)$.
Also, the direction of the ray $R_1$ is $\sign(Q)\theta'$ and the direction of the ray $R_2$ is $-\sign(Q)\theta''$.
Then, $R_1,R_2$ do not shoot into $\mathcal{E}$ if and only if ${\sign(Q)\theta'}^\top\Gamma(\wtheta^{\lambdao} - \theta^*) \geq 0$ and ${-\sign(Q) \theta''}^\top\Gamma(\wtheta^{\lambdao} - \theta^*) \geq 0$.
The following lemma shows that the probability of $R_1,R_2$ not shooting into $\mathcal{E}$ is small.
In other words, we have 
\begin{align*}
    \wne^{\lambdao}=[s] \wedge \wne^{G}=[s+1] 
    & \implies
    \mathcal{A} \wedge \mathcal{B}
\end{align*}
where $\mathcal{A}$ is the event of ${\sign(Q)\theta'}^\top\Gamma(\wtheta^{\lambdao} - \theta^*) \geq 0$ and $\mathcal{B}$ is the event of ${-\sign(Q) \theta''}^\top\Gamma(\wtheta^{\lambdao} - \theta^*) \geq 0$.
By Lemma \ref{lem:rays_out}, we have
\begin{align*}
    \mathsf{Pr}_{\bX\sim \mathcal{N}(\bzero,\Sigma)^{n}}(\wne^{\lambdao} = [s] \wedge \mathcal{A} \wedge \mathcal{B})
    & \leq
    O\bigg(2^s\cdot \bigg( p^{-\Omega(\frac{1}{\kappa(\Gamma)})} + spe^{-\Omega(\frac{n}{s^2\kappa(\Gamma)^6})}\bigg) \bigg).
\end{align*}
In other words, we have
\begin{align*}
    \mathsf{Pr}_{\bX\sim \mathcal{N}(\bzero,\Sigma)^{n}}(\wne^{\lambdao}=[s] \wedge \wne^{G}=[s+1] )
    & <
    O\bigg(2^s\cdot \bigg( p^{-\Omega(\frac{1}{\kappa(\Gamma)})} + spe^{-\Omega(\frac{n}{s^2\kappa(\Gamma)^6})}\bigg) \bigg).
\end{align*}

\section{Proof of Theorem \ref{thm:cv}}

We shorten $\lambdacv^\bX$ to be $\lambdacv$ if there is no ambiguity.

For any $\lambda>0$, define 
\begin{align*}
    Q(\lambda,\bX)
    & :=
    \frac{1}{2}\E_{Y\sim \mathcal{N}(\bzero,\Sigma)}(Y_p - \sum_{j=1}^{p-1}\wtheta^{\lambda,\bX}_jY_j)^2.
\end{align*}
We observe that, by the definition of $\lambdao$, 
\begin{align*}
    Q(\lambdao,\bX)
    & \leq
    Q(\lambdacv,\bX).
\end{align*}
If we manage to prove that
\begin{align*}
    Q(\lambdacv,\bX)
    & \leq
    Q(\lambdao,\bX) + \eps_n
\end{align*}
where $\eps_n$ is a value that $\eps_n\to 0$ as $n\to \infty$, then we can prove $\abs{\lambdacv-\lambdao} \to 0$ by the fact that $Q$ is a continuous function.

Before we go into the detail, we first state the following useful inequality.
By Chernoff bound and union bound, for any $\eps'>0$, we have
\begin{align*}
    \abs{\E(Y_iY_j) - \frac{1}{n/K}\inner{\bX^{I_k}_i}{\bX^{I_k}_j}} 
    & < 
    \eps' & \text{for all $i,j\in [p-1]$ and $k\in[K]$} \numberthis \label{eq:inner_approx}
\end{align*}
with probability $1-O(p^2Ke^{-\Omega(n\eps'^2/K\svmax(\Sigma)^2)})$.
From now on, our analysis is conditioned on \eqref{eq:inner_approx}.

For any $\lambda>0$ and any $\bZ\in\mathbb{R}^{n\times p},\bX\in \mathbb{R}^{m \times p}$, define 
\begin{align*}
    Q_{\bZ}(\lambda,\bX)
    & :=
    \frac{1}{2n}\norm{\bZ_p - \sum_{j=1}^{p-1}\wtheta^{\lambda,\bX}_j\bZ_j}^2.
\end{align*}

\begin{restatable}{lemma}{popapprox}\label{lem:pop_approx}

For any $\lambda>0$ and any $\bZ\in\mathbb{R}^{n\times p},\bX\in \mathbb{R}^{m \times p}$, we have
\begin{align*}
    \abs{Q(\lambda,\bX) - Q_\bZ(\lambda,\bX)}
    & =
    O(\eps'\cdot (1+\norm{\wtheta^{\lambda,\bX}}_1)^2).
\end{align*}

\end{restatable}
The proof of this lemma can be found in Section \ref{sec:pop_approx}.

\begin{restatable}{lemma}{thetabound}\label{lem:theta_bound}
For any $\lambda>0$ and any matrix $\bX\sim\mathcal{N}(\bzero,\Sigma)$, we have
\begin{align*}
    \norm{\wtheta^{\lambda,\bX}}_1
    & \leq
    O(\sqrt{p}\kappa(\Sigma))
\end{align*}
with probability $1-O(e^{-\Omega(n)})$.
Recall that $\kappa(\Sigma)$ is the condition number of $\Sigma$.
\end{restatable}
The proof of this lemma can be found in Section \ref{sec:theta_bound}.

By Lemma \ref{lem:pop_approx} with $\lambda=\lambdacv$ and $\bZ=\bX$, we first have
\begin{align*}
    Q(\lambdacv,\bX)
    & \leq
    Q_\bX(\lambdacv,\bX) + O(\eps'\cdot(1+\norm{\wtheta^{\lambdacv,\bX}}_1)^2).
\end{align*}
Furthermore, by Lemma \ref{lem:theta_bound}, we have
\begin{align*}
    Q(\lambdacv,\bX)
    & \leq
    Q_\bX(\lambdacv,\bX) + O(\eps'p\kappa(\Sigma)^2). \numberthis \label{eq:cv_first_ineq}
\end{align*}

For the term $Q_\bX(\lambdacv,\bX)$, we further have
\begin{align*}
    \MoveEqLeft Q_\bX(\lambdacv,\bX) \\
    & =
    Q_\bX(\lambdacv,\bX) + \lambdacv\norm{\wtheta^{\lambdacv,\bX}}_1 - \lambdacv\norm{\wtheta^{\lambdacv,\bX}}_1 \\
    & \leq
    Q_\bX(\lambdacv,\bX^{-I_k}) + \lambdacv\norm{\wtheta^{\lambdacv,\bX^{-I_k}}}_1 - \lambdacv\norm{\wtheta^{\lambdacv,\bX}}_1 \\
    & \leq
    Q_{\bX^{I_k}}(\lambdacv,\bX^{-I_k}) + O(\eps'\cdot(1+\norm{\wtheta^{\lambdacv,\bX^{-I_k}}}_1)^2)+ \lambdacv\norm{\wtheta^{\lambdacv,\bX^{-I_k}}}_1 - \lambdacv\norm{\wtheta^{\lambdacv,\bX}}_1 \\
    & \leq
    Q_{\bX^{I_k}}(\lambdacv,\bX^{-I_k}) + O(\eps'p\kappa(\Sigma)^2)+ \lambdacv(\norm{\wtheta^{\lambdacv,\bX^{-I_k}}}_1 - \norm{\wtheta^{\lambdacv,\bX}}_1) \numberthis\label{eq:main_error}
\end{align*}
for $k\in [K]$.
The penultimate inequality is due to the fact that $\wtheta^{\lambdacv,\bX} = \arg\min_{\theta\in\mathbb{R}^{p-1}} \frac{1}{2n}\norm{\bX_p - \sum_{j=1}^{p-1}\theta_j\bX_j}^2 + \lambdacv\norm{\theta}_1$ and the last inequality is due to applying Lemma \ref{lem:pop_approx} twice and triangle inequality.

In \eqref{eq:main_error}, the first term $Q_{\bX^{I_k}}(\lambdacv,\bX^{-I_k})$ is what we are looking for.
By the definition of $\lambdacv$, observe that
\begin{align*}
    \frac{1}{K}\sum_{k=1}^K Q_{\bX^{I_k}}(\lambdacv,\bX^{-I_k}) 
    & \leq
    \frac{1}{K}\sum_{k=1}^K Q_{\bX^{I_k}}(\lambdao,\bX^{-I_k}). \numberthis\label{eq:cv_def_ineq}
\end{align*}
Moreover, for each term $Q_{\bX^{I_k}}(\lambdao,\bX^{-I_k})$, we have
\begin{align*}
    \MoveEqLeft Q_{\bX^{I_k}}(\lambdao,\bX^{-I_k})\\
    & = 
    Q_{\bX^{-I_k}}(\lambdao,\bX^{-I_k}) + O(\eps'p\kappa(\Sigma)^2)\\
    & \leq
    Q_{\bX^{-I_k}}(\lambdao,\bX) + \lambdao\norm{\wtheta^{\lambdao,\bX}}_1 - \lambdao\norm{\wtheta^{\lambdao,\bX^{-I_k}}}_1 + O(\eps'p\kappa(\Sigma)^2) \\
    & \leq
    Q(\lambdao,\bX) + O(\eps'p\kappa(\Sigma)^2) + \lambdao(\norm{\wtheta^{\lambdao,\bX}}_1 - \norm{\wtheta^{\lambdao,\bX^{-I_k}}}_1)\numberthis \label{eq:cv_final_ineq}
\end{align*}

If we plug \eqref{eq:cv_final_ineq} into \eqref{eq:cv_def_ineq} and further plug it and \eqref{eq:main_error} into \eqref{eq:cv_first_ineq}, we have
\begin{align*}
    Q(\lambdacv,\bX)
    & \leq
    Q(\lambdao,\bX) + O(\eps'p\kappa(\Sigma)^2) \\
    & \qquad + 
    \lambdacv(\norm{\wtheta^{\lambdacv,\bX^{-I_k}}}_1 - \norm{\wtheta^{\lambdacv,\bX}}_1) + \lambdao\cdot\frac{1}{K}\sum_{k=1}^K(\norm{\wtheta^{\lambdao,\bX}}_1 - \norm{\wtheta^{\lambdao,\bX^{-I_k}}}_1).
\end{align*}

\begin{restatable}{lemma}{wthetalambdabound}\label{lem:wtheta_lambda_bound}

For any $\lambda>0$ and any matrices $\bX,\bZ \sim \mathcal{N}(\bzero,\Sigma)$, we have
\begin{align*}
    \abs{\norm{\wtheta^{\lambda,\bX}}_1 - \norm{\wtheta^{\lambda,\bZ}}_1}
    \leq
    O(\sqrt{\eps'\cdot \frac{p^2\kappa(\Sigma)^2}{\svmin(\Sigma)}})
    \qquad \text{and} \qquad
    \lambda 
    \leq
    O(\sqrt{p}\svmax(\Sigma)\kappa(\Sigma))
\end{align*}
as long as $\lambda$ is not too large such that $\wtheta^\lambda\neq \bzero$.
\end{restatable}
The proof of this lemma can be found in Section \ref{sec:wtheta_lambda_bound}.

By Lemma \ref{lem:wtheta_lambda_bound}, we have
\begin{align*}
    Q(\lambdacv,\bX)
    & \leq
    Q(\lambdao,\bX) + \sqrt{\eps'}\cdot \poly(p)\cdot C_\Sigma
\end{align*}
where $C_{\Sigma}$ is a constant depending only on $\Sigma$.
Taking $\eps'=\frac{1}{n^{1/10}\poly(p)\cdot C_\Sigma^2}$, we have
\begin{align*}
    Q(\lambdacv,\bX)
    & \leq
    Q(\lambdao,\bX) + \frac{1}{n^{1/20}}.
\end{align*}
In other words, we have
\begin{align*}
    \abs{Q(\lambdacv,\bX)-Q(\lambdao,\bX)}
    & <
    \frac{1}{n^{1/20}}
\end{align*}
with probability $1-O(p^2Ke^{-\Omega(\poly(n)\cdot \poly(1/p)C'_{\Sigma})})$ for some constant $C'_{\Sigma}$ depending only on $\Sigma$.
Using the fact that $Q$ is a continuous function, we can conclude our result.

\section{Omitted proofs}\label{sec:omitted_proof}

\subsection{Proof of Lemma \ref{lem:o_ellipsoid}}\label{sec:o_ellipsoid}

\oellipsoid*

\begin{proof}

Recall that the definition of $a$, $v$ and $\Gamma$ from \eqref{eq:sigma_def}.
Also, since $\Gamma$ is a positive definite matrix, there exists a matrix $H$ such that $\Gamma = HH^\top$.
For any $\theta\in\mathbb{R}^{p-1}$, we first expand $\E_{Y\sim \mathcal{N}(\bzero,\Sigma)}(Y_p - \sum_{j=1}^{p-1}\theta_jY_j)^2$ as
\begin{align*}
    \MoveEqLeft \E_{Y\sim \mathcal{N}(\bzero,\Sigma)}(Y_p - \sum_{j=1}^{p-1}\theta_jY_j)^2 \\
    & =
    \theta^\top\Gamma \theta - 2\theta^\top v + a \\
    & =
    \theta^\top \Gamma\theta - 2\theta^\top H H^{-1} v + v^\top (H^{-1})^\top H^{-1} v - v^\top (H^{-1})^\top H^{-1} v+a \\
    & =
    \norm{H^\top \theta - H^{-1} v}_2^2 + a - v^\top \Gamma^{-1} v.
\end{align*}
Recall that $\theta^* = \Gamma^{-1} v$.
By plugging it into the above equation, we have
\begin{align*}
    \E_{Y\sim \mathcal{N}(\bzero,\Sigma)}(Y_p - \sum_{j=1}^{p-1}\theta_jY_j)^2
    & =
    \norm{H^\top \theta - H^\top \theta^*}_2^2 + a - v^\top \Gamma^{-1} v \\
    & =
    (\theta- \theta^*)^\top \Gamma (\theta - \theta^*) + a - v^\top \Gamma^{-1} v \numberthis \label{eq:o_ineq}
\end{align*}

By the definition of $\lambdao$, we have the following inequality
\begin{align*}
    \E_{Y\sim \mathcal{N}(\bzero,\Sigma)}(Y_p - \sum_{j=1}^{p-1}\wtheta^{\lambdao}_jY_j)^2
    & \leq
    \E_{Y\sim \mathcal{N}(\bzero,\Sigma)}(Y_p - \sum_{j=1}^{p-1}\wtheta^{\lambda}_jY_j)^2 & \text{for any $\lambda>0$.}
\end{align*}
By \eqref{eq:o_ineq}, it implies 
\begin{align*}
    (\wtheta^{\lambdao}- \theta^*)^\top \Gamma (\wtheta^{\lambdao} - \theta^*)
    & \leq
    (\wtheta^{\lambda}- \theta^*)^\top \Gamma (\wtheta^{\lambda} - \theta^*).
\end{align*}

\end{proof}

\subsection{Proof of Lemma \ref{lem:wne_size}}\label{sec:wne_size}

\wnesize*

\begin{proof}

Suppose $\abs{\wne^G} \geq \abs{\wne^{\lambdao}} + 2$.
Without loss of generality, we assume $\wne^{\lambdao} = [s]$ for some integer $s$ and $[s+2]\subseteq \wne^G$.
Consider the following system of linear equations with $\theta\in\mathbb{R}^{p-1}$ and $\lambda\in\mathbb{R}$ as variables.
\begin{align*}
    \begin{cases}
        \frac{1}{n}\inner{\bX_p - \sum_{j=1}^{p-1}\theta_j \bX_j}{\bX_i} = \sign(G_{\bX,i}(\wtheta^{\lambdao}))\lambda & \text{for $i\in[s+2]$,} 
        \\
        \theta_i = 0 & \text{for $i\notin [s]$.}
    \end{cases} \numberthis\label{eq:main_sys}
\end{align*}
By the definition of $\wne^{\lambdao}$, $(\wtheta^{\lambdao},\lambdao)$ is a solution of this system.
It means this system of linear equations \eqref{eq:main_sys} has a solution.
Since \eqref{eq:main_sys} has a solution, the determinant of  matrix $\wGamma_{[s+2],[s]\cup\{p,*\}}$ is $0$ where $\wGamma_{[s+2],[s]\cup\{p,*\}}$ is the $(s+2)$-by-$(s+2)$ matrix that the rows are indexed by $[s+2]$ and the columns are indexed by $[s]\cup\{p,*\}$ and the $(r,c)$-entry is $\begin{cases}
    \frac{1}{n}\inner{\bX_r}{\bX_c} & \text{for $r\in[s+2]$, $c\in[s]\cup\{p\}$}\\
    \sign(G_{\bX,r}(\wtheta^{\lambdao})) & \text{for $r\in[s+2]$, $c=*$}
\end{cases}$, i.e.
\begin{align*}
    \wGamma_{[s+2],[s]\cup\{p,*\}}
    & =
    \begin{bmatrix}
        \frac{1}{n}\inner{\bX_1}{\bX_1} & \cdots & \frac{1}{n}\inner{\bX_1}{\bX_s} & \frac{1}{n}\inner{\bX_1}{\bX_p} & \sign(G_{\bX,1}(\wtheta^{\lambdao})) \\
        \vdots & \ddots & \vdots & \vdots & \vdots \\
        \frac{1}{n}\inner{\bX_{s+2}}{\bX_1} & \cdots & \frac{1}{n}\inner{\bX_{s+2}}{\bX_s} & \frac{1}{n}\inner{\bX_{s+2}}{\bX_p} & \sign(G_{\bX,s+2}(\wtheta^{\lambdao}))
    \end{bmatrix}.
\end{align*}

Now, we project $\bX_p$ onto the subspace spanned by $\{\bX_1,\bX_2,\dots,\bX_{s+2}\}$ and write the project as $\sum_{j=1}^{s+2} \alpha_j \bX_j$ for some $\alpha_j$.
Plugging it into $\det(A)=0$ and using the properties of determinant, we have
\begin{align*}
    \alpha_{s+1}\det(\wGamma_{[s+2],[s]\cup\{s+1,*\}}) + \alpha_{s+2}\det(\wGamma_{[s+2],[s]\cup\{s+2,*\}})
    & =
    0 \numberthis\label{eq:gaussian_zero}
\end{align*}
where 
\begin{align*}
    \wGamma_{[s+2],[s]\cup\{j,*\}}
    & =
    \begin{bmatrix}
        \frac{1}{n}\inner{\bX_1}{\bX_1} & \cdots & \frac{1}{n}\inner{\bX_1}{\bX_s} & \frac{1}{n}\inner{\bX_1}{\bX_j} & \sign(G_{\bX,1}(\wtheta^{\lambdao})) \\
        \vdots & \ddots & \vdots & \vdots & \vdots \\
        \frac{1}{n}\inner{\bX_{s+2}}{\bX_1} & \cdots & \frac{1}{n}\inner{\bX_{s+2}}{\bX_s} & \frac{1}{n}\inner{\bX_{s+2}}{\bX_j} & \sign(G_{\bX,s+2}(\wtheta^{\lambdao}))
    \end{bmatrix} & \text{for $j=\{s+1,s+2\}.$}
\end{align*}
Note that, conditioned on $\bX_1,\dots,\bX_{s+2}$, $\alpha_{s+1},\alpha_{s+2}$ are distributed as Gaussian if $n > s+2$.
If one of $\det(\wGamma_{[s+2],[s]\cup\{s+1,*\}})$ and $\det(\wGamma_{[s+2],[s]\cup\{s+2,*\}})$ is not zero, the expression $\alpha_{s+1}\det(\wGamma_{[s+2],[s]\cup\{s+1,*\}}) + \alpha_{s+2}\det(\wGamma_{[s+2],[s]\cup\{s+2,*\}})$ is distributed as Gaussian conditioned on $\bX_1,\dots,\bX_{s+2}$ and hence is non-zero almost surely which contradicts \eqref{eq:gaussian_zero}.
Therefore, $\abs{\wne^G} \leq \abs{\wne^{\lambdao}} + 1$.

It remains to show that $\det(\wGamma_{[s+2],[s]\cup\{s+2,*\}})$ (or $\det(\wGamma_{[s+2],[s]\cup\{s+1,*\}})$) is non-zero almost surely.
By Cramer's rule, $\frac{\det(\wGamma_{[s+2],[s]\cup\{s+2,*\}})}{\det(\wGamma_{[s+2]})}$ is the $(s+2)$-th entry of $\wGamma_{[s+2]}^{-1}q_{[s+2]}$ where $\wGamma_{[s+2]}$ is the $(s+2)$-by-$(s+2)$ matrix whose $(r,c)$-entry is $\frac{1}{n}\inner{\bX_r}{\bX_c}$ for $r,c\in[s+2]$ and $q_{[s+2]}$ is the $(s+2)$-dimensional vector whose $i$-th entry is $\sign(G_{\bX,i})$ for $i\in[s+2]$.
Note that $\det(\wGamma_{[s+2]}) \neq 0$ almost surely if $n > s+2$.
By the block matrix calculation, we have
\begin{align*}
    (\wGamma_{[s+2]}^{-1}q_{[s+2]})_{s+2}
    & =
    \frac{q_{s+2} - \wu^\top\wGamma_{[s+1]}^{-1}q_{[s+1]}}{\wGamma_{s+2,s+2} - \wu^\top\wGamma_{[s+1]}^{-1}\wu}
\end{align*}
where $\wu$ is the $(s+1)$-dimensional vector whose $i$-th entry is $\frac{1}{n}\inner{\bX_i}{\bX_{s+2}}$ for $i\in [s+1]$, $\wGamma_{[s+1]}$ is the $(s+1)$-by-$(s+1)$ submatrix of $\wGamma_{[s+2]}$ whose indices are in $[s+1]$ and $q_{[s+1]}$ is the $(s+1)$-dimensional subvector of $q_{[s+2]}$ whose indices are in $[s+1]$.
Suppose $(\wGamma_{[s+2]}^{-1}q_{[s+2]})_{s+2} = 0$.
We have
\begin{align*}
    q_{s+2} - \wu^\top\wGamma_{[s+1]}^{-1}q_{[s+1]}  = 0.\numberthis\label{eq:gaussian_zero_2}
\end{align*}
Conditioned on $\bX_1,\dots,\bX_{s+1}$, the term $\wu^\top\wGamma_{[s+1]}^{-1}q_{[s+1]}$ depends linearly on $\bX_{s+2}$ and hence it is distributed as Gaussian which is almost surely not $1$ or $-1$.
It contradicts \eqref{eq:gaussian_zero_2}.
Therefore, $\det(\wGamma_{[s+2],[s]\cup\{s+2,*\}})\neq 0$.

\end{proof}

\subsection{Proof of Lemma \ref{lem:line_lasso}}\label{sec:line_lasso}

\linelasso*

\begin{proof}

To use Lemma \ref{lem:subgradient}, we examine $G_{\bX,i}(\wtheta^{\lambdao} + \delta \theta')$ for any $i$ and we need to check the two cases of coordinates of $\wtheta^{\lambdao} + \delta \theta'$ being non-zero or not in Lemma \ref{lem:subgradient} which correspond to $i\in[s]$ and $i\notin[s]$.

We first examine it for $i\in[s]$.
By direct calculation, we have
\begin{align*}
    G_{\bX,i}(\wtheta^{\lambdao} + \delta \theta')
    & =
    \frac{1}{n}\inner{\bX_p - \sum_{j=1}^{p-1}(\wtheta^{\lambdao}_j + \delta\theta'_j)\bX_{j}}{\bX_i} 
     =
    q_i(\lambdao + \delta) & \text{for $i \in [s]$}\numberthis\label{eq:check_kkt_a}
\end{align*}
where $q_i$ is the $i$-th entry of $q_{[s]}$.

We now examine it for $i\notin[s]$.
By direct calculation, we have
\begin{align*}
    \abs{G_{\bX,i}(\wtheta^{\lambdao} + \delta\theta')}
    & =
    \abs{\frac{1}{n}\inner{\bX_p - \sum_{j=1}^{p-1}(\wtheta_j^{\lambdao}+\delta\theta_j')\bX_j}{\bX_i}} \\
    & \leq
    \abs{\frac{1}{n}\inner{\bX_p - \sum_{j=1}^{p-1}\wtheta_j^{\lambdao}\bX_j}{\bX_i}} + C\cdot\abs{\delta} \\
    & =
    \abs{G_{\bX,i}(\wtheta^{\lambdao})} + C\cdot\abs{\delta} & \text{for $i \notin [s]$}
\end{align*}
where $C = \abs{\sum_{j=1}^{p-1}\theta'_j\inner{\bX_j}{\bX_i}}$.
If we combine with the fact that $\abs{G_{\bX,i}(\wtheta^{\lambdao})} < \lambdao$ for $i\notin [s]$, for a sufficiently small $\abs{\delta}$, we have
\begin{align*}
    \abs{G_{\bX,i}(\wtheta^{\lambdao} + \delta\theta')}
    & <
    \abs{G_{\bX,i}(\wtheta^{\lambdao})} + C\cdot\abs{\delta}
    <
    \lambdao + \delta.\numberthis\label{eq:check_kkt_a2}
\end{align*}

Using Lemma \ref{lem:subgradient} with \eqref{eq:check_kkt_a} and \eqref{eq:check_kkt_a2}, we conclude that $\wtheta^{\lambdao} + \delta \theta'$ is a Lasso solution for $\lambda = \lambdao+\delta$ for a sufficiently small $\abs{\delta}$.

\end{proof}

\subsection{Proof of Lemma \ref{lem:rays_lasso}}\label{sec:rays_lasso}

\rayslasso*

\begin{proof}

To use Lemma \ref{lem:subgradient}, we examine both $G_{\bX,i}(\wtheta^{\lambdao}+\delta \cdot \sign(Q)\theta')$ when $\delta\geq 0$ and $G_{\bX,i}(\wtheta^{\lambdao}+\delta \cdot\sign(Q)\theta'')$ when $\delta<0$ for any $i$.

We first analyze $G_{\bX,i}(\wtheta^{\lambdao}+\delta\cdot \sign(Q) \theta')$ when $\delta\geq 0$ and the analysis is similar to the analysis in Lemma \ref{lem:line_lasso}.
We need to check the two cases of coordinates of $\wtheta^{\lambdao}+\delta\cdot\sign(Q) \theta'$ being non-zero or not which correspond to $i\in[s]$ and $i\notin[s]$.
For $i\in[s]$, by the same calculation as in Lemma \ref{lem:line_lasso}, we have
\begin{align*}
    G_{\bX,i}(\wtheta^{\lambdao} + \delta\cdot \sign(Q) \theta')
    & =
    q_i(\lambdao + \delta\cdot\sign(Q)) & \text{for $i \in [s]$.}\numberthis\label{eq:check_kkt_b}
\end{align*}
For $i\notin[s]$, we further split into two cases: $i\notin[s+1]$ and $i=s+1$.
For $i\notin[s+1]$, by the same calculation as in Lemma \ref{lem:line_lasso}, we have
\begin{align*}
    \abs{G_{\bX,i}(\wtheta^{\lambdao} + \delta\cdot\sign(Q)\theta')}
    & \leq
    \abs{G_{\bX,i}(\wtheta^{\lambdao})} + C\cdot\abs{\delta}. & \text{for $i \notin [s+1]$}
\end{align*}
where $C = \abs{\sum_{j=1}^{p-1}\theta'_j\inner{\bX_j}{\bX_i}}$ and hence
\begin{align*}
    \abs{G_{\bX,i}(\wtheta^{\lambdao} + \delta\theta')}
    & <
    \lambdao + \delta\cdot\sign(Q)\numberthis\label{eq:check_kkt_b2}
\end{align*}
for a sufficiently small $\abs{\delta}$.
For $i=s+1$, we have
\begin{align*}
    &
    G_{\bX,s+1}(\wtheta^{\lambdao} + \delta\cdot\sign(Q) \theta') \\
    & =
    q_{s+1}\lambdao + \delta\cdot\sign(Q) \sum_{j=1}^{s}\theta'_j\cdot \frac{1}{n}\inner{\bX_j}{\bX_{s+1}} \\
    & =
    q_{s+1}\big(\lambdao + \delta\cdot\sign(Q) - \delta\cdot\sign(Q)\big(1-q_{s+1}\sum_{j=1}^{s}\theta'_j\cdot \frac{1}{n}\inner{\bX_j}{\bX_{s+1}}\big)\big) \\
    & =
    q_{s+1}(\lambdao + \delta\cdot\sign(Q) - \sign(Q)Q\cdot\delta) \qquad \text{recall that $Q = 1-q_{s+1}\sum_{j=1}^{s}\theta'_j\cdot \frac{1}{n}\inner{\bX_j}{\bX_{s+1}}$}.
\end{align*}
and hence, using the fact that $\sign(Q)Q = \abs{Q} > 0$ and $\delta \geq 0$,
\begin{align*}
    \abs{G_{\bX,s+1}(\wtheta^{\lambdao} + \delta \theta')}
    & =
    \lambdao + \delta\cdot\sign(Q) - \abs{Q}\cdot\delta
    \leq 
    \lambdao + \delta\cdot\sign(Q)\numberthis\label{eq:check_kkt_b3}
\end{align*}
for a sufficiently small $\abs{\delta}$.
Using Lemma \ref{lem:subgradient} with \eqref{eq:check_kkt_b}, \eqref{eq:check_kkt_b2} and \eqref{eq:check_kkt_b3}, we conclude that $\wtheta+\delta\cdot\sign(Q) \theta'$ is a Lasso solution for $\lambda = \lambdao + \delta\cdot\sign(Q)$ for a sufficiently small $\abs{\delta}$ and $\delta\geq 0$.

We now analyze $G_{\bX,i}(\wtheta^{\lambdao} + \delta\cdot\sign(Q)\theta'')$ when $\delta<0$.
We need to check the two cases of coordinates of $\wtheta^{\lambdao} + \delta\cdot\sign(Q) \theta''$ being non-zero or not which correspond to $i\in[s+1]$ and $i\notin[s+1]$.
For $i\in[s+1]$, by a similar calculation as in Lemma \ref{lem:line_lasso}, we have
\begin{align*}
    G_{\bX,i}(\wtheta^{\lambdao} + \delta\cdot\sign(Q) \theta'')
    & =
    q_i(\lambdao + \delta\cdot\sign(Q)) & \text{for $i\in[s+1]$.} \numberthis\label{eq:check_kkt_b4}
\end{align*}
We further split into two cases: $i\in[s]$ and $i=s+1$.
For $i\in[s]$, we note that $q_i = \sign(\wtheta_i^{\lambdao}) = \sign(\wtheta_i^{\lambdao} + \delta\cdot\sign(Q) \theta_i'')$ for a sufficiently small $\abs{\delta}$.
For $i=s+1$, we check that $\wtheta_{s+1}^{\lambdao} + \delta\cdot\sign(Q) \theta_{s+1}'' = \delta\cdot\sign(Q) \theta''_{s+1}$ by the assumption on $\wtheta^{\lambdao}$.
From the definition of $\theta''$ in \eqref{eq:theta_dprime} and direct block matrix calculation, we have
\begin{align*}
    \theta''_{s+1}
    & =
    \frac{-q_{s+1}(1-q_{s+1}\wu^\top\wGamma_{[s]}^{-1}q_{[s]})}{\wGamma_{s+1,s+1}-\wu^\top\wGamma_{[s]}^{-1}\wu}
    =
    \frac{-q_{s+1}\cdot Q}{\wGamma_{s+1,s+1}-\wu^\top\wGamma_{[s]}^{-1}\wu}
\end{align*}
where $\wu$ is the $s$-dimensional vector whose $i$-th entry is $\wGamma_{i,s+1} = \frac{1}{n}\inner{\bX_i}{\bX_{s+1}}$ for $i\in[s]$.
Since the denominator $\wGamma_{s+1,s+1}-\wu^\top\wGamma_{[s]}^{-1}\wu$ is positive, $\sign(Q)Q = \abs{Q}>0$ and $\delta<0$, we have 
\begin{align*}
    \sign(\wtheta_{s+1}^{\lambdao} + \delta\cdot\sign(Q) \theta_{s+1}'')
    & =
    \sign(\delta\cdot\sign(Q) \theta_{s+1}'') 
    = 
    q_{s+1}.\numberthis\label{eq:check_kkt_b5}
\end{align*}
For $i\notin[s+1]$, by a similar calculation as in Lemma \ref{lem:line_lasso}, we have
\begin{align*}
    \abs{G_{\bX,i}(\wtheta^{\lambdao} + \delta\cdot\sign(Q)\theta'')}
    & \leq
    \abs{G_{\bX,i}(\wtheta^{\lambdao})} + C_1\cdot\abs{\delta}. & \text{for $i \notin [s+1]$}
\end{align*}
where $C_1 = \abs{\sum_{j=1}^{p-1}\theta''_j\inner{\bX_j}{\bX_i}}$ and hence
\begin{align*}
    \abs{G_{\bX,i}(\wtheta^{\lambdao} + \delta\cdot\sign(Q)\theta')}
    & <
    \lambdao + \delta\cdot\sign(Q)\numberthis\label{eq:check_kkt_b6}
\end{align*}
for a sufficiently small $\abs{\delta}$.
Using Lemma \ref{lem:subgradient} with \eqref{eq:check_kkt_b4}, \eqref{eq:check_kkt_b5} and \eqref{eq:check_kkt_b6}, we conclude that $\wtheta^{\lambdao} + \delta\cdot\sign(Q) \theta''$ is a Lasso solution for $\lambda=\lambdao+\delta\cdot\sign(Q)$ for a sufficiently small $\abs{\delta}$ and $\delta<0$.

\end{proof}

\subsection{Proof of Lemma \ref{lem:pop_approx}}\label{sec:pop_approx}

\popapprox*

\begin{proof}

By the definition of $Q(\lambda,\bX)$ and $Q_\bZ(\lambda,\bX)$,  we have
\begin{align*}
    Q(\lambda,\bX) - Q_\bZ(\lambda,\bX)
    & =
    \frac{1}{2}\E(Y_p - \sum_{j=1}^{p-1}\wtheta^{\lambda,\bX}_jY_j)^2  - \frac{1}{2n}\norm{\bZ_p - \sum_{j=1}^{p-1}\wtheta^{\lambda,\bX}_j\bZ_j}_2^2\\
    & =
    \frac{1}{2}\E(Y_p^2 - 2Y_p\sum_{j=1}^{p-1}\wtheta^{\lambda,\bX}_jY_j + \sum_{i=1}^{p-1}\sum_{j=1}^{p-1}\wtheta^{\lambda,\bX}_i\wtheta^{\lambda,\bX}_jY_iY_j) \\
    & \qquad -
    \frac{1}{2}\big(\frac{1}{n}\inner{\bZ_p}{\bZ_p} - 2\sum_{j=1}^{p-1}\wtheta^{\lambda,\bX}_j\frac{1}{n}\inner{\bZ_p}{\bZ_j} + \sum_{i=1}^{p-1}\sum_{j=1}^{p-1}\wtheta^{\lambda,\bX}_i\wtheta^{\lambda,\bX}_j\frac{1}{n}\inner{\bZ_i}{\bZ_j}\big)
\end{align*}

Recall that $\abs{\E(Y_iY_i) - \frac{1}{n}\inner{\bZ_i}{\bZ_j}} <\eps'$.
Hence, we have
\begin{align*}
    \abs{Q(\lambda,\bX) - Q_\bZ(\lambda,\bX)}
    & <
    \frac{\eps'}{2}\big(1 + 2\sum_{j=1}^{p-1}\abs{\wtheta^{\lambda,\bX}_j} + \sum_{i=1}^{p-1}\sum_{j=1}^{p-1}\abs{\wtheta^{\lambda,\bX}_i}\abs{\wtheta^{\lambda,\bX}_j}\big) \\
    & =
    \eps'\cdot (1+\norm{\wtheta^{\lambda,\bX}}_1)^2
\end{align*}

\end{proof}

\subsection{Proof of Lemma \ref{lem:theta_bound}}\label{sec:theta_bound}

\thetabound*

\begin{proof}

It is easy to see that
\begin{align*}
    \norm{\wtheta^{\lambda,\bX}}_1
    & \leq
    \norm{\wtheta^{0,\bX}}_1
    =
    \norm{\wGamma^{-1}\wv}_1.
\end{align*}
Furthermore, 
\begin{align*}
    \norm{\wGamma^{-1}\wv}_1
    & \leq
    \sqrt{p}\norm{\wGamma^{-1}\wv}_2 
    \leq
    \sqrt{p}\norm{\wGamma^{-1}}_2\norm{\wv}_2
\end{align*}
where $\wv$ is the $(p-1)$-dimensional vector whose $i$-th entry is $\frac{1}{n}\inner{\bX_i}{\bX_p}$ for $i\in [p-1]$.

By matrix Chernoff bound, we have
\begin{align*}
    \norm{\wGamma - \Gamma}_2 < O(\norm{\Gamma}_2)
    \qquad \text{and} \qquad
    \norm{\wv - v}_2 < O(\norm{v}_2)
\end{align*}
with probability $1-O(e^{-\Omega(n)})$.
Hence, we have
\begin{align*}
    \norm{\wtheta^{\lambda,\bX}}_1
    & \leq
    O(\sqrt{p}\norm{\Gamma^{-1}}_2\norm{v}_2)
    \leq 
    O(\sqrt{p}\kappa(\Sigma)).
\end{align*}

\end{proof}

\subsection{Proof of Lemma \ref{lem:wtheta_lambda_bound}}\label{sec:wtheta_lambda_bound}

\wthetalambdabound*

\begin{proof}

For any $\lambda>0$ and any matrix $\bX\in\mathbb{R}^{n\times p}$, define
\begin{align*}
    F_{\lambda,\bX}(\theta) 
    & := 
    \frac{1}{2n}\norm{\bX_p - \sum_{j=1}^{p-1}\theta_j\bX_j}^2 + \lambda\norm{\theta}_1.
\end{align*}

For any $\theta\in\mathbb{R}^{p-1}$, By Taylor expansion, we have
\begin{align*}
    F_{\lambda,\bX}(\theta)
    & =
    F_{\lambda,\bX}(\wtheta^{\lambda,\bX}) + \nabla F_{\lambda,\bX}(\wtheta^{\lambda,\bX})^\top (\theta - \wtheta^{\lambda,\bX}) \\
    & \qquad +
    \frac{1}{2}(\theta - \wtheta^{\lambda,\bX})^\top\nabla^2 F_{\lambda,\bX}(\wtheta^{\lambda,\bX})(\theta - \wtheta^{\lambda,\bX})
\end{align*}
Note that $\nabla F_{\lambda,\bX}(\wtheta^{\lambda,\bX})^\top (\theta - \wtheta^{\lambda,\bX}) \geq 0$; otherwise  it contradicts the fact that $\wtheta^{\lambda,\bX}$ is $\arg\min_{\theta\in\mathbb{R}^{p-1}} F_{\lambda,\bX}(\theta)$.
It means that 
\begin{align*}
    F_{\lambda,\bX}(\theta)
    & \geq
    F_{\lambda,\bX}(\wtheta^{\lambda,\bX})  +
    \frac{1}{2}(\theta - \wtheta^{\lambda,\bX})^\top\nabla^2 F_{\lambda,\bX}(\wtheta^{\lambda,\bX})(\theta - \wtheta^{\lambda,\bX}) \\
    &  \geq
    F_{\lambda,\bX}(\wtheta^{\lambda,\bX})  +
    \frac{1}{2}\svmin(\wGamma)\norm{\theta - \wtheta^{\lambda,\bX}}_2^2
\end{align*}
By matrix Chernoff bound, we have
\begin{align*}
    \norm{\wGamma - \Gamma}_2 < O(\norm{\Gamma}_2)
\end{align*}
with probability $1-O(e^{-\Omega(n)})$.
Namely, we have
\begin{align*}
    F_{\lambda,\bX}(\theta)
    &  \geq
    F_{\lambda,\bX}(\wtheta^{\lambda,\bX})  +
    \Omega(\svmin(\Gamma)\cdot\norm{\theta - \wtheta^{\lambda,\bX}}_2^2) \\
    & \geq 
    F_{\lambda,\bX}(\wtheta^{\lambda,\bX})  +
    \Omega(\frac{\svmin(\Gamma)}{p}\cdot\norm{\theta - \wtheta^{\lambda,\bX}}_1^2). \numberthis \label{eq:F_lower}
\end{align*}

On the other hand, 
\begin{align*}
    F_{\lambda,\bX}(\wtheta^{\lambda,\bZ})
    & =
    Q_{\bX}(\lambda,\bZ) + \lambda\norm{\wtheta^{\lambda,\bZ}}_1 & \text{recall the definition of $F_{\lambda,\bX}$ and $Q_{\bX}$} \\
    & \leq
    Q_{\bZ}(\lambda,\bZ) +  O(\eps'p\kappa(\Sigma)^2) + \lambda\norm{\wtheta^{\lambda,\bZ}}_1 & \text{by Lemma \ref{lem:pop_approx} and \ref{lem:theta_bound}} \\
    & \leq
    Q_{\bZ}(\lambda,\bX) +  O(\eps'p\kappa(\Sigma)^2) + \lambda\norm{\wtheta^{\lambda,\bX}}_1 & \text{$\wtheta^{\lambda,\bZ}$ is the minimum point}\\
    & \leq
    Q_{\bX}(\lambda,\bX) +  O(\eps'p\kappa(\Sigma)^2) + \lambda\norm{\wtheta^{\lambda,\bX}}_1 & \text{by Lemma \ref{lem:pop_approx} and \ref{lem:theta_bound}} \\
    & = 
    F_{\lambda,\bX}(\wtheta^{\lambda,\bX}) + O(\eps'p\kappa(\Sigma)^2) & \text{recall the definition of $F_{\lambda,\bX}$ and $Q_{\bX}$} \numberthis\label{eq:F_upper}
\end{align*}

Plugging \eqref{eq:F_upper} into \eqref{eq:F_lower} with $\theta = \wtheta^{\lambda,\bZ}$, we have
\begin{align*}
    \norm{\wtheta^{\lambda,\bZ} - \wtheta^{\lambda,\bX}}_1^2
    & \leq
    O(\eps'\cdot\frac{p^2\kappa(\Sigma)^2}{\svmin(\Sigma)})
\end{align*}
or 
\begin{align*}
    \abs{\norm{\wtheta^{\lambda,\bZ}}_1 - \norm{\wtheta^{\lambda,\bX}}_1}
    & \leq
    O(\sqrt{\eps'\cdot\frac{p^2\kappa(\Sigma)^2}{\svmin(\Sigma)}}).
\end{align*}

By Lemma \ref{lem:subgradient}, we have
\begin{align*}
    \frac{1}{n}\inner{\bX_p - \sum_{j=1}^{p-1}\wtheta^{\lambda,\bX}_j\bX_j}{\bX_i} =  \sign(\wtheta^{\lambda,\bX}_i)\lambda
\end{align*}
for $i\in [p-1]$ that $\wtheta^{\lambda,\bX}_i\neq 0$.
Hence,
\begin{align*}
    \lambda \leq M \cdot (1+\norm{\widehat{\theta}^{\lambda,\bX}}_1)
\end{align*}
where $M = \max\inner{\bX_i}{\bX_j}$ as long as $\lambda$ is not too large such as $\widehat{\theta}^{\lambda}\neq \bzero$.
By Lemma \ref{lem:theta_bound}, we have
\begin{align*}
    \lambda 
    & \leq
    O(M \cdot \sqrt{p}\kappa(\Sigma))
    \leq
    O(\sqrt{p}\svmax(\Sigma)\kappa(\Sigma)).
\end{align*}

\end{proof}

\subsection{Proof of Lemma \ref{lem:almost_tangent}}

\begin{lemma} \label{lem:almost_tangent}
Recall that 
\begin{itemize}
    \item $q_{[s]}$ is the $s$-dimensional vector whose $i$-th entry is $\sign(\wtheta^{\lambdao}_i)$ for $i\in[s]$
    \item $\wGamma_{[s]}$ is the $s$-by-$s$ matrix whose $(r,c)$-entry is $\frac{1}{n}\inner{\bX_r}{\bX_c}$ for $r,c\in[s]$
    \item $\Gamma_{[s]}$ is the $s$-by-$s$ submatrix of $\Gamma$ in \eqref{eq:sigma_def} whose indices are in $[s]$
    \item $\Dtheta = \wtheta^{\lambdao} - \theta^*$
\end{itemize}
Let $\mathcal{C}$ be the event of $\abs{q_{[s]}\wGamma_{[s]}^{-1} \Gamma_{[s]} \Dtheta_{[s]}}\leq \frac{1}{100\sqrt{s\sigma_{\max}(\Gamma)}}\sqrt{{\Dtheta_{[s]}}^\top\Gamma_{[s]}\Dtheta_{[s]}}$.
Then, we have
\begin{align*}
    \mathsf{Pr}_{\bX\sim \mathcal{N}(\bzero,\Sigma)^{n}}(\wne^{\lambdao}=[s] \wedge \mathcal{C})
    & \leq
    O\bigg(2^s\cdot \bigg( p^{-\Omega(\frac{1}{\kappa(\Gamma)})} + spe^{-\Omega(\frac{n}{s^2\kappa(\Gamma)^3})}\bigg) \bigg).
\end{align*}
\end{lemma}

\begin{proof}

From Lemma \ref{lem:subgradient}, we have
\begin{align*}
    \begin{cases}
        \text{if $i\in [s]$, then $G_{\bX,i}(\wtheta^{\lambdao})=\frac{1}{n}\inner{\bX_p - \sum_{j=1}^{p-1}\wtheta^{\lambdao}_j\bX_j}{\bX_i} = \sign(\wtheta^{\lambdao}_i)\lambdao$} \\
        \text{if $i\notin [s]$, then $G_{\bX,i}(\wtheta^{\lambdao})=\frac{1}{n}\inner{\bX_p - \sum_{j=1}^{p-1}\wtheta^{\lambdao}_j\bX_j}{\bX_i} = \lambdaio{i}$}
    \end{cases}
\end{align*}
where $\lambdaio{i}$ are some values whose absolute value is less than $\lambdao$, i.e. $\abs{\lambdaio{i}} \leq \lambdao$.
Let $\ww$ be the $(p-1)$-dimensional vector whose $i$-th entry is $\frac{1}{n}\inner{\bX_p - \sum_{j=1}^{p-1}\theta^*_j\bX_j}{\bX_i}$ for $i\in [p-1]$ and $z$ be the $(p-1)$-dimensional vector whose $i$-th entry is $\begin{cases}
    \sign(\wtheta^{\lambdao}_i) \lambdao &\text{for $i\in [s]$} \\
    \lambdaio{i} & \text{for $i\notin [s]$}
\end{cases}$.
We can write it in the matrix form.
\begin{align*}
    \ww-\widehat{\Gamma}\Dtheta = z \numberthis \label{eq:matrix_form_before}
\end{align*}

Recall that $\Gamma$ is a positive definite matrix which means $\Gamma$ can be decomposed as 
\begin{align*}
    \Gamma &= HH^\top & \text{for some matrix $H$.}
\end{align*}
We now multiple the both sides of \eqref{eq:matrix_form_before} by $H^{-1}$ and we have
\begin{align*}
    H^{-1}\ww - H^{-1}\wGamma \Dtheta = H^{-1}z. \numberthis \label{eq:matrix_form_before_2}
\end{align*}
for any $H$ that satisfies $\Gamma = HH^\top$.

Note that there are infinitely many such decomposition by introducing an orthonormal matrix, i.e.
\begin{align*}
    \Gamma &= HU(HU)^\top & \text{for some orthonormal matrix $U$.}
\end{align*}
In other words, we can always introduce an orthonormal matrix to ensure $H$ satisfies certain properties.
We now multiple both sides by $(HU)^{-1}$ and we have
\begin{align*}
    (HU)^{-1} \ww - (HU)^{-1}\wGamma\Dtheta = (HU)^{-1}z \numberthis\label{eq:matrix_form}
\end{align*}
There exists an orthonormal matrix $U$ such that
\begin{align*}
\begin{cases}
    \text{$((HU)^{-1})_{r,c} = 0$ for $r\in [s]$ and $c\notin [s]$}  \\
    \text{$((HU)^{-1} z)_i$ are positive and the same for $i\in [s]$.}\numberthis\label{eq:h_condition}
\end{cases}
\end{align*}
That is,
\begin{align*}
    (HU)^{-1}
    =
    \begin{bmatrix}
        \begin{matrix}
            & & \\
            & * & \\
            & &
        \end{matrix} & \begin{matrix}
            & & \\
            & O_{s \times (p-1-s)} & \\
            & &
        \end{matrix} \\
        \begin{matrix}
            & & \\
            & * & \\
            & &
        \end{matrix} & \begin{matrix}
            & & \\
            & * & \\
            & &
        \end{matrix}
    \end{bmatrix}
    \qquad \text{and} \qquad 
    (HU)^{-1}z
    = 
    \frac{1}{\sqrt{s}}\norm{((HU)^{-1}z)_{[s]}}_2
    \begin{bmatrix}
    1 \\
    \vdots \\
    1 \\
    * \\
    \vdots \\
    *
    \end{bmatrix}
\end{align*}
where $O_{s \times (p-1-s)}$ is the $s$-by-$(p-1-s)$ zero matrix and $((HU)^{-1}z)_{[s]}$ is the $s$-dimensional subvector of $(HU)^{-1}z$ whose indices are in $[s]$.
Note that $U$ depends on the samples by the second condition.
Moreover, by the first condition, $U$ only depends on $\sign(\wtheta^{\lambdao}_i)$ for $i\in [s]$.
Hence, there are $2^s$ possibilities and we will take union bound over all of them.

Let $H'$ be $HU$ such that $U$ satisfies \eqref{eq:h_condition}.
Observe that
\begin{align*}
    \norm{({H'}^{-1}z)_{-[s]}}_2^2
    & =
    \sum_{j=s+1}^{p-1} ({H'}^{-1}z)_j^2
\end{align*}
where $({H'}^{-1}z)_{-[s]}$ is the $(p-1-s)$-dimensional subvector of ${H'}^{-1}z$ whose indices are not in $[s]$.
Then,  there are at least $\frac{p-1-s}{2}$ of $({H'}^{-1}z)_j^2$ less than $\frac{2}{p-1-s}\norm{({H'}^{-1}z)_{-[s]}}_2^2$.
Also, we can bound the term $\norm{({H'}^{-1}z)_{-[s]}}_2$ by
\begin{align*}
    \norm{({H'}^{-1}z)_{-[s]}}_2
    & \leq
    \norm{{H'}^{-1}}_2\cdot \norm{z}_2 
     \leq
    \svmax({H'}^{-1}) \cdot \sqrt{p-1}\lambdao 
     =
    \sqrt{\frac{p-1}{\svmin(\Gamma)}}\lambdao.
\end{align*}
In other words, at least $\frac{p-1-s}{2}$ of $i\notin[s]$ such that
\begin{align*}
    ({H'}^{-1}z)_i 
    & \leq
    \sqrt{\frac{2(p-1)}{p-1-s}\frac{1}{\svmin(\Gamma)}} \lambdao
\end{align*}

On the other hand, all of $({H'}^{-1} z)_i$ are positive and the same for $i \in [s]$ and hence
\begin{align*}
    ({H'}^{-1} z)_i
    & =
    \frac{1}{\sqrt{s}}\norm{({H'}^{-1} z)_{[s]}}_2 
    =
    \frac{1}{\sqrt{s}}\norm{({H'}^{-1})_{[s]} z_{[s]}}_2 \qquad \text{by \eqref{eq:h_condition}}
    \\
    & \geq 
    \frac{1}{\sqrt{s}}\svmin({H'}^{-1})\norm{z_{[s]}}_2 \\
    & =
    \sqrt{\frac{1}{\svmax(\Gamma)}}\lambdao \numberthis\label{eq:hz_larger}
\end{align*}
where $(H'^{-1})_{[s]}$ is the $s$-by-$s$ submatrix of $H'^{-1}$ whose indices are in $[s]$ and $({H'}^{-1} z)_{[s]}$ (resp. $z_{[s]}$) is the $s$-dimensional subvector of  ${H'}^{-1} z$ (resp. $z$) whose indices are in $[s]$.

From \eqref{eq:matrix_form}, we have at least $\frac{p-1-s}{2}$ of $j\notin[s]$ such that, for all $i\in [s]$, 
\begin{align*}
    ({H'}^{-1} \ww - {H'}^{-1}\wGamma{H'}^{-\top}{H'}^\top\Dtheta)_i
    & \geq
    \eta \cdot ({H'}^{-1} \ww - {H'}^{-1}\wGamma{H'}^{-\top}{H'}^\top\Dtheta)_j \numberthis \label{eq:main_ineq_1}
\end{align*}
where $\eta = \sqrt{\frac{p-1-s}{2(p-1)}\frac{1}{\kappa(\Gamma)}}$ and $\kappa(\Gamma)$ is the condition number of $\Gamma$, i.e. $\kappa(\Gamma) = \frac{\svmax(\Gamma)}{\svmin(\Gamma)}$.

By matrix Chernoff bound, we have
\begin{align*}
    \norm{\wGamma_{[s]} - \Gamma_{[s]}}_2 < t\norm{\Gamma_{[s]}}_2 \numberthis\label{eq:mcb}
\end{align*}
with probability $1-O(e^{-\Omega(t^2n)})$ for any $t>0$.
Here, $\Gamma_{[s]}$ (resp. $\wGamma_{[s]}$) is the $s$-by-$s$ submatrix of $\Gamma$ (resp. $\wGamma$) whose indices are in $[s]$.
By Weyl's inequality, we further have
\begin{align*}
    \abs{\norm{\wGamma_{[s]}^{-1}}_2 - \norm{\Gamma_{[s]}^{-1}}_2}
    & =
    \abs{\frac{1}{\norm{\wGamma_{[s]}^{-1}}_2} - \frac{1}{\norm{\Gamma_{[s]}^{-1}}_2}}\cdot \norm{\wGamma_{[s]}^{-1}}_2\norm{\Gamma_{[s]}^{-1}}_2 \\
    & =
    \abs{\sigma_{\min}(\wGamma_{[s]}) - \sigma_{\min}(\Gamma_{[s]})}\cdot \norm{\wGamma_{[s]}^{-1}}_2\norm{\Gamma_{[s]}^{-1}}_2 \\
    & \leq
    \norm{\wGamma_{[s]} - \Gamma_{[s]}}_2\cdot \norm{\wGamma_{[s]}^{-1}}_2\norm{\Gamma_{[s]}^{-1}}_2
    \leq
    t\kappa(\Gamma) \norm{\wGamma_{[s]}^{-1}}_2.\numberthis\label{eq:weyl_ineq}
\end{align*}
which implies that
\begin{align*}
    \norm{\Gamma_{[s]}\wGamma_{[s]}^{-1} - I}_2
    & \leq
    \norm{\Gamma_{[s]} - \wGamma_{[s]}}_{2}\norm{\wGamma_{[s]}^{-1}}_2 
    \leq
    t\norm{\Gamma_{[s]}}_2\cdot \frac{1}{1-t\kappa(\Gamma)}\norm{\Gamma_{[s]}^{-1}}_2
    =
    \frac{t\kappa(\Gamma)}{1-t\kappa(\Gamma)}.
\end{align*}

Let $H'_{[s]}$ be the $s$-by-$s$ submatrix of $H'$ whose indices are in $[s]$ and $z_{[s]}$ be the $s$-dimensional subvector of $z$ whose indices are in $[s]$.
Now, we have
\begin{align*}
    \abs{({H_{[s]}'}^{-1}\Gamma_{[s]}\wGamma_{[s]}^{-1}z_{[s]})_i - ({H'_{[s]}}^{-1}z_{[s]})_i}
    & \leq
    \norm{{H_{[s]}'}^{-1}(\Gamma_{[s]}\wGamma_{[s]}^{-1}-I)z_{[s]}}_2 \\
    & \leq
    \norm{{H'_{[s]}}^{-1}}_2\norm{(\Gamma_{[s]}\wGamma_{[s]}^{-1}-I)}_2\norm{z_{[s]}}_2 \\
    & \leq
    \sqrt{\frac{1}{\svmin(\Gamma)}}\cdot\frac{t\kappa(\Gamma)}{1-t\kappa(\Gamma)}\cdot \sqrt{s}\lambdao.
\end{align*}
If we pick $t
    =
    \frac{1}{\kappa(\Gamma)(1+\sqrt{s\kappa(\Gamma)})}
    =
    \Theta(\frac{1}{\sqrt{s\kappa(\Gamma)^3}})
$ then we have
\begin{align*}
    \MoveEqLeft \abs{({H_{[s]}'}^{-1}\Gamma_{[s]}\wGamma_{[s]}^{-1}z_{[s]})_i - ({H'_{[s]}}^{-1}z_{[s]})_i}\\
    & \leq
    \frac{\lambdao}{50\sqrt{\sigma_{\max}(\Gamma)}} \\
    & \leq
    \frac{1}{50}({H_{[s]}'}^{-1} z_{[s]})_i \qquad \text{recall that, from \eqref{eq:hz_larger},  $({H_{[s]}'}^{-1} z_{[s]})_i = ({H'}^{-1} z)_i \geq \sqrt{\frac{1}{\svmax(\Gamma)}}\lambdao$ for $i\in [s]$} \numberthis\label{eq:mcb_ineq_h}
\end{align*}
with probability $1-O(e^{-\Omega(\frac{n}{s\kappa(\Gamma)^3})})$.

Now, we can analyze the event $\mathcal{C}$ which is $\abs{q_{[s]}\wGamma_{[s]}^{-1} \Gamma_{[s]} \Dtheta_{[s]}}\leq \frac{1}{100\sqrt{s\sigma_{\max}(\Gamma)}}\sqrt{{\Dtheta_{[s]}}^\top\Gamma_{[s]}\Dtheta_{[s]}}$.
Note that
\begin{align*}
    \lambdao \cdot q_{[s]}\wGamma_{[s]}^{-1} \Gamma_{[s]} \Dtheta_{[s]} 
    & =
    \lambdao \cdot q_{[s]}^\top\wGamma_{[s]}^{-1}\Gamma_{[s]}{H_{[s]}'}^{-\top}{H_{[s]}'}^{\top}\Dtheta_{[s]} \\
    & =
    \lambdao \cdot \sum_{j=1}^{s} ({H_{[s]}'}^{-1}\Gamma_{[s]}\wGamma_{[s]}^{-1}q_{[s]})_j({H_{[s]}'}^{\top}\Dtheta_{[s]})_j \\
    & \leq
    \sum_{j\in I_+} \frac{51}{50}({H_{[s]}'}^{-1} z_{[s]})_j({H_{[s]}'}^{\top}\Dtheta_{[s]})_j + \sum_{j\in I_-} \frac{49}{50}({H_{[s]}'}^{-1} z_{[s]})_j({H_{[s]}'}^{\top}\Dtheta_{[s]})_j 
\end{align*}
where $I_+$ (resp. $I_-$) is the subset of $[s]$ that $({H_{[s]}'}^{\top}\Dtheta_{[s]})_j$ is larger (resp. smaller) than $0$ for $j\in I_+$ (resp. $j\in I_-$).
Recall that $({H_{[s]}'}^{-1} z_{[s]})_j = \frac{1}{\sqrt{s}}\norm{({H'}^{-1})_{[s]} z_{[s]}}_2$ for all $j\in[s]$ and hence
\begin{align*}
    \lambdao \cdot q_{[s]}\wGamma_{[s]}^{-1} \Gamma_{[s]} \Dtheta_{[s]} 
    & \leq
    \frac{1}{\sqrt{s}}\norm{({H'}^{-1})_{[s]} z_{[s]}}_2\bigg(\sum_{j\in I_+} \frac{51}{50}({H_{[s]}'}^{\top}\Dtheta_{[s]})_j + \sum_{j\in I_-} \frac{49}{50}({H_{[s]}'}^{\top}\Dtheta_{[s]})_j\bigg)\numberthis\label{eq:i_0_upper}
\end{align*}

On the other hand, note that $\frac{\lambdao}{\sqrt{\sigma_{\max}(\Gamma)}} \leq \frac{1}{\sqrt{s}}\norm{({H'}^{-1})_{[s]} z_{[s]}}_2$ and $\sqrt{{\Dtheta_{[s]}}^\top\Gamma_{[s]}\Dtheta_{[s]}} = \norm{{H_{[s]}'}^{\top}\Dtheta_{[s]}}_2 \leq \sqrt{s}\bigg( \abs{({H_{[s]}'}^{\top}\Dtheta_{[s]})_{i_0}} + \abs{({H_{[s]}'}^{\top}\Dtheta_{[s]})_{i_1}} \bigg)$ where $i_0$ is the index such that $i_0 = \arg\max_{i\in I_+} \abs{({H_{[s]}'}^{\top}\Dtheta_{[s]})_i}$ (the largest positive value) and $i_1$ is the index such that $i_1 = \arg\max_{i\in I_-} \abs{({H_{[s]}'}^{\top}\Dtheta_{[s]})_i}$ (the largest negative value).
We have
\begin{align*}
    \lambdao \cdot q_{[s]}\wGamma_{[s]}^{-1} \Gamma_{[s]} \Dtheta_{[s]}
    & \geq 
    -\frac{\lambdao}{100\sqrt{s\sigma_{\max}(\Gamma)}}\sqrt{{\Dtheta_{[s]}}^\top\Gamma_{[s]}\Dtheta_{[s]}} \\
    & \geq
    -\frac{1}{100}\cdot \frac{1}{\sqrt{s}}\norm{({H'}^{-1})_{[s]} z_{[s]}}_2 \cdot \bigg( \abs{({H_{[s]}'}^{\top}\Dtheta_{[s]})_{i_0}} + \abs{({H_{[s]}'}^{\top}\Dtheta_{[s]})_{i_1}} \bigg).\numberthis\label{eq:i_0_lower}
\end{align*}

By comparing \eqref{eq:i_0_upper} and \eqref{eq:i_0_lower}, we have
\begin{align*}
    -\frac{1}{100}\bigg( \abs{({H_{[s]}'}^{\top}\Dtheta_{[s]})_{i_0}} + \abs{({H_{[s]}'}^{\top}\Dtheta_{[s]})_{i_1}} \bigg)
    & \leq
    \sum_{j\in I_+} \frac{51}{50}({H_{[s]}'}^{\top}\Dtheta_{[s]})_j + \sum_{j\in I_-} \frac{49}{50}({H_{[s]}'}^{\top}\Dtheta_{[s]})_j
\end{align*}
and it implies
\begin{align*}
    \sum_{j\in I_-} \abs{\frac{({H_{[s]}'}^{\top}\Dtheta_{[s]})_j}{({H_{[s]}'}^{\top}\Dtheta_{[s]})_{i_0}}}
    & <
    \frac{103}{97}s \numberthis \label{eq:htheta_ratio}
\end{align*}
where $i_0$ is the index such that $i_0 = \arg\max_i ({H_{[s]}'}^{\top}\Dtheta_{[s]})_i$.

Recall that $\Dtheta_i=0$ for $i\notin [s]$ and $(H'^{\top})_{r,c}=0$ for $r \notin[s]$ and $s\in [s]$.
Then, we have
\begin{align*}
    {H'}^{-1}\wGamma{H'}^{-\top}{H'}^\top\Dtheta
    & =
    ({H'}^{-1}\wGamma{H'}^{-\top})_{[p-1],[s]}H'^\top_{[s]}\theta^\Delta_{[s]}
\end{align*}
where $({H'}^{-1}\wGamma{H'}^{-\top})_{[p-1],[s]}$ is the $(p-1)$-by-$s$ submatrix of ${H'}^{-1}\wGamma{H'}^{-\top}$ whose row indices are in $[p-1]$ and column indices are in $[s]$.

Moreover, recall that
\begin{align*}
    \ww_i
    & =
    \frac{1}{n}\inner{\bX_p - \sum_{j=1}^{p-1}\theta^*_j\bX_j}{\bX_i}
    =
    \frac{1}{n}\bX^\top b
\end{align*}
and we can rewrite it as
\begin{align*}
    H^{-1}\ww
    & =
    \frac{1}{n}\sum_{j=1}^n b_jH^{-1}{\bX^{(j)}}^\top
\end{align*}
where $b$ is the $n$-dimensional vector $\bX_p - \sum_{j=1}^{p-1}\theta^*_j\bX_j$ and $\bX^{(i)}$ be the $i$-th row of $\bX$.
Note that $H^{-1}{\bX^{(i)}}^\top$ are distributed as $\mathcal{N}(\bzero,I)$.
Indeed, it is easy to check that 
\begin{align*}
    \E_{\bX^{(i)}\sim\mathcal{N}(\bzero,\Sigma)}\left(H^{-1}{\bX^{(i)}}^\top\bX^{(i)}H^{-\top}\right) 
    & = 
    I & \text{for any $i\in[n]$}. \numberthis\label{eq:independent}
\end{align*}

Recall that $\bX^{(i)}$ is the $i$-th row of $\bX$.
By \eqref{eq:independent}, $H'^{-1}{\bX^{(i)}}^\top$ are distributed as $\mathcal{N}(\bzero,I)$.
Consider the entries of $(H'^{-1}\wGamma H'^{-\top})_{r,c}$ for $r\in[p-1]$ and $c\in [s]$.
By Chernoff bound and union bound, for all $r\in [p-1]$ and $c\in [s]$, we have
\begin{align*}
    \abs{(H'^{-1}\widehat{\Gamma}H'^{-\top})_{r,c} - 1_{r=c}} < t \numberthis \label{eq:hgh_delta}
\end{align*}
with probability $1-O(spe^{-\Omega(t^2n)})$ for any $t>0$.
Here, $1_{r=c} = \begin{cases}
    1 & \text{if $r=c$} \\
    0 & \text{if $r\neq c$}.
\end{cases}$

Combining \eqref{eq:hgh_delta} and \eqref{eq:htheta_ratio}, there exists an index $i_0\in[s]$ such that, for all $i\notin[s]$, we have
\begin{align*}
    \MoveEqLeft ({H'}^{-1}\wGamma{H'}^{-\top}{H'}^\top\Dtheta)_{i_0} - \eta\cdot({H'}^{-1}\wGamma{H'}^{-\top}{H'}^\top\Dtheta)_i \\
    & =
    (({H'}^{-1}\wGamma{H'}^{-\top})_{[p-1],[s]}H'^\top_{[s]}\theta^\Delta_{[s]})_{i_0} - \eta\cdot(({H'}^{-1}\wGamma{H'}^{-\top})_{[p-1],[s]}H'^\top_{[s]}\theta^\Delta_{[s]})_i \\
    & =
    \sum_{j=1}^s ({H'}^{-1}\wGamma{H'}^{-\top})_{i_0,j}(H'^\top_{[s]}\theta^\Delta_{[s]})_j - \eta\cdot\sum_{j=1}^s ({H'}^{-1}\wGamma{H'}^{-\top})_{i,j}(H'^\top_{[s]}\theta^\Delta_{[s]})_j \\
    & =
    (H'^\top_{[s]}\theta^\Delta_{[s]})_{i_0}\bigg( \sum_{j=1}^s ({H'}^{-1}\wGamma{H'}^{-\top})_{i_0,j}\frac{(H'^\top_{[s]}\theta^\Delta_{[s]})_j}{(H'^\top_{[s]}\theta^\Delta_{[s]})_{i_0}} - \eta\cdot\sum_{j=1}^s ({H'}^{-1}\wGamma{H'}^{-\top})_{i,j}\frac{(H'^\top_{[s]}\theta^\Delta_{[s]})_j}{(H'^\top_{[s]}\theta^\Delta_{[s]})_{i_0}} \bigg) \\
    & \geq
    (H'^\top_{[s]}\theta^\Delta_{[s]})_{i_0} \bigg( (1-t) - (t(s-1) + t\frac{103}{97}s) - \eta \cdot (ts + t\frac{103}{97}s) \bigg) \\
    & \geq
    0 \numberthis\label{eq:i_0}
\end{align*}
with probability $1-O(spe^{-\Omega(\frac{n}{s^2})})$ if we pick $t=\frac{1}{1000s}$.

If we plug \eqref{eq:i_0} into \eqref{eq:main_ineq_1}, there exists an index $i_0\in[s]$ such that, for at least $\frac{p-1-s}{2}$ of $i\notin[s]$, we have 
\begin{align*}
    (H'^{-1}\ww)_{i_0}
    & \geq 
    \eta \cdot (H'^{-1}\ww)_{i}. \numberthis\label{eq:main_ineq_2}
\end{align*}
Recall that
\begin{align*}
    \ww_i
    & =
    \frac{1}{n}\inner{\bX_p - \sum_{j=1}^{p-1}\theta^*_j\bX_j}{\bX_i}
    \qquad \text{and} \qquad
    H'^{-1}\ww
    =
    \sum_{j=1}^n b_jH'^{-1}{\bX^{(j)}}^\top
\end{align*}
where $b$ is the $n$-dimensional vector $\bX_p - \sum_{j=1}^{p-1}\theta^*_j\bX_j$.
The entries of $b$ are distributed as $\mathcal{N}(0,a - v^\top\Gamma v)$ independently and independent to the entries of $H'^{-1}{\bX^{(j)}}^\top$ for $j\in [n]$.
If we further rewrite \eqref{eq:main_ineq_2} as
\begin{align*}
    \sum_{j=1}^n\frac{b_j}{\norm{b}_2}\frac{(H'^{-1}{\bX^{(j)}}^\top)_{i_0}}{\sqrt{n}}
    & \geq
    \eta \cdot\bigg(\sum_{j=1}^n\frac{b_j}{\norm{b}_2}\frac{(H'^{-1}{\bX^{(j)}}^\top)_{i}}{\sqrt{n}}\bigg).
\end{align*}
Hence, we can view the term $\frac{b}{\norm{b}_2}$ as a random projection and both side are just a Gaussian variable from $\mathcal{N}(0,1)$.
The probability of this event is
\begin{align*}
    \int_{-\infty}^{\infty} \left(\erf(\frac{x}{\eta})\right)^{\frac{p-1-s}{2}}\frac{1}{\sqrt{2\pi}}e^{-\frac{1}{2}x^2}\dir x
\end{align*}
where $\erf(*) = \int_{-\infty}^{*} \frac{1}{\sqrt{2\pi}} e^{-\frac{1}{2}x^2} \dir x$.
To bound this probability, we first see that 
\begin{align*}
    \erf(\frac{x}{\eta})
    & =
    \int_{-\infty}^{\frac{x}{\eta}} \frac{1}{\sqrt{2\pi}}e^{-\frac{1}{2}y^2} \dir y
    =
    1 - \int_{\frac{x}{\eta}}^{\infty} \frac{1}{\sqrt{2\pi}}e^{-\frac{1}{2}y^2} \dir y 
    <
    1 - \Omega(e^{-O((\frac{x}{\eta})^2)})
\end{align*}
Let $\xi$ be the value such that $\erf(\frac{\xi}{\eta}) < 1-\frac{1}{\sqrt{\frac{p-1-s}{2}}}$ which means $\xi = \Theta(\eta\sqrt{\log p})$.
Then, the probability can be further expressed as
\begin{align*}
    \MoveEqLeft\int_{-\infty}^{\infty} \left(\erf(\frac{x}{\eta})\right)^{\frac{p-1-s}{2}}\frac{1}{\sqrt{2\pi}}e^{-\frac{1}{2}x^2}\dir x \\
    & =
    \int_{-\infty}^{\xi} \left(\erf(\frac{x}{\eta})\right)^{\frac{p-1-s}{2}}\frac{1}{\sqrt{2\pi}}e^{-\frac{1}{2}x^2}\dir x + \int_{\xi}^{\infty} \left(\erf(\frac{x}{\eta})\right)^{\frac{p-1-s}{2}}\frac{1}{\sqrt{2\pi}}e^{-\frac{1}{2}x^2}\dir x.
\end{align*}
For the first term,
\begin{align*}
    \int_{-\infty}^{\xi} \left(\erf(\frac{x}{\eta})\right)^{\frac{p-1-s}{2}}\frac{1}{\sqrt{2\pi}}e^{-\frac{1}{2}x^2}\dir x
    & <
    \int_{-\infty}^{\xi} \left(1-\frac{1}{\sqrt{\frac{p-1-s}{2}}}\right)^{\frac{p-1-s}{2}}\frac{1}{\sqrt{2\pi}}e^{-\frac{1}{2}x^2}\dir x \\
    & <
    e^{-\frac{1}{\sqrt{\frac{p-1-s}{2}}}}.
\end{align*}
For the second term,
\begin{align*}
    \int_{\xi}^{\infty} \left(\erf(\frac{x}{\eta})\right)^{\frac{p-1-s}{2}}\frac{1}{\sqrt{2\pi}}e^{-\frac{1}{2}x^2}\dir x
    & <
    O(e^{-\Omega(\xi^2)})
    =
    O(p^{-\Omega(\eta^2)}).
\end{align*}
By combining these two terms and recalling that $\eta = \sqrt{\frac{p-1-s}{2(p-1)}\frac{1}{\kappa(\Gamma)}}$, we have
\begin{align*}
    \int_{-\infty}^{\infty} \left(\erf(\frac{x}{\eta})\right)^{\frac{p-1-s}{2}}\frac{1}{\sqrt{2\pi}}e^{-\frac{1}{2}x^2}\dir x
    & \leq
    e^{-\frac{1}{\sqrt{\frac{p-1-s}{2}}}} + O(p^{-\Omega(\eta^2)})
    \leq
    O(p^{\Omega(-\frac{1}{\kappa(\Gamma)})}). \numberthis\label{eq:key_prob}
\end{align*}

Finally, combining the failure probabilities in \eqref{eq:mcb_ineq_h}, \eqref{eq:i_0} and \eqref{eq:key_prob} and taking union bound over all $2^s$ choices of $H'$ for satisfying \eqref{eq:h_condition}, we have
\begin{align*}
    \mathsf{Pr}_{\bX\sim \mathcal{N}(\bzero,\Sigma)^{n}}(\wne^{\lambdao}=[s] \wedge \mathcal{C})
    & \leq
    O\bigg(2^s\cdot \bigg( p^{-\Omega(\frac{1}{\kappa(\Gamma)})} + spe^{-\Omega(\frac{n}{s^2})} + e^{-\Omega(\frac{n}{s\kappa(\Gamma)^3})} \bigg) \bigg)\\
    & \leq
    O\bigg(2^s\cdot \bigg( p^{-\Omega(\frac{1}{\kappa(\Gamma)})} + spe^{-\Omega(\frac{n}{s^2\kappa(\Gamma)^3})}\bigg) \bigg).
\end{align*}

\end{proof}

\subsection{Proof of Lemma \ref{lem:rays_out}}

\begin{lemma}\label{lem:rays_out}

Recall that 
\begin{itemize}
    \item the event $\mathcal{A}$ is ${\sign(Q)\theta'}^\top\Gamma(\wtheta^{\lambdao} - \theta^*) \geq 0$
    \item the event $\mathcal{B}$ is ${-\sign(Q) \theta''}^\top\Gamma(\wtheta^{\lambdao} - \theta^*) \geq 0$
\end{itemize}
where $Q$, $\theta'$ and $\theta''$ are defined in \eqref{eq:q_def}, \eqref{eq:theta_prime} and \eqref{eq:theta_dprime} respectively.
Then, we have
\begin{align*}
    \mathsf{Pr}_{\bX\sim \mathcal{N}(\bzero,\Sigma)^{n}}(\wne^{\lambdao} = [s] \wedge \mathcal{A}\wedge\mathcal{B})
    & <
    O\bigg(2^s\cdot \bigg( p^{-\Omega(\frac{1}{\kappa(\Gamma)})} + spe^{-\Omega(\frac{n}{s^2\kappa(\Gamma)^6})}\bigg) \bigg).
\end{align*}
    
\end{lemma}

\begin{proof}

Recall that  $\Dtheta=\wtheta^{\lambdao} - \theta^*$.
We first expand $\theta''^\top \Gamma\Dtheta$.
By the fact $\Dtheta_i=0$ for $i\notin[s]$ from $\ne^*=\wne^{\lambdao}=[s]$ and the definition of $\theta''$ in \eqref{eq:theta_dprime}, we have
\begin{align*}
    \theta''^\top\Gamma\Dtheta
    & =
    -q_{[s+1]}^\top\wGamma_{[s+1]}^{-1}\Gamma_{[s+1],[s]}\Dtheta_{[s]}
\end{align*}
where $\Gamma_{[s+1],[s]}$ is the $(s+1)$-by-$s$ sub-matrix of $\Gamma$ whose row indices are in $[s+1]$ and column indices are in $[s]$.
By direct block matrix calculation, we have
\begin{align*}
    \MoveEqLeft q_{[s+1]}^\top\wGamma_{[s+1]}^{-1}\Gamma_{[s+1],[s]} \\
    & =
    q_{[s]}^\top\wGamma_{[s]}^{-1}\Gamma_{[s]} + \frac{q_{[s]}^\top\wGamma_{[s]}^{-1}\wu}{\wGamma_{s+1} - \wu^\top\wGamma_{[s]}^{-1}\wu}\wu^\top\wGamma_{[s]}^{-1}\Gamma_{[s]} - \frac{q_{s+1}}{\wGamma_{s+1} - \wu^\top\wGamma_{[s]}^{-1}\wu}\wu^\top\wGamma_{[s]}^{-1}\Gamma_{[s]} \\
    & \qquad -\frac{q_{[s]}^\top\wGamma_{[s]}^{-1}\wu}{\wGamma_{s+1} - \wu^\top\wGamma_{[s]}^{-1}\wu}u^\top + \frac{q_{s+1}}{\wGamma_{s+1} - \wu^\top\wGamma_{[s]}^{-1}\wu}u^\top \\
    & =
    q_{[s]}^\top\wGamma_{[s]}^{-1}\Gamma_{[s]} - \frac{q_{[s]}^\top\wGamma_{[s]}^{-1}\wu - q_{s+1}}{\wGamma_{s+1} - \wu^\top\wGamma_{[s]}^{-1}\wu}(u^\top - \wu^\top\wGamma_{[s]}^{-1}\Gamma_{[s]})
\end{align*}
where $\wu$ (resp. $u$) is the $s$-dimensional vector whose $i$-th entry is $\wGamma_{i,s+1}$ (resp. $\Gamma_{i,s+1}$) for $i\in[s]$.
In other words, we have
\begin{align*}
    -\sign(Q)\theta''^\top \Gamma \Dtheta
    & =
    \sign(Q)q_{[s]}^\top\wGamma_{[s]}^{-1}\Gamma_{[s]}\theta_{[s]}^{\Delta} - \sign(Q)\frac{q_{[s]}^\top\wGamma_{[s]}^{-1}\wu - q_{s+1}}{\wGamma_{s+1} - \wu^\top\wGamma_{[s]}^{-1}\wu}(u^\top - \wu^\top\wGamma_{[s]}^{-1}\Gamma_{[s]})\theta_{[s]}^{\Delta}.\numberthis\label{eq:b_almost_tangent}
\end{align*}
By the fact $\Dtheta_i=0$ for $i\notin [s]$ from $\ne^*=\wne^{\lambdao} =[s]$ and the definition of $\theta'$ in \eqref{eq:theta_prime}, we have
\begin{align*}
    \theta'^\top\Gamma\Dtheta
    & =
    -q_{[s]}^\top\wGamma_{[s]}^{-1}\Gamma_{[s]}\Dtheta_{[s]}.
\end{align*}
We plug it into \eqref{eq:b_almost_tangent} and use the assumption that $-\sign(Q)\theta''^\top\Gamma\Dtheta\geq 0$.
Then, we have
\begin{align*}
    \sign(Q)\theta'^\top\Gamma\Dtheta
    & \leq
     - \sign(Q)\frac{q_{[s]}^\top\wGamma_{[s]}^{-1}\wu - q_{s+1}}{\wGamma_{s+1} - \wu^\top\wGamma_{[s]}^{-1}\wu}(u^\top - \wu^\top\wGamma_{[s]}^{-1}\Gamma_{[s]})\theta_{[s]}^{\Delta}. \numberthis \label{eq:b_almost_tangent_2}
\end{align*}
We now further bound the terms in the RHS.

By matrix Chernoff bound, we first have
\begin{align*}
    \norm{\wGamma_{[s+1]} - \Gamma_{[s+1]}}_2 < t\norm{\Gamma_{[s+1]}}_2 
\end{align*}
with probability $1-O(e^{-\Omega(t^2n)})$ for any $t>0$.
Here, $\Gamma_{[s+1]}$ (resp. $\wGamma_{[s+1]}$) is the $(s+1)$-by-$(s+1)$ submatrix of $\Gamma$ (resp. $\wGamma$) whose indices are in $[s+1]$.
By the similar argument as in \eqref{eq:weyl_ineq}, it implies
\begin{align*}
    \abs{\norm{\wGamma_{[s]}^{-1}}_2 - \norm{\Gamma_{[s]}^{-1}}_2}
    & \leq
    t\kappa(\Gamma) \norm{\wGamma_{[s]}^{-1}}_2, \\
    \abs{\norm{\wGamma_{[s+1]}^{-1}}_2 - \norm{\Gamma_{[s+1]}^{-1}}_2},
    & \leq
    t\kappa(\Gamma) \norm{\wGamma_{[s+1]}^{-1}}_2 \\
    \norm{\wGamma_{[s]}^{-1} - \Gamma_{[s]}^{-1}}_2
    & \leq
    t\kappa(\Gamma)\norm{\wGamma_{[s]}^{-1}}_2.
\end{align*}

For the term $q_{[s]}^\top\wGamma_{[s]}^{-1}\wu - q_{s+1}$, we have
\begin{align*}
    \abs{q_{[s]}^\top\wGamma_{[s]}^{-1}\wu - q_{s+1}}
    & \leq
    \norm{q_{[s]}}_2 \norm{\wGamma^{-1}_{[s]}}_2 \norm{\wu}_2 + 1
\end{align*}
Recall that the entries of $q_{[s]}$ has absolute values $1$ and hence $\norm{q_{[s]}}_2 = \sqrt{s}$.
The term $\norm{\wGamma^{-1}_{[s]}}_2$ is bounded by $\frac{1}{(1-t\kappa(\Gamma))\sigma_{\min}(\Gamma)}$ and the term $\norm{\wu}_s$ is bounded by $\norm{\wGamma_{[s+1]}}_2 \leq t\sigma_{\max}(\Gamma)$.
We have
\begin{align*}
    \abs{q_{[s]}^\top\wGamma_{[s]}^{-1}\wu - q_{s+1}}
    & \leq
    \sqrt{s}\cdot \frac{1}{(1-t\kappa(\Gamma))\sigma_{\min}(\Gamma)}\cdot t\sigma_{\max}(\Gamma) + 1
    =
    \frac{\sqrt{s}t\kappa(\Gamma)}{1-t\kappa(\Gamma)} + 1\numberthis\label{eq:b_tangent_part_1}
\end{align*}

For the term $\wGamma_{s+1} - \wu^\top\wGamma_{[s]}\wu$, we have
\begin{align*}
    \abs{\wGamma_{s+1} - \wu^\top\wGamma_{[s]}^{-1}\wu}
    & \geq
    \sigma_{\min}(\wGamma_{[s+1]}) 
    \geq
    (1-t\kappa(\Gamma))\sigma_{\min}(\Gamma).\numberthis\label{eq:b_tangent_part_2}
\end{align*}

For the term $(u^\top - \wu^\top\wGamma_{[s]}^{-1}\Gamma_{[s]})\theta_{[s]}^{\Delta}$, we further expand it as
\begin{align*}
    (u^\top - \wu^\top\wGamma_{[s]}^{-1}\Gamma_{[s]})\theta_{[s]}^{\Delta}
    & =
    u^\top (\Gamma_{[s]}^{-\top} - \wGamma_{[s]}^{-\top})H'_{[s]}{H'_{[s]}}^{\top}\Dtheta_{[s]} + (u^\top - \wu^\top)\wGamma_{[s]}^{-1}H'_{[s]}{H'_{[s]}}^{\top}\Dtheta_{[s]}
\end{align*}
where $H'$ is the matrix satisfying $\Gamma = H'{H'}^\top$ and \eqref{eq:h_condition} and $H'_{[s]}$ is its $s$-by-$s$ submatrix whose indices are in $[s]$.
For the first term, we have
\begin{align*}
    \abs{u^\top (\Gamma_{[s]}^{-\top} - \wGamma_{[s]}^{-\top})H'_{[s]}{H'_{[s]}}^{\top}\Dtheta_{[s]}}
    & \leq
    \norm{u}_2\cdot \norm{\Gamma_{[s]}^{-1} - \wGamma_{[s]}^{-1}}_2 \cdot \norm{H'_{[s]}}_2\cdot \norm{{H'_{[s]}}^{\top}\Dtheta_{[s]}}_2 \\
    & \leq
    \sigma_{\max}(\Gamma)\cdot \frac{t\kappa(\Gamma)}{\sigma_{\min}(\Gamma)} \cdot \sqrt{\sigma_{\max}(\Gamma)} \cdot \norm{{H'_{[s]}}^{\top}\Dtheta_{[s]}}_2 \\
    & =
    t\kappa(\Gamma)^2\sqrt{\sigma_{\max}(\Gamma)} \cdot \norm{{H'_{[s]}}^{\top}\Dtheta_{[s]}}_2.
\end{align*}
For the second term, we have
\begin{align*}
    \abs{(u^\top - \wu^\top)\wGamma_{[s]}^{-1}H'_{[s]}{H'_{[s]}}^{\top}\Dtheta_{[s]}}
    & \leq
    \norm{u - \wu}_2\cdot \norm{\wGamma_{[s]}^{-1}}_2 \cdot \norm{H'_{[s]}}_2\cdot \norm{{H'_{[s]}}^{\top}\Dtheta_{[s]}}_2 \\
    & \leq
    t\sigma_{\max}(\Gamma)\cdot \frac{1}{(1-t\kappa(\Gamma))\sigma_{\min}(\Gamma)} \cdot \sqrt{\sigma_{\max}(\Gamma)} \cdot \norm{{H'_{[s]}}^{\top}\Dtheta_{[s]}}_2 \\ 
    & =
    \frac{t\kappa(\Gamma)\sqrt{\sigma_{\max}(\Gamma)}}{1-t\kappa(\Gamma)} \cdot \norm{{H'_{[s]}}^{\top}\Dtheta_{[s]}}_2.
\end{align*}
It means that
\begin{align*}
    \abs{(u^\top - \wu^\top\wGamma_{[s]}^{-1}\Gamma_{[s]})\theta_{[s]}^{\Delta}}
    & \leq
    \bigg(\kappa(\Gamma) + \frac{1}{1-t\kappa(\Gamma)} \bigg)t\kappa(\Gamma)\sqrt{\sigma_{\max}(\Gamma)}\norm{{H'_{[s]}}^{\top}\Dtheta_{[s]}}_2 \numberthis\label{eq:b_tangent_part_3}
\end{align*}

Plugging \eqref{eq:b_tangent_part_1}, \eqref{eq:b_tangent_part_2} and \eqref{eq:b_tangent_part_3} into the RHS of \eqref{eq:b_almost_tangent_2}, we have
\begin{align*}
    \sign(Q)\theta'^\top\Gamma\Dtheta
    & \leq
    \abs{\frac{q_{[s]}^\top\wGamma_{[s]}^{-1}\wu - q_{s+1}}{\wGamma_{s+1} - \wu^\top\wGamma_{[s]}^{-1}\wu}(u^\top - \wu^\top\wGamma_{[s]}^{-1}\Gamma_{[s]})\theta_{[s]}^{\Delta} } \\
    & \leq
    \frac{\bigg(\frac{\sqrt{s}t\kappa(\Gamma)}{1-t\kappa(\Gamma)} + 1\bigg)\bigg(\kappa(\Gamma) + \frac{1}{1-t\kappa(\Gamma)} \bigg)t\kappa(\Gamma)\sqrt{\sigma_{\max}(\Gamma)}}{(1-t\kappa(\Gamma))\sigma_{\min}(\Gamma)}\norm{{H'_{[s]}}^{\top}\Dtheta_{[s]}}_2 \\
    & =
    \frac{\bigg(\frac{\sqrt{s}t\kappa(\Gamma)}{1-t\kappa(\Gamma)} + 1\bigg)\bigg(\kappa(\Gamma) + \frac{1}{1-t\kappa(\Gamma)} \bigg)t\kappa(\Gamma)^2}{(1-t\kappa(\Gamma))\sqrt{\sigma_{\max}(\Gamma)}}\norm{{H'_{[s]}}^{\top}\Dtheta_{[s]}}_2
\end{align*}

If we pick $t  = \frac{1}{1000\sqrt{s}\kappa(\Gamma)^3}= \Theta(\frac{1}{\sqrt{s}\kappa(\Gamma)^3})$, then we have $1-t\kappa(\Gamma) = 1-\frac{1}{1000\sqrt{s}\kappa(\Gamma)^2}\geq \frac{999}{1000}$, $\sqrt{s}t\kappa(\Gamma) = \frac{1}{1000\kappa(\Gamma)^2} \leq \frac{1}{1000}$ and $t\kappa(\Gamma)^2 = \frac{1}{1000\sqrt{s}\kappa(\Gamma)}$.
It means that we have
\begin{align*}
    \frac{\sqrt{s}t\kappa(\Gamma)}{1-t\kappa(\Gamma)} + 1
    & \leq
    \frac{\frac{1}{1000}}{\frac{999}{1000}} + 1
    =
    \frac{1000}{999} \\
    \kappa(\Gamma) + \frac{1}{1-t\kappa(\Gamma)} 
    & \leq
    \kappa(\Gamma) + \frac{1000}{999}
    \leq
    \frac{1999}{999}\kappa(\Gamma) \\
    \frac{t\kappa(\Gamma)^2}{1-t\kappa(\Gamma)}
    & \leq
    \frac{\frac{1}{1000\sqrt{s}\kappa(\Gamma)}}{\frac{999}{1000}}
    =
    \frac{1}{999\sqrt{s}\kappa(\Gamma)}
\end{align*}
which implies 
\begin{align*}
    \sign(Q)\theta'^\top\Gamma\Dtheta
    & \leq
    \frac{1}{100\sqrt{s\sigma_{\max}(\Gamma)}}\norm{{H'_{[s]}}^{\top}\Dtheta_{[s]}}_2
    =
    \frac{1}{100\sqrt{s\sigma_{\max}(\Gamma)}}\sqrt{{\Dtheta_{[s]}}^\top\Gamma_{[s]}\Dtheta_{[s]}} \numberthis\label{eq:almost_tangent_cond}
\end{align*}
with probability $1-O(e^{-\Omega(\frac{n}{s\kappa(\Gamma)^6})})$.

From the event of $\sign(Q)\theta'^\top\Gamma\Dtheta \geq 0$, it implies $\sign(Q)\theta'^\top\Gamma\Dtheta = \abs{\theta'^\top\Gamma\Dtheta}$ and we have
\begin{align*}
    \abs{\theta'^\top\Gamma\Dtheta}
    & \leq
    \frac{1}{100\sqrt{s\sigma_{\max}(\Gamma)}}\sqrt{{\Dtheta_{[s]}}^\top\Gamma_{[s]}\Dtheta_{[s]}}. \numberthis\label{eq:almost_tangent_prob}
\end{align*}
By Lemma \ref{lem:almost_tangent}, the probability of \eqref{eq:almost_tangent_prob} is less than 
\begin{align*}
    O\bigg(2^s\cdot \bigg( p^{-\Omega(\frac{1}{\kappa(\Gamma)})} + spe^{-\Omega(\frac{n}{s^2\kappa(\Gamma)^3})}\bigg) \bigg)
\end{align*}
Combining the failure probability of \eqref{eq:almost_tangent_cond}, the probability of $\wne^{\lambdao} = [s] \wedge \mathcal{A}\wedge\mathcal{B}$ is bounded by 
\begin{align*}
    O\bigg(2^s\cdot \bigg( p^{-\Omega(\frac{1}{\kappa(\Gamma)})} + spe^{-\Omega(\frac{n}{s^2\kappa(\Gamma)^3})} + e^{-\Omega(\frac{n}{s\kappa(\Gamma)^6})}\bigg) \bigg)
    & <
    O\bigg(2^s\cdot \bigg( p^{-\Omega(\frac{1}{\kappa(\Gamma)})} + spe^{-\Omega(\frac{n}{s^2\kappa(\Gamma)^6})}\bigg) \bigg).
\end{align*}

\end{proof}

\section{Experiment details}\label{Appendix Experiement}

\subsection{Set-up}\label{Expt setup}
\paragraph{Computing}
All experiments were conducted on an 8-core Intel Xeon processor E5-2680v4 with
2.40 GHz frequency, and 16GB of memory. All experiments were set three hours wall time limit. Exceeding time limit or undefined FDR value for all zero estimates are marked as NULL for data recording and as missing points for plotting.

\paragraph{Graph models}
We include four common graphs: the Band graph, Scale-Free (SF), Erdös-Rényi (ER) and K-Nearest Neighbor (KNN) graphs for GGM simulation. Details for non-Gaussian simulations refer to~\ref{Non-Gaussian simu}.
The Band and SF graphs were generated via flare~\cite{flare} package, and the rest ER and KNN were implemented via i-graph~\cite{igraph} and mstknnclust~\cite{knn} in R. Specifically, graphs are initialized as following
\begin{itemize}
    \item Band graphs: These are single chain graph given the number of observations $n$ and the number of variables $p$, and $u,v,g$ are set as default in the r-flare package.
    \item Scale-free (SF) graphs: These are generated by Barabási–Albert model; The graph is initialized with two connected nodes, and the probability of a new node connecting to one existing node is proportional to the degree of the existing node.
    \item Erdös-Rény (ER) graphs: These are random undirected graph with $n$ number of nodes, and $p-1$ number of edges, which are selected uniformly and randomly from the set of possible edges. 
    \item K-nearest neighbor (KNN) graphs: There are random graphs where nodes are connected only if they are one of the $k$-nearest neighbors based on corresponding distance between them. A uniformly distributed $p\times p$ matrix is generated as some random data to calculate euclidean distances between nodes.
\end{itemize}

\paragraph{Methods}
For each graph model, each of the following method was implemented provided a penalty parameter.
\begin{itemize}
    \item Neighbourhood selection (NS): The neighbourhood selection method in~\cite{meinshausen2006} was implemented, and code is available at:~\url{https://anonymous.link}.
    \item Graphical Lasso (Glasso) in~\cite{friedman2008sparse} was implemented based on the glasso R package~\cite{glasso}, and code is available at:~\url{https://github.com/cran/glasso}. 
    \item Constrained $\ell_1$-minimization for inverse matrix estimation (CLIME) in ~\cite{cai2011constrained} was implemented based on the flare R package~\cite{flare}, and mirror code is available at:~\url{https://github.com/cran/clime}.
    \item Tuning-Insensitive Graph Estimation and Regression (TIGER) in ~\cite{liu2017tiger} was implemented based on flare R package~\cite{flare}, and mirror code is available at:~\url{https://github.com/cran/tiger}.
\end{itemize}

\paragraph{Metrics}
For each graph model, we evaluate the performance of each method with the penalty parameter selected by the following criteria: Akaike information criterion \citep[AIC,][]{akaike1974new}, Bayesian information criterion \citep[BIC,][]{schwarz1978estimating}, extended Bayesian information criterion \citep[EBIC,][]{foygel2010extended}, and 5-fold cross-validation (CV). The performance was measured by the following metrics:
\begin{itemize}
    \item Structured Hamming distance (SHD): The number of edge insertions, deletions or flips (in directed graph) that is needed to transform the estimated graph to the true graph.
    \item True Positive Rate (TPR):The proportion of correctly identified edges to the total number of edges in the true graph.
    \item False Discovery Rate (FDR): The proportion of incorrectly identified edges to the total number of edges in the estimated graph.
\end{itemize}

\paragraph{Remarks}
Experiment jobs were auto-written given configs (the graph type, the method, $n$ and $p$) and submitted via slurm. Experiments and analysis were conducted using R (\cite{R}).

\subsection{Non-Gaussian simulation}\label{Non-Gaussian simu}
\begin{figure*}[htbp]
  \centering
  \includegraphics[width=\textwidth]{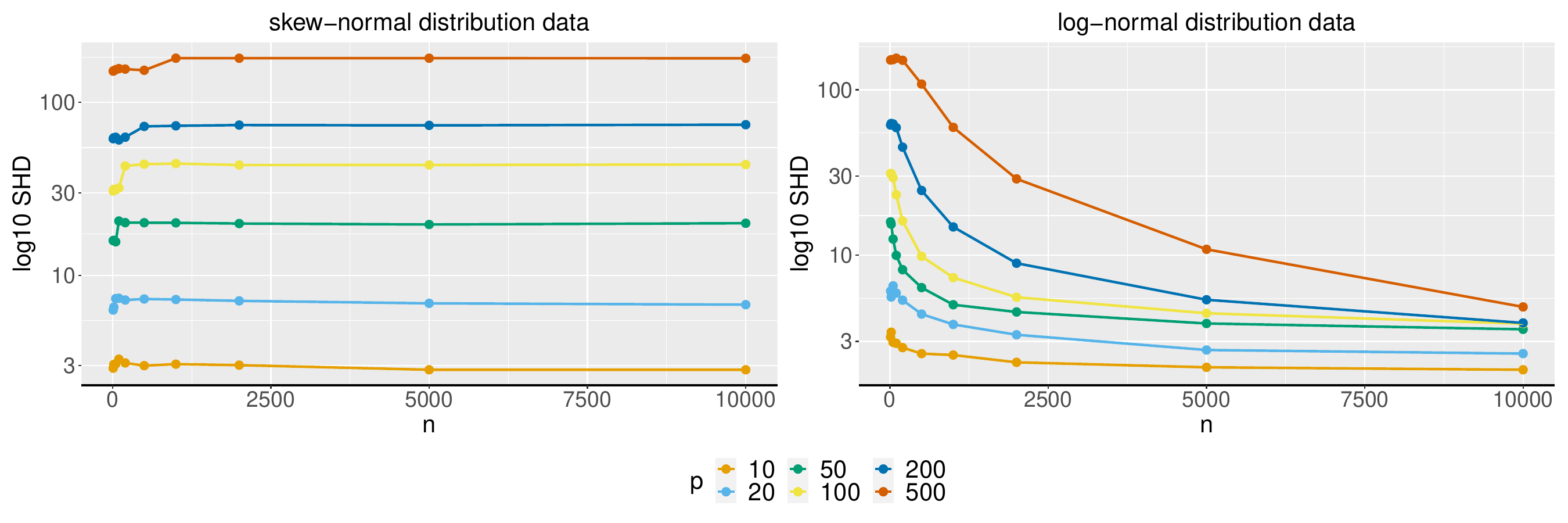}
  \caption{Log-SHD vs. sample size $n$ and the number of variables $p$ on Skew-normal and log-normal data via CV. The bottom black line denotes zero SHD, which is never attained.}
  \label{fig: fig3 non Gaussian exp}
\end{figure*}

Although our theoretical results are specific to the Gaussian setting, we can also demonstrate the fallibility of CV for non-Gaussian data.
Specifically, we randomly generate $n$ i.i.d samples $\mathbf{X}\in\mathbb{R}^{n\times p}$ from skew-normal and log-normal distributions via sn~\cite{sknormal} and MASS packages~\cite{ non-gauss} in R, respectively. The true coefficients $\theta_i\sim \text{Unif}([-2,-1]\cup[1,2])$ if it's in the neighborhood, otherwise, we set it to zero. The response variable is set by $\mathbf{Y} = \mathbf{X}\theta +\eps$, with $\eps\sim N(0,0.1)$. We implement glmnet~\cite{glmnet} package for R to test with $5$-fold CV to select the penalty parameter and obtain the estimates. We repeat this 100 times, and take the average SHD for each $n$ and $p$.
Although Figure~\ref{fig: fig3 non Gaussian exp} shows a decreasing trend in SHD with increasing $n$ for all $p$, CV plateaus and never achieves perfect selection (highlighted with the black line for zero SHD). Even when $p=10$ and $n=10000$, CV fails to correctly select all neighbors. This once again indicates that CV is suboptimal for structure learning.

\subsection{GGM simulation}
For each graph type, we construct the covariance and precision matrices $\Sigma^*$ and $K^*$, and simulate data generated from $N(0,\Sigma^*)$. We choose the largest penalty parameter $\lambda_{\mathrm{max}}$ (the smallest value which will result in a null, no-edge model) and the smallest penalty parameter $\lambda_{\mathrm{min}}$ (the largest value which will result in a model whose number of edge is less than $2\norm{K^*}_1$). $100$ penalty parameters $\lambda$ are chosen logarithmically evenly spaced between $\lambda_{\mathrm{max}}$ and $\lambda_{\mathrm{min}}$. At each penalty $\lambda$, an estimate $\widehat{K}_\lambda$ is computed via a given algorithm (NS, Glasso, CLIME or TIGER), and is used to model the edge set $E_{\lambda}$. We next compute the unpenalized maximum likelihood estimate with the same support as $E_{\lambda}$. We can compute all criteria and choose $\lambda$s corresponding to the lowest criteria among all $\lambda$s. This is repeated for 10 trials for each combination of $n,~p$, graph type and algorithms.

We perform $5$-fold CV with generated data. For each $\lambda$ and each training set and test set, we fit the model with training set using glasso, and evaluate the performance via samples in the test set by calculating the Gaussian log-likelihood. The penalty value with the maximum averaged Gaussian log-likelihood is selected as $\lambdacv$. We measure its performance via the average SHD, TPR and FDR to compare with other criteria.

\subsection{Additional results}
As noted in Corollary \ref{col:asym_thm}, the sparsity level $s$ also plays a role in determining the probability of correct recovery of the neighborhood. To emprically illustrate this point, we set $p=50,500$ and simulated $5000$ data from a Gaussian linear model. The sparsity $s$ is enforced by randomly setting $s$ out of $p$ coefficients to be exactly 0. The result shown in Figure~\ref{fig: SHD vs sparsity} indeed confirms our theoretical result: a larger $s$ implies a lower chance of recovery.
\begin{figure}
    \centering
    \includegraphics[width=0.7\textwidth]{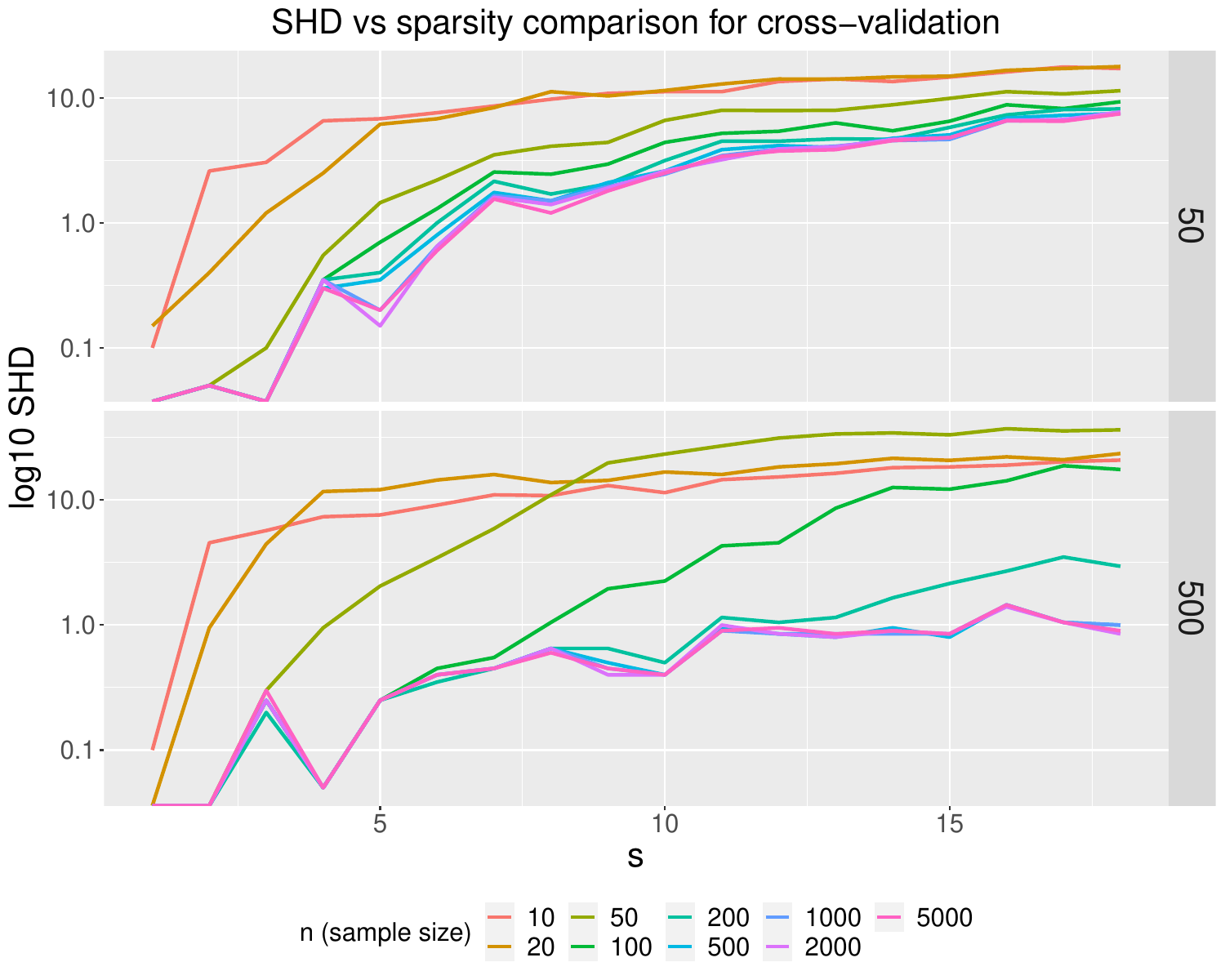}
    \caption{SHD vs sparsity comparison for CV for $p=50$ and $50$ respectively for varying $n$.}
    \label{fig: SHD vs sparsity}
\end{figure}

Below we provide remaining experimental results for 10 runs when utilizing NS, Glasso, CLIME, and TIGER; see Figures~\ref{fig: fig all in one NS},~\ref{fig: fig all in one Glasso},~\ref{fig: fig all in one CLIME},~\ref{fig: fig all in one TIGER} respectively. For each run, a random seed is set to generate varying datasets. We also provide detailed numerical results of FDR in Table~\ref{tab:NS CV table} and average SHD in Table~\ref{tab:NS CV table2} for NS method tuned by CV for the Band graph, which clearly suggests CV neither reaches $0\%$ FDR nor obtain fully correct graph (0 SHD).

To further investigate the behaviour of CV for large $n$, and to verify that it indeed fails to achieve exact recovery (i.e. zero average SHD), we provide additional experimental results with larger $n$ over 100 runs in Figures~\ref{fig: fig all in one NS 100 runs},~\ref{fig: fig all in one Glasso 100 runs},~\ref{fig: fig all in one CLIME 100 runs} and~\ref{fig: fig all in one TIGER 100 runs}. 

\begin{table*}[hb]
\centering
\begin{tabular}{lrrrrrrrr}
\toprule
 &  $n=10$   &  $n=20$   &  $n=50$   &  $n=100$  &  $n=200$  &  $n=500$  &  $n=800$  &  $n=1000$ \\
\midrule
$p=100$ &  0.587 &  0.369 &  0.0782 &  0.0329 &  0.0289 &  0.0158 &  0.00796 &  0.0119 \\
$p=200$ &  0.558 &  0.327 &  0.0738 &  0.0256 &  0.0148 &  0.0099 &  0.00792 &  0.0089 \\
$p=500$ &  0.674 &  0.333 &  0.0602 &  0.0176 &  0.0068 &  0.0052 &  0.00597 &  0.0051 \\
\bottomrule
\end{tabular}
\caption{Average FDR for the Band graph via NS method tuned by CV}
\label{tab:NS CV table}
\end{table*}

\begin{table*}[hb]
\centering
\begin{tabular}{lrrrrrrrr}
\toprule
 &  n=10   &  n=20   &  n=50   &  n=100  &  n=200  &  n=500  &  n=800  &  n=1000 \\
\midrule
p=100 &  103.2 &  81.8 &  20.2 &  3.4 &  3.0 &  1.6 &  0.8 &  1.2 \\
p=200 &  202.8 &  168.4 &  47.2 &  6.8 &  3.0 &  2.0 &  1.6 &  1.8 \\
p=500 &  510.0 &  454.2 &  151.4 &  19.2 &  3.4 &  2.6 &  3.0 &  2.6 \\
\bottomrule
\end{tabular}
\caption{Average SHD for the Band graph via NS method tuned by CV}
\label{tab:NS CV table2}
\end{table*}

\begin{figure*}[htbp]
  \centering \includegraphics[width=\textwidth]{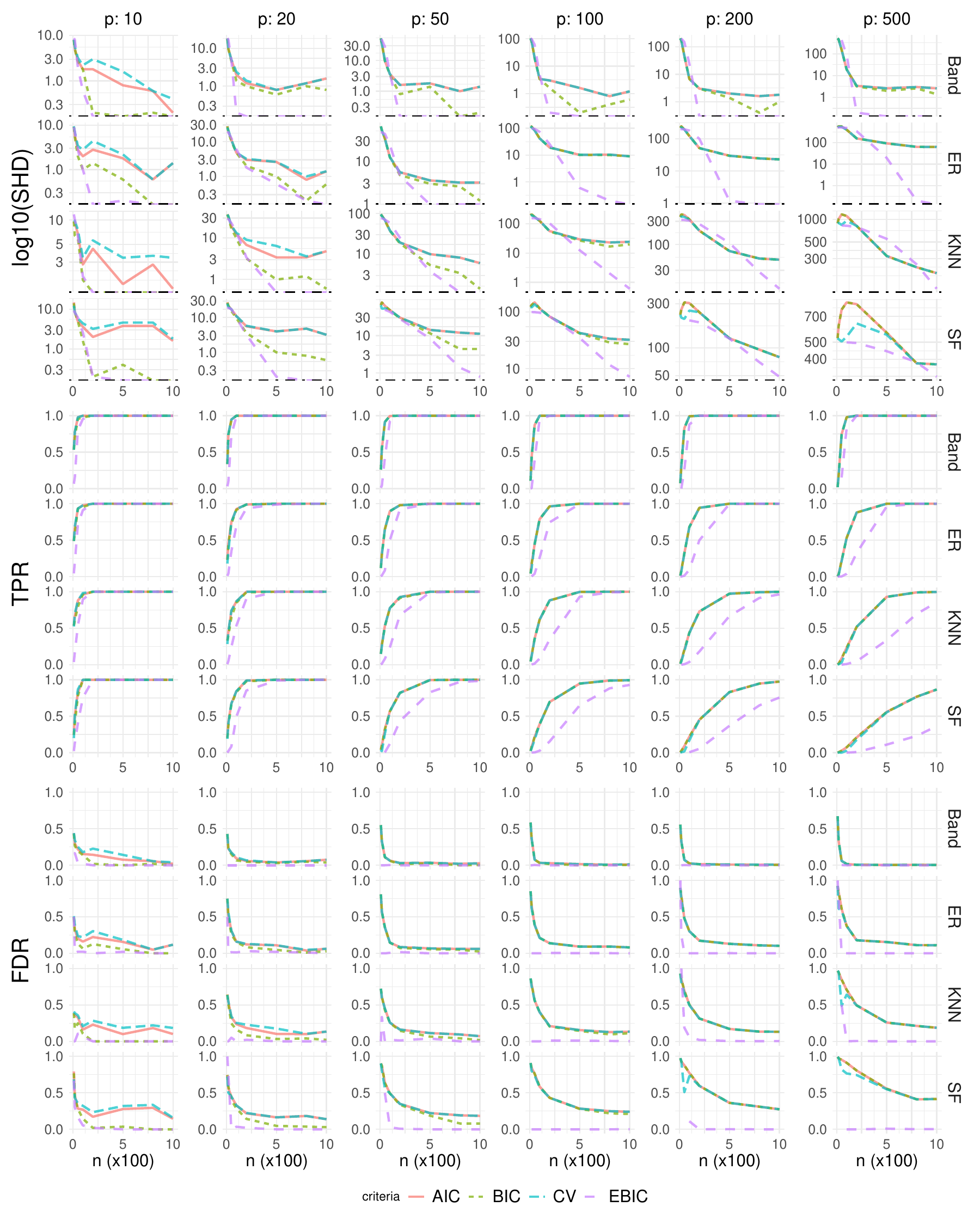}
  \caption{Log scale of SHD plots ($y$-axis values are not transformed into log-scale for direct comparison), TPR and FDR over varying sample size $n$ (in hundreds) and the number of variables $p$ on groups of graph types using Neighbourhood Selection (NS) to compare criteria. The dotdash line represents the $0$-SHD, i.e. perfect neighborhood selection. Wall time limit was set to three hours. Exceeding time limit or undefined FDR value for all zero estimates are marked as missing points for plotting.}
  \label{fig: fig all in one NS}
\end{figure*}
\begin{figure*}[htbp]
  \centering \includegraphics[width=\textwidth]{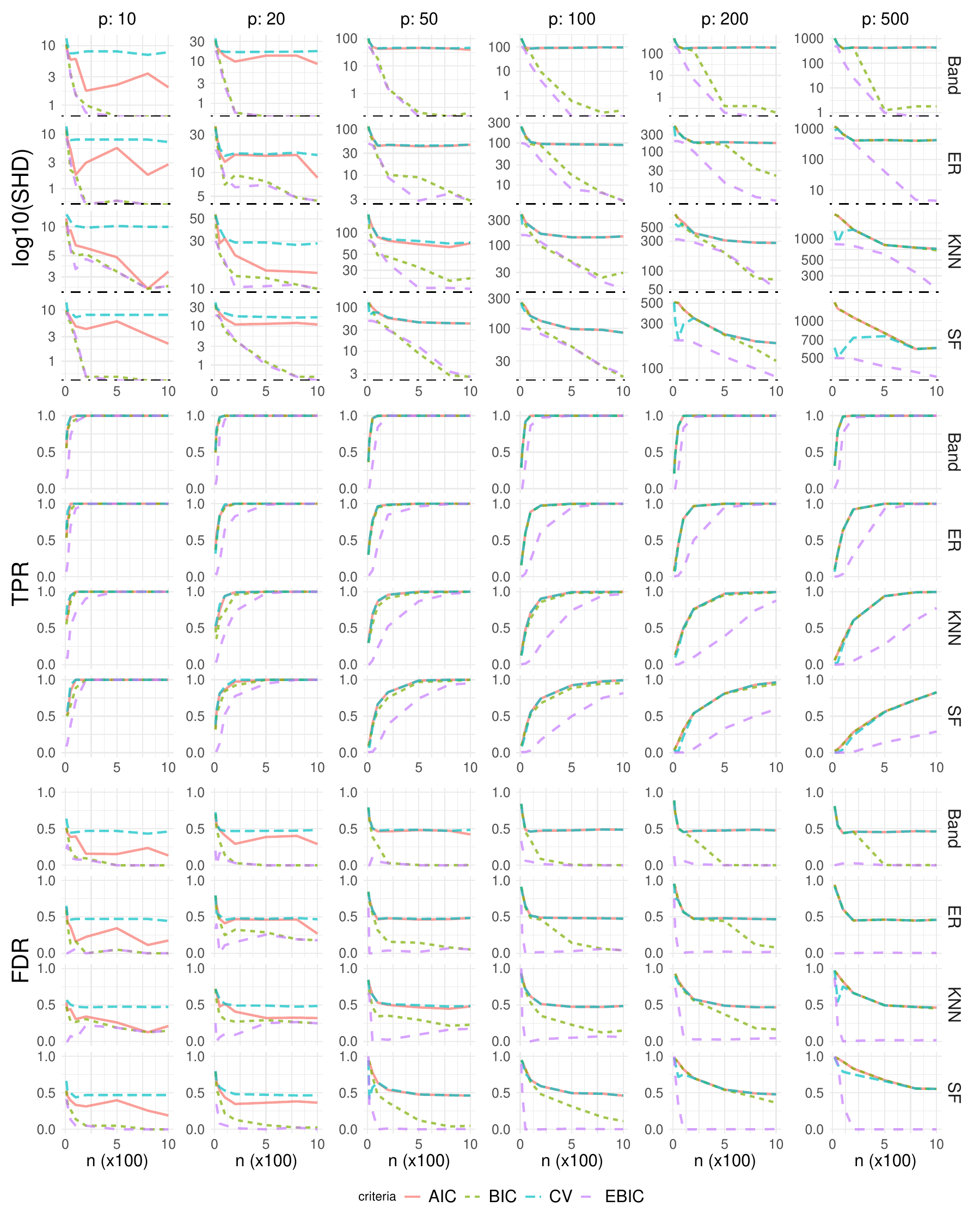}
  \caption{Log scale of SHD plots ($y$-axis values are not transformed into log-scale for direct comparison), TPR and FDR over varying sample size $n$ (in hundreds) and the number of variables $p$ on groups of graph types using Glasso to compare criteria. The dotdash line represents the $0$-SHD, i.e. perfect neighborhood selection. Wall time limit was set to three hours. Exceeding time limit or undefined FDR value for all zero estimates are marked as missing points for plotting.}
  \label{fig: fig all in one Glasso}
\end{figure*}

\begin{figure*}[htbp]
  \centering \includegraphics[width=\textwidth]{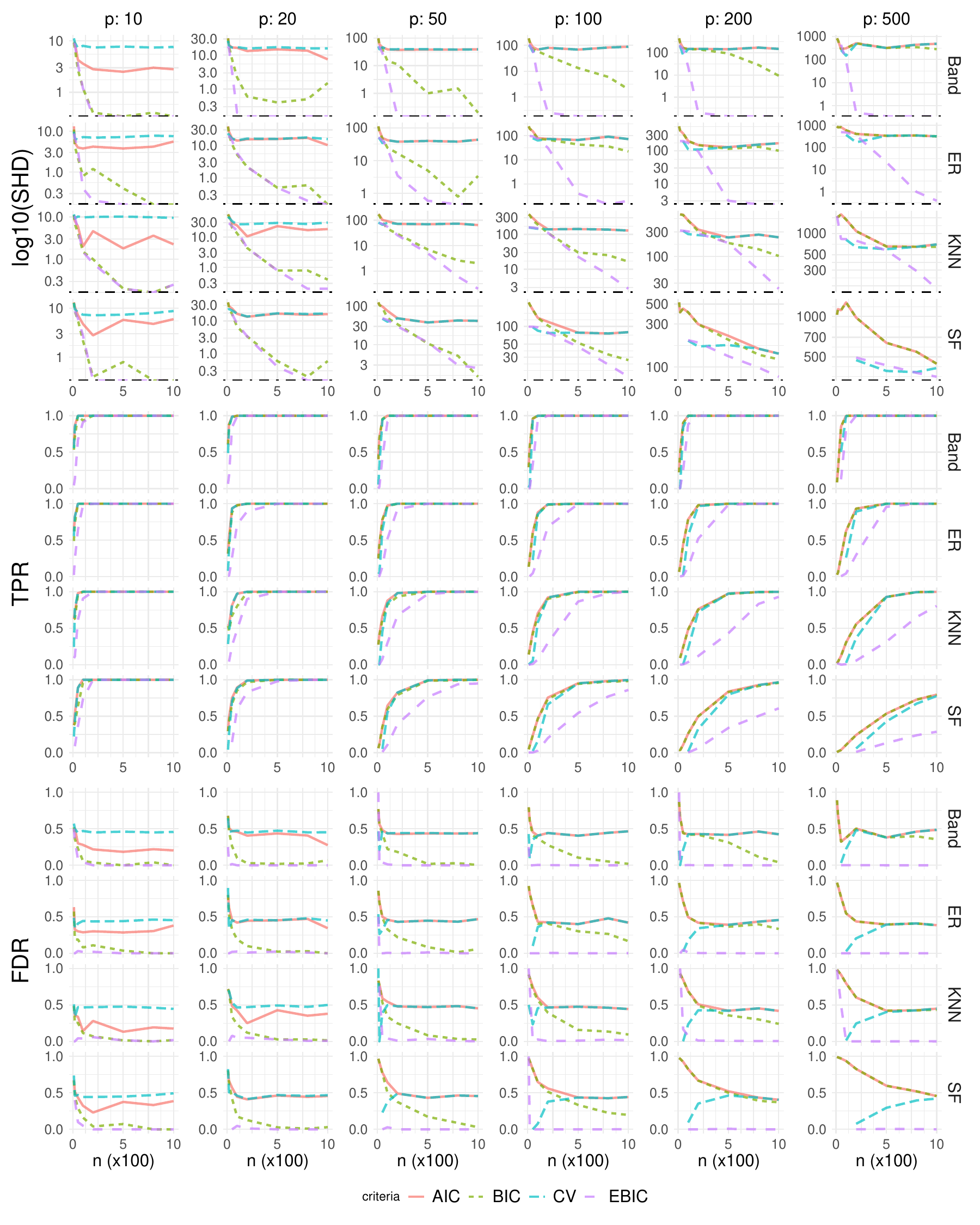}
  \caption{Log scale of SHD plots ($y$-axis values are not transformed into log-scale for direct comparison), TPR and FDR over varying sample size $n$ (in hundreds) and the number of variables $p$ on groups of graph types using CLIME to compare criteria. The dotdash line represents the $0$-SHD, i.e. perfect neighborhood selection. Wall time limit was set to three hours. Exceeding time limit or undefined FDR value for all zero estimates are marked as missing points for plotting.}
  \label{fig: fig all in one CLIME}
\end{figure*}

\begin{figure*}[htbp]
  \centering \includegraphics[width=\textwidth]{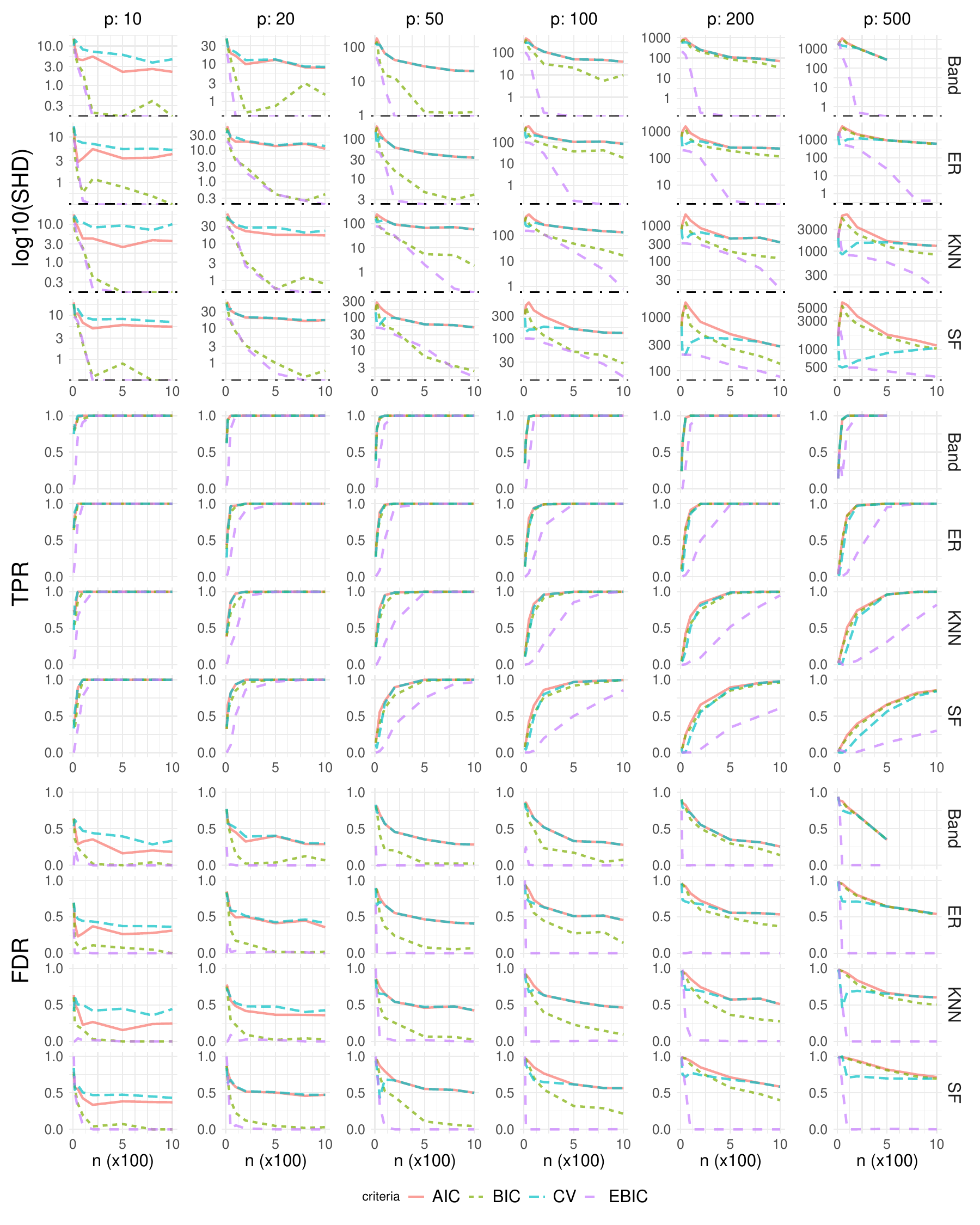}
  \caption{Log scale of SHD plots ($y$-axis values are not transformed into log-scale for direct comparison), TPR and FDR over varying sample size $n$ (in hundreds) and the number of variables $p$ on groups of graph types using TIGER to compare criteria. The dotdash line represents the $0$-SHD, i.e. perfect neighborhood selection. Wall time limit was set to three hours. Exceeding time limit or undefined FDR value for all zero estimates are marked as missing points for plotting.}
  \label{fig: fig all in one TIGER}
\end{figure*}

\begin{figure*}[htbp]
  \centering \includegraphics[width=\textwidth]{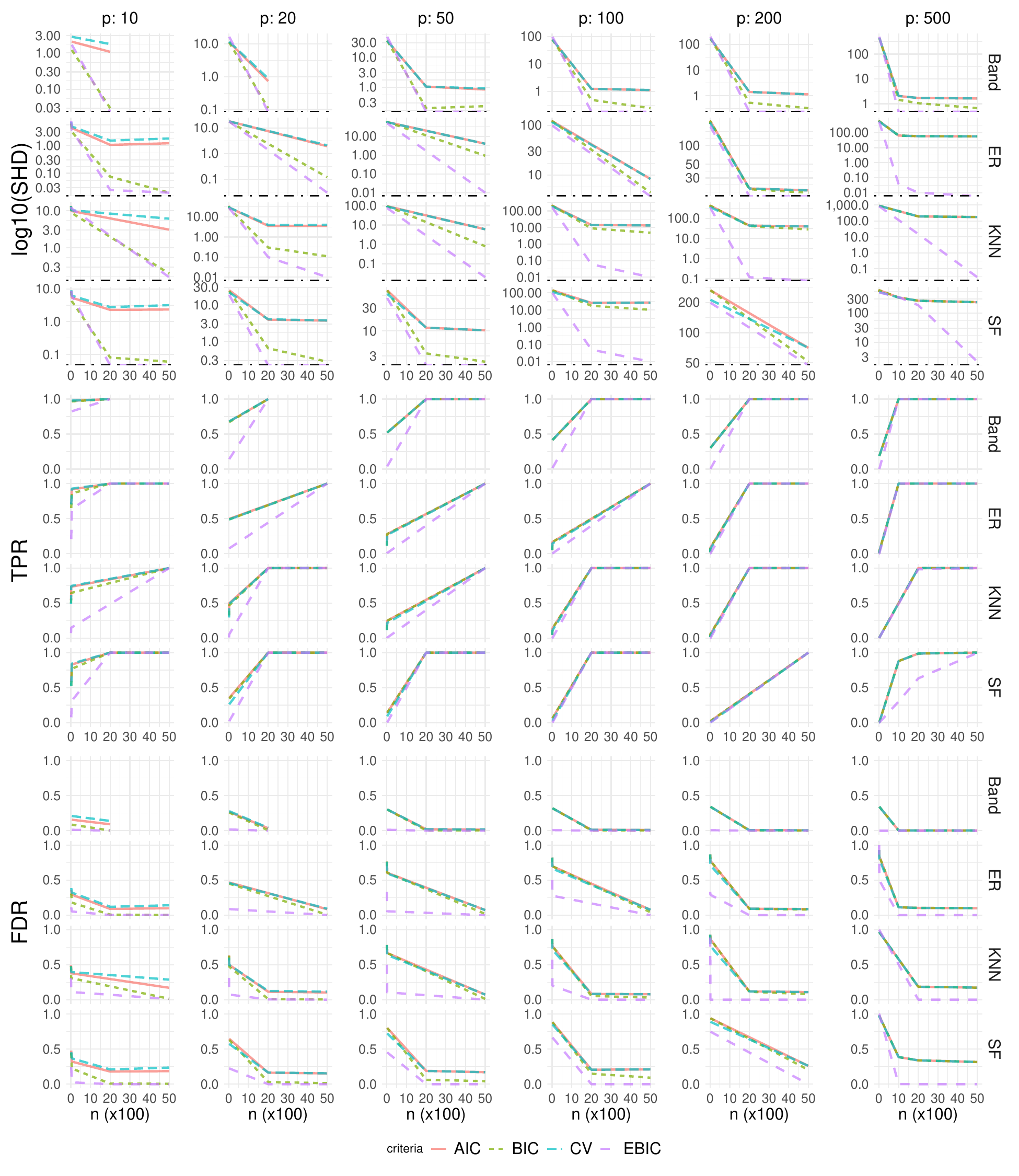}
  \caption{Log scale of SHD plots ($y$-axis values are not transformed into log-scale for direct comparison), TPR and FDR over varying sample size $n$ (in hundreds) and the number of variables $p$ on groups of graph types using Neighbourhood Selection (NS) to compare criteria on 100 runs. The dotdash line represents the $0$-SHD, i.e. perfect neighborhood selection. Wall time limit was set to three hours. Exceeding time limit or undefined FDR value for all zero estimates are marked as missing points for plotting.}
  \label{fig: fig all in one NS 100 runs}
\end{figure*}
\begin{figure*}[htbp]
  \centering \includegraphics[width=\textwidth]{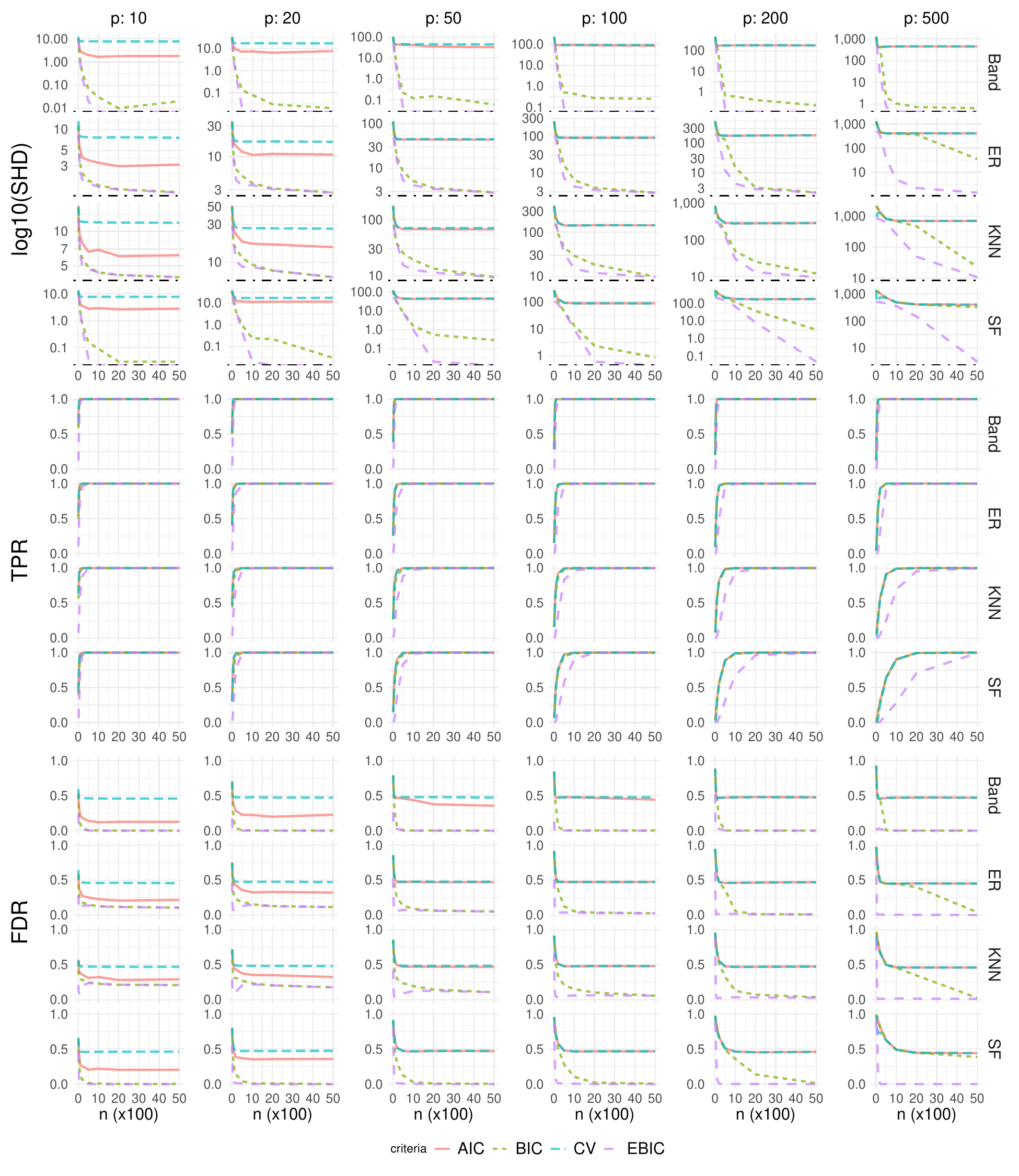}
  \caption{Log scale of SHD plots ($y$-axis values are not transformed into log-scale for direct comparison), TPR and FDR over varying sample size $n$ (in hundreds) and the number of variables $p$ on groups of graph types using Glasso to compare criteria on 100 runs. The dotdash line represents the $0$-SHD, i.e. perfect neighborhood selection. Wall time limit was set to three hours. Exceeding time limit or undefined FDR value for all zero estimates are marked as missing points for plotting.}
  \label{fig: fig all in one Glasso 100 runs}
\end{figure*}

\begin{figure*}[htbp]
  \centering \includegraphics[width=\textwidth]{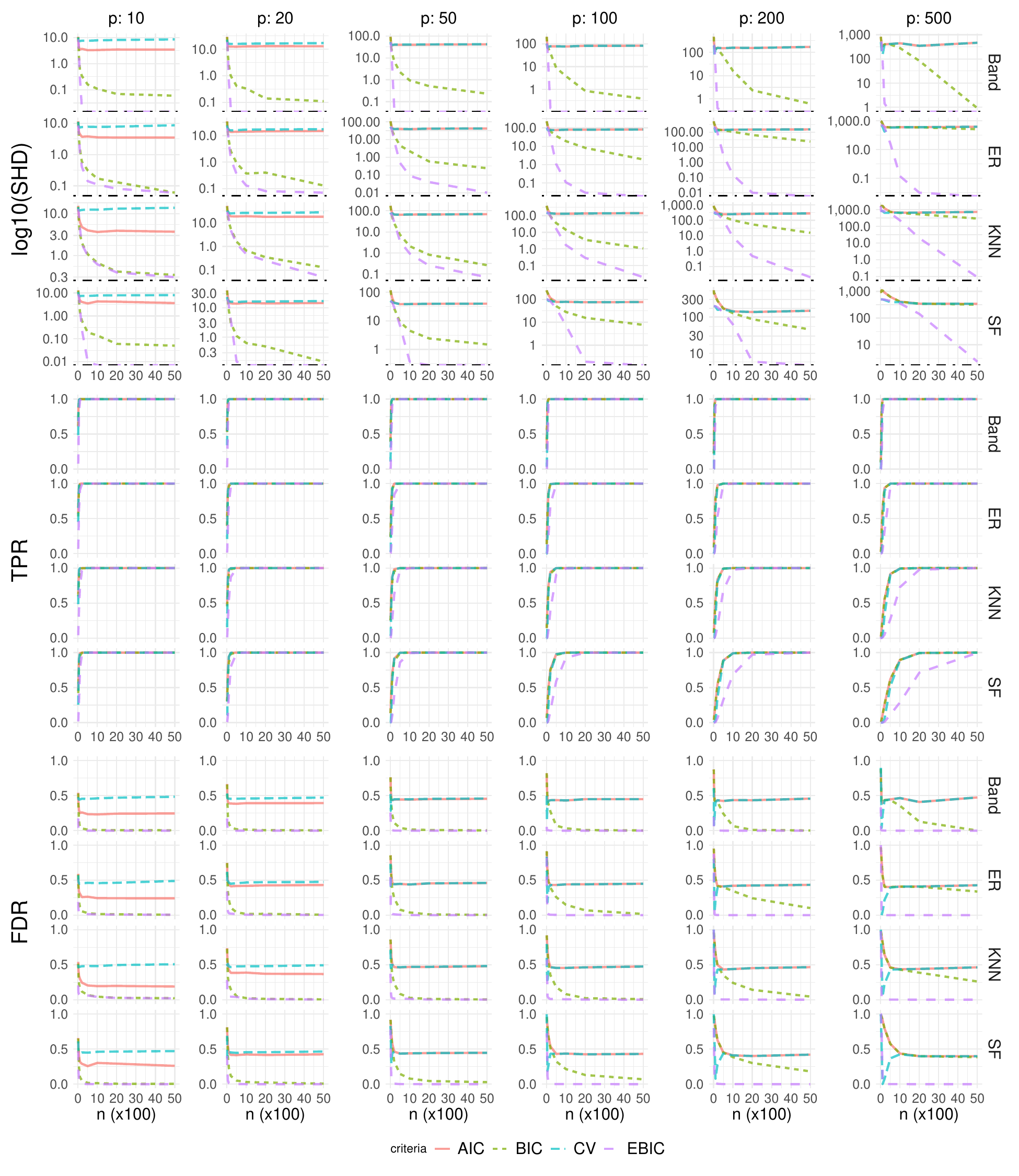}
  \caption{Log scale of SHD plots ($y$-axis values are not transformed into log-scale for direct comparison), TPR and FDR over varying sample size $n$ (in hundreds) and the number of variables $p$ on groups of graph types using CLIME to compare criteria on 100 runs. The dotdash line represents the $0$-SHD, i.e. perfect neighborhood selection. Wall time limit was set to three hours. Exceeding time limit or undefined FDR value for all zero estimates are marked as missing points for plotting.}
  \label{fig: fig all in one CLIME 100 runs}
\end{figure*}

\begin{figure*}[htbp]
  \centering \includegraphics[width=\textwidth]{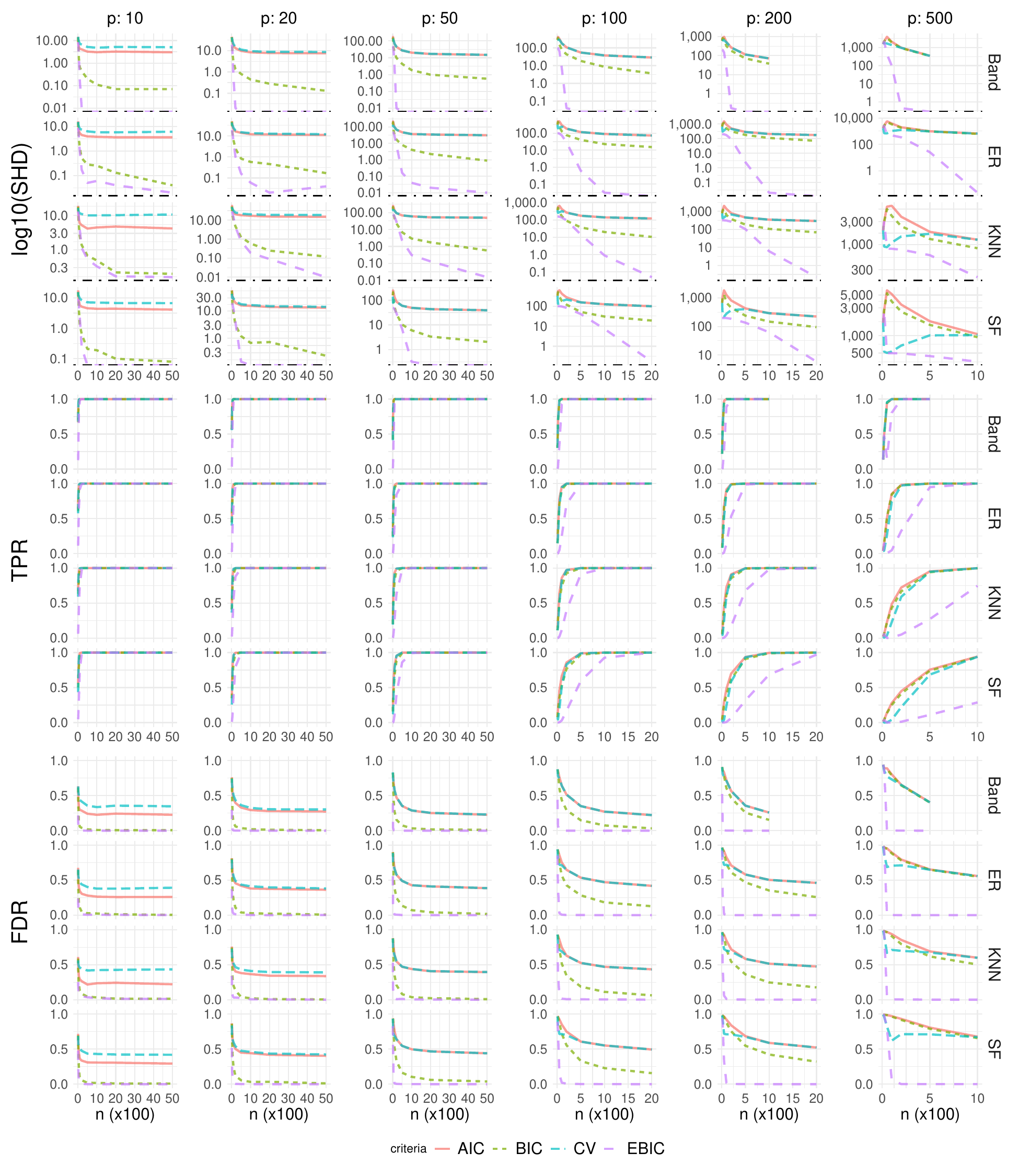}
  \caption{Log scale of SHD plots ($y$-axis values are not transformed into log-scale for direct comparison), TPR and FDR over varying sample size $n$ (in hundreds) and the number of variables $p$ on groups of graph types using TIGER to compare criteria on 100 runs. The dotdash line represents the $0$-SHD, i.e. perfect neighborhood selection. Wall time limit was set to three hours. Exceeding time limit or undefined FDR value for all zero estimates are marked as missing points for plotting.}
  \label{fig: fig all in one TIGER 100 runs}
\end{figure*}

\end{document}